\documentclass[reqno,10pt]{amsart}

\usepackage{amsmath,amssymb,amsthm,mathrsfs,graphicx,enumitem,verbatim}
\usepackage{amsaddr}

\usepackage[pdftex=true,
	bookmarks=true,
	bookmarksopen=true,
	bookmarksnumbered=true,
	pdftoolbar=true,
	pdfmenubar=true,
	pdffitwindow=false,
	pdfstartview={FitH},
	colorlinks=true,
	linkcolor=red,
	citecolor=blue,
	breaklinks=true]{hyperref}

\allowdisplaybreaks

\numberwithin{equation}{section}
\newtheorem{thm}{Theorem}[section]
\newtheorem{prop}[thm]{Proposition}
\newtheorem{lemma}[thm]{Lemma}
\newtheorem{cor}[thm]{Corollary}
\newtheorem*{thm*}{Theorem}
\newtheorem*{prop*}{Proposition}
\newtheorem*{cor*}{Corollary}
\newtheorem*{conj*}{Conjecture}

\theoremstyle{definition}
\newtheorem{definition}[thm]{Definition}

\theoremstyle{remark}
\newtheorem{rmk}[thm]{Remark}

\newcommand{\mc}{\mathcal}

\newcommand{\cC}{\mc C}

\newcommand{\cH}{\mc H}

\newcommand{\cL}{\mc L}
\newcommand{\cM}{\mc M}

\newcommand{\cO}{\mc O}
\newcommand{\cP}{\mc P}
\newcommand{\cQ}{\mc Q}

\newcommand{\cU}{\mc U}
\newcommand{\cV}{\mc V}

\newcommand{\cX}{\mc X}

\newcommand{\cCH}{\cC\cH}

\newcommand{\ms}{\mathscr}

\newcommand{\sD}{\ms D}

\newcommand{\C}{\mathbb{C}}

\newcommand{\R}{\mathbb{R}}
\newcommand{\Z}{\mathbb{Z}}
\newcommand{\Sph}{\mathbb{S}}

\newcommand{\sfs}{\mathsf{s}}
\newcommand{\sfw}{\mathsf{w}}

\newcommand{\ran}{\operatorname{ran}}

\renewcommand{\Re}{\operatorname{Re}}
\renewcommand{\Im}{\operatorname{Im}}
\newcommand{\id}{\operatorname{id}}

\newcommand{\supp}{\operatorname{supp}}
\newcommand{\sgn}{\operatorname{sgn}}

\newcommand{\la}{\langle}
\newcommand{\ra}{\rangle}
\newcommand{\pa}{\partial}
\newcommand{\tn}{\textnormal}

\newcommand{\eps}{\epsilon}

\newcommand{\wt}{\widetilde}
\newcommand{\wh}{\widehat}
\newcommand{\ol}{\overline}

\newcommand{\tot}{\mathrm{tot}}

\newcommand{\bop}{{\mathrm{b}}}

\newcommand{\cp}{{\mathrm{c}}}

\newcommand{\cl}{{\mathrm{cl}}}

\newcommand{\semi}{\hbar}

\newcommand{\psdo}{ps.d.o.}
\newcommand{\Diff}{\mathrm{Diff}}

\newcommand{\Vf}{\mathcal V}

\newcommand{\Vb}{\Vf_\bop}
\newcommand{\Diffb}{\Diff_\bop}

\newcommand{\Psib}{\Psi_\bop}

\newcommand{\Psih}{\Psi_\semi}

\newcommand{\Diffh}{\Diff_\semi}

\newcommand{\WF}{\mathrm{WF}}
\newcommand{\Ell}{\mathrm{Ell}}

\newcommand{\rcT}{\overline{T}}

\newcommand{\WFb}{\WF_{\bop}}

\newcommand{\Tb}{{}^{\bop}T}
\newcommand{\rcTb}{{}^{\bop}\overline{T}}

\newcommand{\Sb}{{}^{\bop}S}
\newcommand{\Nb}{{}^{\bop}N}

\newcommand{\WFh}{\mathrm{WF}_{\semi}}
\newcommand{\Ellh}{\mathrm{Ell}_{\semi}}

\newcommand{\ham}{H}
\newcommand{\rham}{{\mathsf{H}}}

\newcommand{\numin}{\nu_{\mathrm{min}}}

\newcommand{\bhm}{M_\bullet}

\newcommand{\loc}{{\mathrm{loc}}}
\newcommand{\CI}{\cC^\infty}

\newcommand{\CIc}{\cC^\infty_\cp}

\newcommand{\Hloc}{H_{\loc}}

\newcommand{\Hb}{H_{\bop}}

\newcommand{\dS}{{\mathrm{dS}}}

\newcommand{\Rnhalfc}{{\overline{\R^n_+}}}
\newcommand{\fw}{{\mathrm{fw}}}
\newcommand{\bw}{{\mathrm{bw}}}
\newcommand{\Hbfw}{H_{\bop,\fw}}
\newcommand{\Hbbw}{H_{\bop,\bw}}
\newcommand{\Hfw}{H_\fw}
\newcommand{\Hbw}{H_\bw}

\newcommand{\openbigpmatrix}[1]{\addtolength{\arraycolsep}{-#1}\begin{pmatrix}}
\newcommand{\closebigpmatrix}[1]{\end{pmatrix}\addtolength{\arraycolsep}{#1}}

\newcommand{\itref}[1]{(\ref{#1})}

\newlength{\enummargin}
\begin{document}

\title[Linear waves near the Cauchy horizon]{Analysis of linear waves near the Cauchy horizon of cosmological black holes}

\author{Peter Hintz}
\address{Department of Mathematics, University of California, Berkeley, CA 94720-3840, USA}
\email{phintz@berkeley.edu}

\author{Andr{\'a}s Vasy}
\address{Department of Mathematics, Stanford University, CA 94305-2125, USA}
\email{andras@math.stanford.edu}


\date{December 25, 2015. Final revision: July 12, 2017.}

\begin{abstract}
  We show that linear scalar waves are bounded and continuous up to the Cauchy horizon of Reissner--Nordstr\"om--de Sitter and Kerr--de Sitter spacetimes, and in fact decay exponentially fast to a constant along the Cauchy horizon. We obtain our results by modifying the spacetime beyond the Cauchy horizon in a suitable manner, which puts the wave equation into a framework in which a number of standard as well as more recent microlocal regularity and scattering theory results apply. In particular, the conormal regularity of waves at the Cauchy horizon --- which yields the boundedness statement --- is a consequence of radial point estimates, which are microlocal manifestations of the blue-shift and red-shift effects.
\end{abstract}

\maketitle

\section{Introduction}
\label{SecIntro}

We present a detailed analysis of the regularity and decay properties of linear scalar waves near the Cauchy horizon of cosmological black hole spacetimes. Concretely, we study charged and non-rotating (Reissner--Nordstr\"om--de Sitter) as well as uncharged and rotating (Kerr--de Sitter) black hole spacetimes for which the cosmological constant $\Lambda$ is positive. See Figure~\ref{FigIntroPenrose} for their Penrose diagrams. These spacetimes, in the region of interest for us, have the topology $\R_{t_0}\times I_r\times\Sph^2_\omega$, where $I\subset\R$ is an interval, and are equipped with a Lorentzian metric $g$ of signature $(1,3)$. The spacetimes have three horizons located at different values of the radial coordinate $r$, namely the \emph{Cauchy horizon} at $r=r_1$, the \emph{event horizon} at $r=r_2$ and the \emph{cosmological horizon} at $r=r_3$, with $r_1<r_2<r_3$. In order to measure decay, we use a time function $t_0$, which is equivalent to the Boyer--Lindquist coordinate $t$ away from the cosmological, event and Cauchy horizons, i.e.\ $t_0$ differs from $t$ by a smooth function of the radial coordinate $r$; and $t_0$ is equivalent to the Eddington--Finkelstein coordinate $u$ near the Cauchy and cosmological horizons, and to the Eddington--Finkelstein coordinate $v$ near the event horizon. We consider the Cauchy problem for the linear wave equation with Cauchy data posed on a surface $H_I$ as indicated in Figure~\ref{FigIntroPenrose}.

\begin{figure}[!ht]
  \centering
  \includegraphics{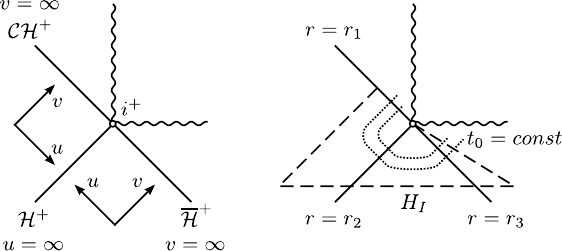}
  \caption{\textit{Left:} Penrose diagram of the Reissner--Nordstr\"om--de Sitter spacetime, and of an $\omega=const$ slice of the Kerr--de Sitter spacetime with angular momentum $a\neq 0$. Indicated are the Cauchy horizon $\cCH^+$, the event horizon $\cH^+$ and the cosmological horizon $\ol\cH^+$, as well as future timelike infinity $i^+$. The coordinates $u,v$ are Eddington--Finkelstein coordinates. \textit{Right:} The same Penrose diagram. The region enclosed by the dashed lines is the domain of dependence of the Cauchy surface $H_I$. The dotted lines are two level sets of the function $t_0$; the smaller one of these corresponds to a larger value of $t_0$.}
\label{FigIntroPenrose}
\end{figure}

The study of asymptotics and decay for linear scalar (and non-scalar) wave equations in a neighborhood of the exterior region $r_2<r<r_3$ of such spacetimes has a long history. Methods of scattering theory have proven very useful in this context, see \cite{SaBarretoZworskiResonances,BonyHaefnerDecay,DyatlovQNM,DyatlovQNMExtended,WunschZworskiNormHypResolvent,VasyMicroKerrdS,MelroseSaBarretoVasySdS,HintzVasyKdsFormResonances} and references therein (we point out that near the black hole exterior, Reissner--Nordstr\"om--de Sitter space can be studied using exactly the same methods as Schwarzschild--de Sitter space); see \cite{DafermosRodnianskiSdS} for a different approach using vector field commutators. There is also a substantial amount of literature on the case $\Lambda=0$ of the asymptotically flat Reissner--Nordstr\"om and Kerr spacetimes; we refer the reader to \cite{KayWaldSchwarzschild,BachelotSchwarzschildScattering,DafermosEinsteinMaxwellScalarStability,DafermosRodnianskiPrice,DafermosRodnianskiLectureNotes,MarzuolaMetcalfeTataruTohaneanuStrichartz,DonningerSchlagSofferPrice,TohaneanuKerrStrichartz,TataruDecayAsympFlat,SterbenzTataruMaxwellSchwarzschild,AnderssonBlueMaxwellKerr,DafermosRodnianskiShlapentokhRothmanDecay} and references therein.

The purpose of the present work is to show how a uniform analysis of linear waves up to the Cauchy horizon can be accomplished using methods from scattering theory and microlocal analysis. Our main result is:

\begin{thm}
\label{ThmIntroMain}
  Let $g$ be a non-degenerate Reissner--Nordstr\"om--de Sitter metric with non-zero charge $Q$, or a non-degenerate Kerr--de Sitter metric with small non-zero angular momentum $a$, with spacetime dimension $\geq 4$. Then there exists $\alpha>0$, only depending on the parameters of the spacetime, such that the following holds: if $u$ is the solution of the Cauchy problem $\Box_g u=0$ with smooth initial data, then there exists $C>0$ such that $u$ has a partial asymptotic expansion
  \begin{equation}
  \label{EqIntroMainExp}
    u=u_0+u',
  \end{equation}
  where $u_0\in\C$, and
  \[
    |u'| \leq C e^{-\alpha t_0}
  \]
  uniformly in $r>r_1$. The same bound, with a different constant $C$, holds for derivatives of $u'$ along any finite number of stationary vector fields which are tangent to the Cauchy horizon. Moreover, $u$ is continuous up to the Cauchy horizon.
  
  More precisely, $u'$ as well as all such derivatives of $u'$ lie in the weighted spacetime Sobolev space $e^{-\alpha t_0}H^{1/2+\alpha/\kappa_1-0}$ in $t_0>0$, where $\kappa_1$ is the surface gravity of the Cauchy horizon.

  For the massive Klein--Gordon equation $(\Box_g-m^2)u=0$, $m>0$ small, the same result holds true without the constant term $u_0$.
\end{thm}

Here, the spacetime Sobolev space $H^s$, for $s\in\Z_{\geq 0}$, consists of functions which remain in $L^2$ under the application of up to $s$ stationary vector fields; for general $s\in\R$, $H^s$ is defined using duality and interpolation. The final part of Theorem~\ref{ThmIntroMain} in particular implies that $u'$ lies in $H^{1/2+\alpha/\kappa_1-0}$ near the Cauchy horizon on any surface of fixed $t_0$. After introducing the Reissner--Nordstr\"om--de Sitter and Kerr--de Sitter metrics at the beginning of \S\S\ref{SecRNdS} and \ref{SecKdS}, we will prove Theorem~\ref{ThmIntroMain} in \S\S\ref{SubsecRNdSConormal} and \ref{SubsecKdSRes}, see Theorems~\ref{ThmRNdSPartialAsympConormal} and \ref{ThmKdSPartialAsympConormal}. Our analysis carries over directly to non-scalar wave equations as well, as we discuss for differential forms in \S\ref{SubsecRNdSBundles}; however, we do not obtain uniform boundedness near the Cauchy horizon in this case. Furthermore, a substantial number of ideas in the present paper can be adapted to the study of asymptotically flat ($\Lambda=0$) spacetimes; corresponding boundedness, regularity and (polynomial) decay results on Reissner--Nordstr\"om and Kerr spacetimes are discussed in \cite{HintzCauchyKerr}.

Let us also mention that a minor extension of our arguments yields analogous boundedness, decay and regularity results for the Cauchy problem with a `two-ended' Cauchy surface $H_I$ up to the bifurcation sphere $B$, see Figure~\ref{FigIntroBifurcation}.

\begin{figure}[!ht]
  \centering
  \includegraphics{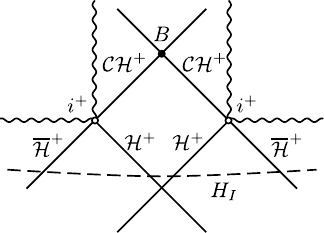}
  \caption{A piece of the maximal analytic extension of the Reissner--Nordstr\"om--de Sitter spacetime, with two exterior regions, each bounded to the future by an event and a cosmological horizon. The two parts of the Cauchy horizon intersect in the bifurcation sphere $B\cong\Sph^2$. For solutions of the Cauchy problem with initial data posed on $H_I$, our methods imply boundedness and precise regularity results, as well as asymptotics and decay towards $i^+$, in the causal past of $B$.}
\label{FigIntroBifurcation}
\end{figure}

Theorem~\ref{ThmIntroMain} is the first result known to the authors establishing asymptotics and regularity near the Cauchy horizon of rotating black holes. (However, we point out that Dafermos and Luk have recently announced the $C^0$ stability of the Cauchy horizon of the Kerr spacetime for Einstein's vacuum equations \cite{DafermosICM2014}.) In the case of $\Lambda=0$ and in spherical symmetry (Reissner--Nordstr\"om), Franzen \cite{FranzenRNBoundedness} proved the uniform boundedness of waves in the black hole interior and $\cC^0$ regularity up to $\cCH^+$, while Luk and Oh \cite{LukOhReissnerNordstrom} showed that linear waves generically do not lie in $\Hloc^1$ at $\cCH^+$. (We also mention the more recent \cite{LukSbierskiKerr}.) There is also ongoing work by Franzen on the analogue of her result for Kerr spacetimes \cite{FranzenKerrBoundedness}. Gajic \cite{GajicExtremalRN}, based on previous work by Aretakis \cite{AretakisExtremalRN1,AretakisExtremalRN2}, showed that for \emph{extremal} Reissner--Nordstr\"om spacetimes, waves \emph{do} lie in $\Hloc^1$. We do not present a microlocal study of the event horizon of extremal black holes here, however we remark that our analysis reveals certain high regularity phenomena at the Cauchy horizon of \emph{near-extremal} black holes, which we will discuss below. Closely related to this, the study of Costa, Gir\~ao, Nat\'ario and Silva \cite{CostaGiraoNatarioSilvaCauchy1,CostaGiraoNatarioSilvaCauchy2,CostaGiraoNatarioSilvaCauchy3} of the nonlinear Einstein--Maxwell--scalar field system in spherical symmetry shows that, close to extremality, rather weak assumptions on initial data on a null hypersurface transversal to the event horizon guarantee $\Hloc^1$ regularity of the metric at $\cCH^+$; however, they assume exact Reissner--Nordstr\"om--de Sitter data on the event horizon, while in the present work, we link non-trivial decay rates of waves along the event horizon to the regularity of waves at $\cCH^+$. Compare this also with the discussions in \S\ref{SubsecRNdSHighReg} and Remark~\ref{RmkRNdSHighReg}.

One could combine the treatment of Reissner--Nordstr\"om--de Sitter and Kerr--de Sitter spacetimes by studying the more general Kerr--Newman--de Sitter family of charged and rotating black hole spacetimes, discovered by Carter \cite{CarterHamiltonJacobiEinstein}, which can be analyzed in a way that is entirely analogous to the Kerr--de Sitter case. However, in order to prevent cumbersome algebraic manipulations from obstructing the flow of our analysis, we give all details for Reissner--Nordstr\"om--de Sitter black holes, where the algebra is straightforward and where moreover mode stability can easily be shown to hold for subextremal spacetimes; we then indicate rather briefly the (mostly algebraic) changes for Kerr--de Sitter black holes, and leave the similar, general case of Kerr--Newman--de Sitter black holes to the reader. In fact, our analysis is stable under suitable perturbations, and one can thus obtain results entirely analogous to Theorem~\ref{ThmIntroMain} for Kerr--Newman--de Sitter metrics with small non-zero angular momentum $a$ and small charge $Q$ (depending on $a$), or for small charge $Q$ and small non-zero angular momentum $a$ (depending on $Q$), by perturbative arguments: indeed, in these two cases, the Kerr--Newman--de Sitter metric is a small stationary perturbation of the Kerr--de Sitter, resp.\ Reissner--Nordstr\"om--de Sitter metric, with the same structure at $\cCH^+$.

In the statement of Theorem~\ref{ThmIntroMain}, we point out that the amount of regularity of the remainder term $u'$ at the Cauchy horizon is directly linked to the amount $\alpha$ of exponential decay of $u'$: the more decay, the higher the regularity. This can intuitively be understood in terms of the \emph{blue-shift effect} \cite{SimpsonPenroseBlueShift}: the more a priori decay $u'$ has along the Cauchy horizon (approaching $i^+$), the less energy can accumulate at the horizon. The precise microlocal statement capturing this is a \emph{radial point estimate} at the intersection of $\cCH^+$ with the boundary at infinity of a compactification of the spacetime at $t_0=\infty$, which we will discuss in \S\ref{SubsecIntroGeometry}.

Now, $\alpha$ can be any real number less than the \emph{spectral gap} $\alpha_0$ of the operator $\Box_g$, which is the infimum of $-\Im\sigma$ over all non-zero \emph{resonances} (or quasi-normal modes) $\sigma\in\C$; the resonance at $\sigma=0$ gives rise to the constant $u_0$ term. (We refer to \cite{SaBarretoZworskiResonances,BonyHaefnerDecay,DyatlovQNM} and \cite{VasyMicroKerrdS} for the discussion of resonances for black hole spacetimes.) Due to the presence of a trapped region in the black hole spacetimes considered here, $\alpha_0$ is bounded from above by a quantity $\gamma_0>0$ associated with the null-geodesic dynamics near the trapped set, as proved by Dyatlov \cite{DyatlovResonanceProjectors,DyatlovSpectralGaps} in the present context following breakthrough work by Wunsch and Zworski \cite{WunschZworskiNormHypResolvent}, and by Nonnenmacher and Zworski \cite{NonnenmacherZworskiDecay}: below (resp.\ above) any line $\Im\sigma=-\gamma_0+\eps$, $\eps>0$, there are infinitely (resp.\ finitely) many resonances. In principle however, one expects that there indeed exists a non-zero number of resonances above this line, and correspondingly the expansion \eqref{EqIntroMainExp} can be refined to take these into account. (In fact, one can obtain a full resonance expansion due to the complete integrability of the null-geodesic flow near the trapped set, see \cite{BonyHaefnerDecay,DyatlovAsymptoticDistribution}.) Since for the mode solution corresponding to a resonance at $\sigma$, $\Im\sigma<0$, we obtain the regularity $H^{1/2-\Im\sigma/\kappa_1-0}$ at $\cCH^+$, shallow resonances, i.e.\ those with small $\Im\sigma$, give the dominant contribution to the solution $u$ both in terms of decay and regularity at $\cCH^+$. The authors are not aware of any rigorous results on shallow resonances, so we shall only discuss this briefly in Remark~\ref{RmkRNdSHighReg}, taking into account insights from numerical results: these suggest the existence of resonant states with imaginary parts roughly equal to $-\kappa_2$ and $-\kappa_3$, and hence the relative sizes of the surface gravities play a crucial role in determining the regularity at $\cCH^+$.

Whether resonant states are in fact no better than $H^{1/2-\Im\sigma/\kappa_1}$, and the existence of shallow resonances, which, if true, would yield a linear instability result for cosmological black hole spacetimes with Cauchy horizons analogous to \cite{LukOhReissnerNordstrom}, will be studied in future work. Once these questions have been addressed, one can conclude that the lack of, say, $H^1$ regularity at $\cCH^+$ is caused precisely by shallow quasinormal modes. Thus, somewhat surprisingly, the mechanism for the linear instability of the Cauchy horizon of \emph{cosmological} spacetimes is more subtle than for asymptotically flat spacetimes in that the presence of a cosmological horizon, which ultimately allows for a resonance expansion of linear waves $u$, leads to a much more precise structure of $u$ at $\cCH^+$, with the regularity of $u$ directly tied to quasinormal modes of the black hole exterior.

The interest in understanding the behavior of waves near the Cauchy horizon has its roots in Penrose's Strong Cosmic Censorship conjecture, which asserts that maximally globally hyperbolic developments for the Einstein--Maxwell or Einstein vacuum equations (depending on whether one considers charged or uncharged solutions) with \emph{generic} initial data (and a complete initial surface, and/or under further conditions) are inextendible as suitably regular Lorentzian manifolds. In particular, the smooth, even analytic, extendability of the Reissner--Nordstr\"om(--de Sitter) and Kerr(--de Sitter) solutions past their Cauchy horizons is conjectured to be an unstable phenomenon. It turns out that the question what should be meant by `suitable regularity' is very subtle; we refer to works by Christodoulou \cite{ChristodoulouInstabililtyOfNakedSing}, Dafermos \cite{DafermosInterior,DafermosBlackHoleNoSingularities}, and Costa, Gir\~ao, Nat\'ario and Silva \cite{CostaGiraoNatarioSilvaCauchy1,CostaGiraoNatarioSilvaCauchy2,CostaGiraoNatarioSilvaCauchy3} in the spherically symmetric setting for positive and negative results for various notions of regularity. There is also work in progress by Dafermos and Luk on the $C^0$ stability of the Kerr Cauchy horizon, assuming a quantitative version of the non-linear stability of the exterior region. We refer to these works, as well as to the excellent introductions of \cite{DafermosEinsteinMaxwellScalarStability,DafermosBlackHoleNoSingularities,LukOhReissnerNordstrom}, for a discussion of heuristic arguments and numerical experiments which sparked this line of investigation.

Here, however, we only consider linear equations, motivated by similar studies in the asymptotically flat case by Dafermos \cite{DafermosEinsteinMaxwellScalarStability} (see \cite[Footnote~11]{LukOhReissnerNordstrom}), Franzen \cite{FranzenRNBoundedness}, Sbierski \cite{SbierskiThesis}, and Luk and Oh \cite{LukOhReissnerNordstrom}. The main insight of the present paper is that a uniform analysis up to $\cCH^+$ can be achieved using by now standard methods of scattering theory and geometric microlocal analysis, in the spirit of recent works by Vasy \cite{VasyMicroKerrdS}, Baskin, Vasy and Wunsch \cite{BaskinVasyWunschRadMink} and \cite{HintzVasySemilinear}: the core of the precise estimates of Theorem~\ref{ThmIntroMain} are microlocal propagation results at (generalized) radial sets, as we will discuss in \S\ref{SubsecIntroStrategy}. From this geometric microlocal perspective however, i.e.\ taking into account merely the phase space properties of the operator $\Box_g$, it is both unnatural and technically inconvenient to view the Cauchy horizon as a boundary; after all, the metric $g$ is a non-degenerate Lorentzian metric up to $\cCH^+$ and beyond. Thus, the most subtle step in our analysis is the formulation of a suitable extended problem (in a neighborhood of $r_1\leq r\leq r_3$) which reduces to the equation of interest, namely the wave equation, in $r>r_1$.

\subsection{Geometric setup}
\label{SubsecIntroGeometry}

The Penrose diagram is rather singular at future timelike infinity $i^+$, yet all relevant phenomena, in particular trapping and red-/blue-shift effects, should be thought of as taking place there, as we will see shortly; therefore, we work instead with a compactification of the region of interest, the domain of dependence of $H_I$ in Figure~\ref{FigIntroPenrose}, in which the horizons as well as the trapped region remain separated, and the metric remains smooth, as $t_0\to\infty$. Concretely, using the coordinate $t_0$ employed in Theorem~\ref{ThmIntroMain}, the radial variable $r$ and the spherical variable $\omega\in\Sph^2$, we consider a region
\[
  M = [0,\infty)_\tau \times X,\quad X=(0,r_3+2\delta)_r \times \Sph^2_\omega,\quad \tau=e^{-t_0},
\]
i.e.\ we add the ideal boundary at future infinity, $\tau=0$, to the spacetime, and equip $M$ with the natural smooth structure in which $\tau$ vanishes simply and non-degenerately at $\pa M$. (It is tempting, and useful for purposes of intuition, to think of $M$ as being a submanifold of the blow-up of the compactification suggested by the Penrose diagram --- adding an `ideal sphere at infinity' at $i^+$ --- at $i^+$. However, the details are somewhat subtle; see \cite{MelroseSaBarretoVasySdS}.)

Due to the stationary nature of the metric $g$, the (null-)geodesic flow should be studied in a version of phase space which has a built-in uniformity as $t_0\to\infty$. A clean way of describing this uses the language of b-geometry (and b-analysis); we refer the reader to Melrose \cite{MelroseAPS} for a detailed introduction, and \cite[\S3]{VasyMicroKerrdS} and \cite[\S2]{HintzVasySemilinear} for brief overviews. We recall the most important features here: on $M$, the metric $g$ is a non-degenerate Lorentzian \emph{b-metric}, i.e.\ a linear combination with smooth (on $M$) coefficients of
\[
  \frac{d\tau^2}{\tau^2},\quad \frac{d\tau}{\tau}\otimes dx_i+dx_i\otimes\frac{d\tau}{\tau},\quad dx_i\otimes dx_j+dx_j\otimes dx_i,
\]
where $(x_1,x_2,x_3)$ are coordinates in $X$; in fact the coefficients are independent of $\tau$. Then, $g$ is a section of the symmetric second tensor power of a natural vector bundle on $M$, the \emph{b-cotangent bundle} $\Tb^*M$, which is spanned by the sections $\frac{d\tau}{\tau},dx_i$. We stress that $\frac{d\tau}{\tau}=-dt_0$ is a smooth, non-degenerate section of $\Tb^*M$ \emph{up to and including} the boundary $\tau=0$. Likewise, the dual metric $G$ is a section of the second symmetric tensor power of the \emph{b-tangent bundle} $\Tb M$, which is the dual bundle of $\Tb^*M$ and thus spanned by $\tau\pa_\tau,\pa_{x_i}$. The dual metric function, which we also denote by $G\in\CI(\Tb^*M)$ by a slight abuse of notation, associates to $\zeta\in\Tb^*M$ the squared length $G(\zeta,\zeta)$.

Over the interior of $M$, denoted $M^\circ$, the b-cotangent bundle is naturally isomorphic to the standard cotangent bundle. The geodesic flow, lifted to the cotangent bundle, is generated by the Hamilton vector field $\ham_G\in\cV(T^*M^\circ)$, which extends to a smooth vector field $\ham_G\in\cV(\Tb^*M)$ tangent to $\Tb^*_X M$. Now, $\ham_G$ is homogeneous of degree $1$ with respect to dilations in the fiber, and it is often convenient to rescale it by multiplication with a homogeneous degree $-1$ function $\wh\rho$, obtaining the homogeneous degree $0$ vector field $\rham_G=\wh\rho\ham_G$. As such, it extends smoothly to a vector field on the \emph{radial (or projective) compactification} $\rcTb^*M$ of $\Tb^*M$, which is a ball bundle over $M$, with fiber over $z\in M$ given by the union of $\Tb_z^*M$ with the `sphere at fiber infinity' $\wh\rho=0$. The b-cosphere bundle $\Sb^*M=(\Tb^*M\setminus o)/\R_+$ is then conveniently viewed as the boundary $\Sb^*M=\pa\rcTb^*M$ of the compactified b-cotangent bundle at fiber infinity.

The projection to the base $M$ of integral curves of $\ham_G$ or $\rham_G$ with null initial direction, i.e.\ starting at a point in $\Sigma=G^{-1}(0)\setminus o\subset\Tb^*M$, yields (reparameterizations of) null-geodesics on $(M,g)$; this is clear in the interior of $M$, and the important observation is that this gives a well-defined notion of null-geodesics, or null-bicharacteristics, at the boundary at infinity, $X$. We remark that the characteristic set $\Sigma$ has two components, the union of the future null cones $\Sigma_-$ and of the past null cones $\Sigma_+$.

The red-shift or blue-shift effect manifests itself in a special structure of the $\rham_G$ flow near the b-conormal bundles $L_j=\Nb^*\{\tau=0,\ r=r_j\}\setminus o\subset\Sigma$ of the horizons $r=r_j$, $j=1,2,3$. (Here, $\Nb_x^*Z$ for a boundary submanifold $Z\subset X$ and $x\in Z$ is the annihilator of the space of all vectors in $\Tb_x Z$ tangent to $Z$; $\Nb^*Z$ is naturally isomorphic to the conormal bundle of $Z$ in $X$.) Indeed, in the case of the Reissner--Nordstr\"om--de Sitter metric, $L_j$, more precisely its boundary at fiber infinity $\pa L_j\subset\Sb^*M\subset\rcTb^*M$, is a \emph{saddle point} for the $\rham_G$ flow, with stable (or unstable, depending on which of the two components $L_{j,\pm}:=L_j\cap\Sigma_\pm$ one is working on) manifold contained in $\Sigma\cap\rcTb^*_X M$, and an unstable (or stable) manifold transversal to $\rcTb^*_X M$. In the Kerr--de Sitter case, $\rham_G$ does not vanish everywhere on $\pa L_j$, but rather is non-zero and tangent to it, so there are non-trivial dynamics within $\pa L_j$, but the dynamics in the directions normal to $\pa L_j$ still has the same saddle point structure. See Figure~\ref{FigIntroRadial}.

\begin{figure}[!ht]
  \centering
  \includegraphics{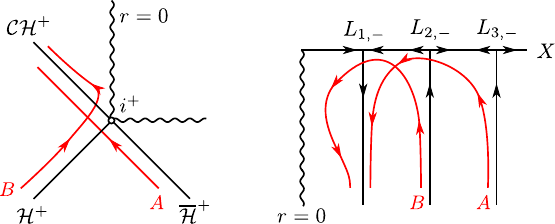}
  \caption{\textit{Left:} Two future-directed null-geodesics; $A$ is a radial null-geodesic, and $B$ is the projection of a non-radial geodesic. \textit{Right:} The compactification of the spacetime at future infinity, together with the same two null-geodesics. The null-geodesic flow, extended to the (b-cotangent bundle over the) boundary, has saddle points at the (b-conormal bundles of the) intersection of the horizons with the boundary at infinity $X$.}
\label{FigIntroRadial}
\end{figure}

\subsection{Strategy of the proof}
\label{SubsecIntroStrategy}

In order to take full advantage of the saddle point structure of the null-geodesic flow near the Cauchy horizon, one would like to set up an initial value problem, or equivalently a forced forward problem $\Box_g u=f$, with vanishing initial data but non-trivial right hand side $f$, on a domain which extends a bit past $\cCH^+$. Because of the finite speed of propagation for the wave equation, one is free to modify the problem beyond $\cCH^+$ in whichever way is technically most convenient; waves in the region of interest $r_1<r<r_3$ are unaffected by the choice of extension.

A natural idea then is to simply add a boundary $\wt H_T:=\{r=r_-\}$, $r_-\in(0,r_1)$, which one could use to cap the problem off beyond $\cCH^+$; now $\wt H_T$ is \emph{timelike}, hence, to obtain a well-posed problem, one needs to impose boundary conditions there. While perfectly feasible, the resulting analysis is technically rather involved as it necessitates studying the reflection of singularities at $\wt H_T$ quantitatively in a uniform manner as $\tau\to 0$. (Near $\wt H_T$, one does not need the precise, microlocal, control as in \cite{TaylorGrazing,MelroseSjostrandSingBVPI,MelroseSjostrandSingBVPII} however.)

A technically much easier modification involves the use of a complex absorbing `potential' $\cQ$ in the spirit of \cite{NonnenmacherZworskiDecay,WunschZworskiNormHypResolvent}; here $\cQ$ is a second order b-pseudodifferential operator on $M$ which is elliptic in a large subset of $r<r_1$ near $\tau=0$. (Without b-language, one can take $\cQ$ for large $t_0$ to be a time translation-invariant, properly supported \psdo{}\ on $M$.) One then considers the operator
\[
  \cP=\Box_g-i\cQ.
\]
The point is that a suitable choice of the sign of $\cQ$ on the two components $\Sigma_\pm$ of the characteristic set leads to an absorption of high frequencies along the future-directed null-geodesic flow over the support of $\cQ$, which allows one to control a solution $u$ of $\cP u=f$ in terms of the right hand side $f$ there. However, since we are forced to work on a domain with boundary in order to study the forward problem, the \emph{pseudodifferential} complex absorption does not make sense near the relevant boundary component, which is the extension of the left boundary in Figure~\ref{FigIntroPenrose} past $r=r_1$.

A doubling construction as in \cite[\S4]{VasyMicroKerrdS} on the other hand, doubling the spacetime across the timelike surface $\wt H_T$, say, amounts to gluing an `artificial exterior region' to our spacetime, with one of the horizons identified with the original Cauchy horizon; this in particular creates another trapped region, which we can however easily hide using a complex absorbing potential! We then cap off the thus extended spacetime beyond the cosmological horizon of the artificial exterior region, located at $r=r_0$, by a \emph{spacelike} hypersurface $H_{I,0}$ at $r=r_0-\delta$, $\delta>0$, at which the analysis is straightforward \cite[\S2]{HintzVasySemilinear}. See Figure~\ref{FigIntroExtended}. (In the spherically symmetric setting, one could also replace the region $r<r_1$ beyond the Cauchy horizon by a static de Sitter type space, thus not generating any further trapping or horizons and obviating the need for complex absorption; but for Kerr--de Sitter, this gluing procedure is less straightforward to implement, hence we use the above doubling-type procedure for Reissner--Nordstr\"om already.) The construction of the extension is detailed in \S\ref{SubsecRNdSMfd}.

\begin{figure}[!ht]
  \centering
  \includegraphics{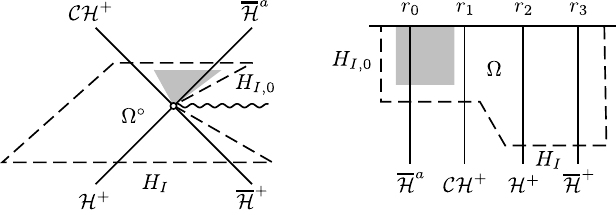}
  \caption{The extended spacetime: we glue an artificial exterior region beyond $\cCH^+$, creating an artificial horizon $\ol\cH^a$, and cap off beyond $\ol\cH^a$ using a \emph{spacelike} hypersurface $H_{I,0}$. Complicated dynamics in the extended region are hidden by a complex absorbing potential $\cQ$ supported in the shaded region.}
\label{FigIntroExtended}
\end{figure}

We thus study the forcing problem
\begin{equation}
\label{EqIntroForcing}
  \cP u=f\quad\tn{in }\Omega,
\end{equation}
with $f$ and $u$ supported in the future of the `Cauchy' surface $H_I$ in $r\geq r_1$, and in the future of $H_{I,0}$ in $r\leq r_1$. The natural function spaces are \emph{weighted b-Sobolev spaces}
\[
  \Hb^{s,\alpha}(M):=\tau^\alpha\Hb^s(M)=e^{-\alpha t_0}H^s(M),
\]
where the spacetime Sobolev space $H^s(M)$ measures regularity relative to $L^2$ with respect to stationary vector fields, as defined after the statement of Theorem~\ref{ThmIntroMain}. More invariantly, $\Hb^s(M)$, for integer $s$, consists of $L^2$ functions which remain in $L^2$ upon applying up to $s$ b-vector fields; the space $\Vb(M)$ of b-vector fields consists of all smooth vector fields on $M$ which are tangent to $\pa M$, and is equal to the space of smooth sections of the b-tangent bundle $\Tb M$.

Now, \eqref{EqIntroForcing} is an equation on a compact space $\Omega$ which degenerates at the boundary: the operator $\Box_g\in\Diffb^2(M)$ is a b-differential operator, i.e.\ a sum of products of b-vector fields, and $\cQ\in\Psib^2(M)$ is a b-\psdo{}. (Note that this point of view is much more precise than merely stating that \eqref{EqIntroForcing} is an equation on a noncompact space $\Omega^\circ$!) Thus, the analysis of the operator $\cP$ consists of two parts: \textit{firstly,} the regularity analysis, in which one obtains precise regularity estimates for $u$ using microlocal elliptic regularity, propagation of singularities and radial point results, see \S\ref{SubsecRNdSRegularity}, which relies on the precise global structure of the null-geodesic flow discussed in \S\ref{SubsecRNdSFlow}; and \textit{secondly,} the asymptotic analysis of \S\S\ref{SubsecRNdSFredholm} and \ref{SubsecRNdSAsymp}, which relies on the analysis of the Mellin transformed in $\tau$ (equivalently: Fourier transformed in $-t_0$) operator family $\wh\cP(\sigma)$, its high energy estimates as $|\Re\sigma|\to\infty$, and the structure of poles of $\wh\cP(\sigma)^{-1}$, which are known as \emph{resonances} or \emph{quasi-normal modes}; this last part, in which we use the shallow resonances to deduce asymptotic expansions of waves, is the only low frequency part of the analysis.

The regularity one obtains for $u$ solving \eqref{EqIntroForcing} with, say, smooth compactly supported (in $\Omega^\circ$) forcing $f$, is determined by the behavior of the null-geodesic flow near the trapping and near the horizons $L_j$, $j=1,2,3$. Near the trapping, we use the aforementioned results \cite{WunschZworskiNormHypResolvent,NonnenmacherZworskiDecay,DyatlovSpectralGaps,HintzVasyNormHyp}, while near $L_j$, we use radial point estimates, originating in work by Melrose \cite{MelroseEuclideanSpectralTheory}, and proved in the context relevant for us in \cite[\S2]{HintzVasySemilinear}; we recall these in \S\ref{SubsecRNdSRegularity}, and refer the reader to \cite{ZworskiRevisitVasy} and \cite{VasyMinicourse} for further details. Concretely, equation \eqref{EqIntroForcing} combines a forward problem for the wave equation near the black hole exterior region $r\in(r_2,r_3)$ with a backward problem near the artificial exterior region $r\in(r_0,r_1)$, with hyperbolic propagation in the region between these two (called `no-shift region' in \cite{FranzenRNBoundedness}). Near $r=r_2$ and $r=r_3$ then, and by propagation estimates in any region $r\geq r_1+\eps$, $\eps>0$, the radial point estimate, encapsulating the red-shift effect, yields smoothness of $u$ relative to a b-Sobolev space with weight $\alpha<0$, i.e.\ allowing for exponential growth (in which case trapping is not an issue), while near $r=r_1$, one is solving the equation \emph{away} from the boundary $X$ at infinity, and hence the radial point estimate, encapsulating the blue-shift effect, there, yields an amount of regularity which is bounded from above by $1/2+\alpha/\kappa_1-0$, where $\kappa_1$ is the surface gravity of $\cCH^+$. In the extended region $r<r_1$, the regularity analysis is very simple, since the complex absorption $\cQ$ makes the problem elliptic at the trapping there and at $\ol\cH^a$, and one then only needs to use real principal type propagation together with standard energy estimates.

Combined with the analysis of $\wh\cP(\sigma)$, which relies on the same dynamical and geometric properties of the extended problem as the b-analysis, we deduce in \S\ref{SubsecRNdSFredholm} that $\cP$ is Fredholm on suitable weighted b-Sobolev spaces (and in fact solvable for any right hand side $f$ if one modifies $f$ in the unphysical region $r<r_1$). In order to capture the high, resp.\ low, regularity near $[r_2,r_3]$, resp.\ $r_1$, these spaces have \emph{variable orders} of differentiability depending on the location in $M$. (Such spaces were used already by Unterberger \cite{UnterbergerVariable}, and in a context closely related to the present paper in \cite{BaskinVasyWunschRadMink}. We present results adapted to our needs in Appendix~\ref{SecVariable}.)

In \S\ref{SubsecRNdSAsymp} then, we show how the properties of the meromorphic family $\wh\cP(\sigma)^{-1}$ yield a partial asymptotic expansion of $u$ as in \eqref{EqIntroMainExp}. Using more refined regularity statements at $L_1$, we show in \S\ref{SubsecRNdSConormal} that the terms in this expansion are in fact \emph{conormal} to $r=r_1$, i.e.\ they do not become more singular upon applying vector fields tangent to the Cauchy horizon.

We stress that the analysis is conceptually very simple, and close to the analysis in \cite{VasyMicroKerrdS,BaskinVasyWunschRadMink,HintzVasySemilinear,GellRedmanHaberVasyFeynman}, in that it relies on tools in microlocal analysis and scattering theory which have been frequently used in recent years.

As a side note, we point out that one could have analyzed $\wh\Box_g(\sigma)$ in $r>r_1$ only by proving very precise estimates for the operator $\wh\Box_g(\sigma)$, which is a hyperbolic (wave-type) operator in $r>r_1$, near $r=r_1$; while this would have removed the necessity to construct and analyze an extended problem, the mechanism underlying our regularity and decay estimates, namely the radial point estimate at the Cauchy horizon, would not have been apparent from this. Moreover, the radial point estimate is very robust; it works for Kerr--de Sitter spaces just as it does for the spherically symmetric Reissner--Nordstr\"om--de Sitter solutions.

A more interesting modification of our argument relies on the observation that it is not necessary for us to incorporate the exterior region in our global analysis, since this has already been studied in detail before; instead, one could start \emph{assuming} asymptotics for a wave $u$ in the exterior region, and then relate $u$ to a solution of a global, extended problem, for which one has good regularity results, and deduce them for $u$ by restriction. Such a strategy, implemented in \cite{HintzCauchyKerr}, is in particular appealing in the study of spacetimes with vanishing cosmological constant using the analytic framework of the present paper, since the precise structure of the `resolvent' $\wh\Box_g(\sigma)^{-1}$ has not been analyzed so far, whereas boundedness and decay for scalar waves on the exterior regions of Reissner--Nordstr\"om and Kerr spacetimes are known by other methods; see the references at the beginning of \S\ref{SecIntro}.

In the remaining parts of \S\ref{SecRNdS}, we analyze the essential spectral gap for near-extremal black holes in \S\ref{SubsecRNdSHighReg}; we find that for \emph{any} desired level of regularity, one can choose near-extremal parameters of the black hole such that solutions $u$ to \eqref{EqIntroForcing} with $f$ in a finite-codimensional space achieve this level of regularity at $\cCH^+$. However, as explained in the discussion of Theorem~\ref{ThmIntroMain}, it is very likely that shallow resonances cause the codimension to increase as the desired regularity increases. Lastly, in \S\ref{SubsecRNdSBundles}, we indicate the simple changes to our analysis needed to accommodate wave equations on natural tensor bundles.

In \S\ref{SecKdS} then, we show how Kerr--de Sitter spacetimes fit directly into our framework: we analyze the flow on a suitable compactification and extension, constructed in \S\ref{SubsecKdSMfd}, in \S\ref{SubsecKdSFlow}, and deduce results completely analogous to the Reissner--Nordstr\"om--de Sitter case in \S\ref{SubsecKdSRes}.

\section{Reissner--Nordstr\"om--de Sitter space}
\label{SecRNdS}

We focus on the case of $4$ spacetime dimensions; the analysis in more than $4$ dimensions is completely analogous. In the domain of outer communications of the 4-dimensional Reissner--Nordstr\"om--de Sitter black hole, given by $\R_t\times(r_2,r_3)_r\times\Sph^2_\omega$, with $0<r_2<r_3$ described below, the metric takes the form
\begin{equation}
\label{EqRNdSMetric}
  g = \mu\,dt^2 - \mu^{-1}\,dr^2 - r^2\,d\omega^2,\quad \mu = 1-\frac{2\bhm}{r}+\frac{Q^2}{r^2} - \lambda r^2,
\end{equation}
Here $\bhm>0$ and $Q>0$ are the mass and the charge of the black hole, and $\lambda=\Lambda/3$, with $\Lambda>0$ the cosmological constant. Setting $Q=0$, this reduces to the Schwarzschild--de Sitter metric. We assume that the spacetime is non-degenerate:

\begin{definition}
\label{DefRNdSNonDegenerate}
  We say that the Reissner--Nordstr\"om--de Sitter spacetime with parameters $\Lambda>0,\bhm>0,Q>0$ is \emph{non-degenerate} if $\mu$ has $3$ simple positive roots $0<r_1<r_2<r_3$.
\end{definition}

Since $\mu\to-\infty$ when $r\to\infty$, we see that
\[
  \mu>0\tn{ in }(0,r_1)\cup(r_2,r_3), \quad \mu<0\tn{ in }(r_1,r_2)\cup(r_3,\infty).
\]
The roots of $\mu$ are called \emph{Cauchy horizon} ($r_1$), \emph{event horizon} ($r_2$) and \emph{cosmological horizon} ($r_3$), with the Cauchy horizon being a feature of charged (or rotating, see \S\ref{SecKdS}) solutions of Einstein's field equations.

To give a concrete example of a non-degenerate spacetime, let us check the non-degeneracy condition for black holes with small charge, and compute the location of the Cauchy horizon: for fixed $\bhm,\lambda>0$, let
\[
  \Delta_r(r,q)=r^2-2\bhm r+q-\lambda r^4,
\]
so $\Delta_r(r,Q^2)=r^2\mu$. For $q=0$, the function $\Delta_r(r,0)$ has a root at $r_{1,0}:=0$. Since $\wt\Delta_r(r):=r^{-1}\Delta_r(r,0)=r-2\bhm-\lambda r^3$ is negative for $r=0$ and for large $r>0$ but positive for large $r<0$, the function $\wt\Delta_r(r)$ has two simple positive roots if and only if $\wt\Delta_r(r_c)>0$, where $r_c=(3\lambda)^{-1/2}$ is the unique positive critical point of $\wt\Delta_r'(r)$; but $\wt\Delta_r(r_c)=2((27\lambda)^{-1/2}-\bhm)>0$ if and only if
\begin{equation}
\label{EqRNdSSdSNonDegenerate}
  9\Lambda \bhm^2<1.
\end{equation}
Then:
\begin{lemma}
\label{LemmaRNdSNonDegenerate}
  Suppose $\Lambda,\bhm>0$ satisfy the non-degeneracy condition \eqref{EqRNdSSdSNonDegenerate}, and denote the three non-negative roots of $\Delta_r(r,0)$ by $r_{1,0}=0<r_{2,0}<r_{3,0}$. Then for small $Q>0$, the function $\mu$ has three positive roots $r_j(Q)$, $j=1,2,3$, with $r_j(0)=r_{j,0}$, depending smoothly on $Q$, and $r_1(Q)=\frac{Q^2}{2\bhm}+\cO(Q^4)$.
\end{lemma}
\begin{proof}
  The existence of the functions $r_j(Q)$ follows from the implicit function theorem, taking into account the simplicity of the roots $r_{j,0}$ of $\Delta_r(r,0)$. Let us write $\wt r_j(q)=r_j(\sqrt{q})$; these are smooth functions of $q$. Differentiating $0=\Delta_r(\wt r_1(q),q)$ with respect to $q$ gives $0=-2\bhm\wt r_1'(0)+1$, hence $\wt r_1(q)=\frac{q}{2\bhm}+\cO(q^2)$, which yields the analogous expansion for $r_1(Q)$.
\end{proof}

\subsection{Construction of the compactified spacetime}
\label{SubsecRNdSMfd}

We now discuss the extension of the metric \eqref{EqRNdSMetric} beyond the event and cosmological horizon, as well as beyond the Cauchy horizon; the purpose of the present section is to define the manifold on which our analysis of linear waves will take place. See Proposition~\ref{PropRNdSMfd} for the final result. We begin by describing the extension of the metric \eqref{EqRNdSMetric} beyond the event and the cosmological horizon, thereby repeating the arguments of \cite[\S6]{VasyMicroKerrdS}; see Figure~\ref{FigRNdSExt23}.

\begin{figure}[!ht]
  \centering
  \includegraphics{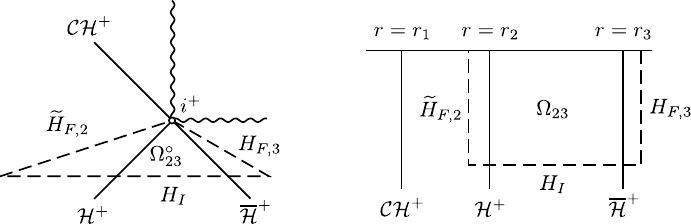}
  \caption{\textit{Left:} Part of the Penrose diagram of the maximally extended Reissner--Nordstr\"om--de Sitter solution, with the cosmological horizon $\ol\cH^+$, the event horizon $\cH^+$ and the Cauchy horizon $\cCH^+$. We first study a region $\Omega_{23}^\circ$ bounded by an initial Cauchy hypersurface $H_I$ and two final Cauchy hypersurfaces $\wt H_{F,2}$ and $H_{F,3}$. \textit{Right:} The same region, compactified at infinity ($i_+$ in the Penrose diagram), with the artificial hypersurfaces put in.}
\label{FigRNdSExt23}
\end{figure}

Write $s_j=-\sgn\mu'(r_j)$, so
\begin{equation}
\label{EqRNdSSigns}
  s_1=1, \quad s_2=-1, \quad s_3=1.
\end{equation}
We denote by $F_{23}(r)\in\CI((r_2,r_3))$ a smooth function such that
\begin{equation}
\label{EqRNdSF23}
  F_{23}'(r) = s_j(\mu^{-1}+c_j)\quad\tn{ for }r\in(r_2,r_3), \ |r-r_j|<2\delta,
\end{equation}
$\delta>0$ small, with $c_j$, smooth near $r=r_j$, to be specified momentarily. (Thus, $F_{23}(r)\to+\infty$ as $r\to r_2+$ and $r\to r_3-$.) We then put
\begin{equation}
\label{EqRNdST23}
  t_{23} := t - F_{23}
\end{equation}
and compute
\begin{equation}
\label{EqRNdSg23}
  g = \mu\,dt_{23}^2 + 2s_j(1+\mu c_j)\,dt_{23}\,dr + (2c_j+\mu c_j^2)\,dr^2 - r^2\,d\omega^2,
\end{equation}
which is a non-degenerate Lorentzian metric up to $r=r_2,r_3$, with dual metric
\[
  G = -(2c_j+\mu c_j^2)\pa_{t_{23}}^2 + 2s_j(1+\mu c_j)\pa_{t_{23}}\pa_r - \mu\pa_r^2 - r^{-2}\pa_\omega^2.
\]
We can choose $c_j$ so as to make $dt_{23}$ timelike, i.e.\ $\la dt_{23},dt_{23}\ra_G=-(2c_j+\mu c_j^2)>0$: indeed, choosing $c_j=-\mu^{-1}$ (which undoes the coordinate change \eqref{EqRNdST23}, up to an additive constant) accomplishes this trivially in $[r_2,r_3]$ away from $\mu=0$; however, we need $c_j$ to be smooth at $\mu=0$ as well. Now, $dt_{23}$ is timelike in $\mu>0$ if and only if $c_j(c_j+2\mu^{-1})<0$, which holds for any $c_j\in(-2\mu^{-1},0)$. Therefore, we can choose $c_2$ smooth near $r_2$, with $c_2=-\mu^{-1}$ for $r>r_2+\delta$, and $c_3$ smooth near $r_3$, with $c_3=-\mu^{-1}$ for $r<r_3-\delta$, and thus a function $F_{23}\in\CI((r_2,r_3))$, such that in the new coordinate system $(t_{23},r,\omega)$, the metric $g$ extends smoothly to $r=r_2,r_3$, and $dt_{23}$ is timelike for $r\in[r_2,r_3]$; and furthermore we can arrange that $t_{23}=t$ in $[r_2+\delta,r_3-\delta]$ by possibly changing $F_{23}$ by an additive constant.

Extending $c_j$ smoothly beyond $r_j$ in an arbitrary manner, the expression \eqref{EqRNdSg23} makes sense for $r\geq r_3$ as well as for $r\in(r_1,r_2]$. We first notice that we can choose the extension $c_j$ such that $dt_{23}$ is timelike also for $r\in(r_1,r_2)\cup(r_3,\infty)$: indeed, for such $r$, we have $\mu<0$, and the timelike condition becomes $c_j(c_j+2\mu^{-1})>0$, which is satisfied as long as $c_j\in(-\infty,0)$ there. In particular, we can take $c_2=\mu^{-1}$ for $r\in(r_1,r_2-\delta]$ and $c_3=\mu^{-1}$ for $r\geq r_3+\delta$, in which case we get
\begin{equation}
\label{EqRNdSg23NearR2}
  g = \mu\,dt_{23}^2 + 4s_j\,dt_{23}\,dr + 3\mu^{-1}\,dr^2 - r^2\,d\omega^2
\end{equation}
for $r\in(r_1,r_2-\delta]$ with $j=2$, and for $r\in[r_3+\delta,\infty)$ with $j=3$. We define $F_{23}'$ beyond $r_2$ and $r_3$ by the same formula \eqref{EqRNdSF23}, using the extensions of $c_2$ and $c_3$ just described; in particular $F_{23}'=-2\mu^{-1}$ in $r\leq r_2-\delta$. We define a time orientation in $r\geq r_2-2\delta$ by declaring $dt_{23}$ to be future timelike.

We introduce spacelike hypersurfaces in the thus extended spacetime as indicated in Figure~\ref{FigRNdSExt23}, namely
\begin{equation}
\label{EqRNdSHypersurfacesIF3}
\begin{gathered}
  H_I = \{ t_{23}=0, \ r_2-2\delta\leq r\leq r_3+2\delta \}, \\
  H_{F,3} = \{ t_{23}\geq 0, \ r=r_3+2\delta \},
\end{gathered}
\end{equation}
and
\begin{equation}
\label{EqRNdSHypersurfacesTildeF2}
  \wt H_{F,2} = \{ t_{23}\geq 0, \ r=r_2-2\delta \}.
\end{equation}

\begin{rmk}
\label{RmkRNdSHypersurfaceNotation}
  Here and below, the subscript `I' (initial), resp.\ `F' (final), indicates that outward pointing timelike vectors are past, resp.\ future, oriented. The number in the subscript denotes the horizon near which the surface is located.
\end{rmk}

Notice here that indeed $G(dr,dr)=-\mu>0$ and $G(dt_{23},dr)=2$ at $H_{F,3}$, so $dt_{23}$ and $dr$ have opposite timelike character there, while likewise $G(-dr,-dr)=-\mu>0$ and $G(dt_{23},-dr)=2$ at $\wt H_{F,2}$. The tilde indicates that $\wt H_{F,2}$ will eventually be disposed of; we only define it here to make the construction of the extended spacetime clearer. The region $\Omega_{23}^\circ$ is now defined as
\begin{equation}
\label{EqRNdSOmega23}
  \Omega_{23}^\circ = \{ t_{23}\geq 0, r_2-2\delta\leq r\leq r_3+2\delta \}.
\end{equation}

\begin{figure}[!ht]
  \centering
  \includegraphics{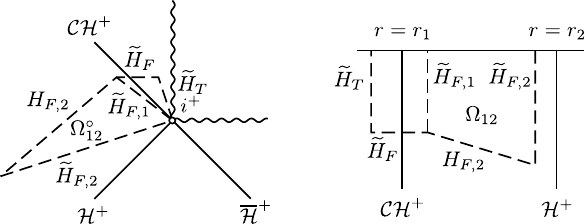}
  \caption{\textit{Left:} We describe a region $\Omega_{12}^\circ$ bounded by three final Cauchy hypersurfaces $H_{F,2}$, $\wt H_{F,1}$ and $\wt H_{F,2}$. A partial extension beyond the Cauchy horizon is bounded by the final hypersurface $\wt H_F$ and a timelike hypersurfaces $\wt H_T$. \textit{Right:} The same region, compactified at infinity, with the artificial hypersurfaces put in.}
\label{FigRNdSExt12}
\end{figure}

Next, we further extend the metric beyond the coordinate singularity of $g$ at $r=r_1$ when written in the coordinates \eqref{EqRNdSg23NearR2}, at $r=r_1$; see Figure~\ref{FigRNdSExt12}: let
\[
  t_{12} = t_{23} - F_{12},
\]
where now $F_{12}' = 3\mu^{-1}+c_1$, with $c_1=-3\mu^{-1}$ for $r\in[r_2-2\delta,r_2)$, and $c_1$ smooth down to $r=r_1$. Thus, by adjusting $F_{12}$ by an additive constant, we may arrange $t_{12}=t_{23}$ for $r\in[r_2-2\delta,r_2-\delta]$. Notice that (formally) $t_{12} = t - F_{23} - F_{12}$, and $F_{23}'+F_{12}'=s_1(\mu^{-1}+c_1)$ in $(r_1,r_2-\delta]$. Thus,
\begin{equation}
\label{EqRNdSg12}
  g = \mu\,dt_{12}^2 + 2(1+\mu c_1)\,dt_{12}\,dr + (2c_1+\mu c_1^2)\,dr^2 - r^2\,d\omega^2, \quad r\in[r_1-2\delta,r_2-\delta],
\end{equation}
after extending $c_1$ smoothly into $r\geq r_1-2\delta$. This expression is of the form \eqref{EqRNdSg23}, with $t_{23}$, $s_j$ and $c_j$ replaced by $t_{12}$, $s_1=1$ and $c_1$, respectively. In particular, by the same calculation as above, $dt_{12}$ is timelike provided $c_1<0$ or $c_1>-2\mu^{-1}$ in $\mu<0$, while in $\mu>0$, any $c_1\in(-2\mu^{-1},0)$ works. However, since we need $c_1=-3\mu^{-1}$ for $r$ near $r_2$ (where $\mu<0$), requiring $dt_{12}$ to be timelike would force $c_1>-2\mu^{-1}\to\infty$ as $r\to r_1+$, which is incompatible with $c_1$ being smooth down to $r=r_1$. In view of the Penrose diagram of the spacetime in Figure~\ref{FigRNdSExt12}, it is clear that this must happen, since we cannot make the level sets of $t_{12}$ (which coincide with the level sets of $t_{23}$, i.e.\ with parts of $H_I$, near $r=r_2$) both remain spacelike and cross the Cauchy horizon in the indicated manner. Thus, we merely require $c_1<0$ for $r\in[r_1-2\delta,r_1+2\delta]$, making $dt_{12}$ timelike there, but losing the timelike character of $dt_{12}$ in a subset of the transition region $(r_1+2\delta,r_2-2\delta)$. Moreover, similarly to the choices of $c_2$ and $c_3$ above, we take $c_1=\mu^{-1}$ in $[r_1+\delta,r_1+2\delta]$ and $c_1=-\mu^{-1}$ in $[r_1-2\delta,r_1-\delta]$.

Using the coordinates $t_{12},r,\omega$, we thus have $\wt H_{F,2}=\{r=r_2-2\delta, \ t_{12}\geq 0\}$; we further define
\begin{equation}
\label{EqRNdSHypersurfaceF2}
  H_{F,2} = \{ \eps t_{12}+r=r_2-2\delta, \ r_1+2\delta\leq r\leq r_2-2\delta \};
\end{equation}
thus, $H_{F,2}$ intersects $H_I$ at $t_{12}=0$, $r=r_2-2\delta$. We choose $\eps$ as follows: we calculate the squared norm of the conormal of $H_{F,2}$ using \eqref{EqRNdSg12} as
\begin{align*}
  G(\eps\,dt_{12}+dr,\eps\,dt_{12}+dr) &= -(2c_1+\mu c_1^2)\eps^2 + 2(1+\mu c_1)\eps - \mu \\
    &= (1-\eps c_1)(2\eps - \mu(1-\eps c_1)),
\end{align*}
which is positive in $[r_1+2\delta,r_2-2\delta]$ provided $\eps>0$, $1-\eps c_1>0$, since $\mu<0$ in this region. Therefore, choosing $\eps$ so that it verifies these inequalities, $H_{F,2}$ is spacelike. Put $t_{12,0}:=\eps^{-1}((r_2-2\delta)-(r_1+2\delta))$, so $t_{12}=t_{12,0}$ at $\{r=r_1+2\delta\}\cap H_{F,2}$, and define
\begin{equation}
\label{EqRNdSHypersurfacesTildeF1FT}
\begin{gathered}
  \wt H_{F,1} = \{ t_{12} \geq t_{12,0}, \ r=r_1+2\delta \}, \\
  \wt H_F = \{ t_{12} = t_{12,0}, \ r_1-2\delta\leq r\leq r_1+2\delta \}, \\
  \wt H_T = \{ t_{12} \geq t_{12,0}, \ r=r_1-2\delta \}.
\end{gathered}
\end{equation}
We note that $\wt H_{F,1}$ is indeed spacelike, as $G(dr,dr)=-\mu>0$ there, and $\wt H_F$ is spacelike by construction of $t_{12}$. The surface $\wt H_T$ is timelike (hence the subscript). Putting
\begin{equation}
\label{EqRNdSOmega12}
  \Omega_{12}^\circ = \{ \eps t_{12}+r \geq r_2-2\delta, \ r_1+2\delta\leq r\leq r_2-2\delta \}
\end{equation}
finishes the definition of all objects in Figure~\ref{FigRNdSExt12}.

In order to justify the subscripts `F', we compute a smooth choice of time orientation: first of all, $dt_{12}$ is future timelike (by choice) in $r\geq r_2-2\delta$; furthermore, in $(r_1,r_2)$, we have $G(dr,dr)=-\mu>0$, so $dr$ is timelike in $(r_1,r_2)$. We then calculate
\[
  \la -dr,dt_{12}\ra_G = 2 > 0
\]
in $[r_2-2\delta,r_2-\delta]$, so $-dr$ and $dt_{12}$ are in the same causal cone there, in particular $-dr$ is future timelike in $(r_1,r_2-\delta]$, which justifies the notation $\wt H_{F,1}$; furthermore $dt_{12}$ is timelike for $r\leq r_1+2\delta$, with
\[
  \la -dr,-dt_{12}\ra_G = 2 > 0
\]
in $[r_1+\delta,r_1+2\delta]$ (using the form \eqref{EqRNdSg12} of the metric with $c_1=\mu^{-1}$ there), hence $-dr$ and $-dt_{12}$ are in the same causal cone here. Thus, $dt_{12}$ is \emph{past} timelike in $r\leq r_1+2\delta$, justifying the notation $\wt H_F$. See also Figure~\ref{FigRNdSRadial} below. Lastly, for $H_{F,2}$, we compute
\[
  G(\eps\,dt_{12}+dr,-dr) = -\eps(1+\mu c_1)+\mu = \mu(1-\eps c_1)-\eps < 0
\]
by our choice of $\eps$, hence the future timelike 1-form $-dr$ is indeed outward pointing at $H_{F,2}$. We remark that from the perspective of $\Omega_{12}^\circ$, the surface $\wt H_{F,2}$ is initial, but we keep the subscript `F' for consistency with the notation used in the discussion of $\Omega_{01}^\circ$.

\begin{figure}[!ht]
  \centering
  \includegraphics{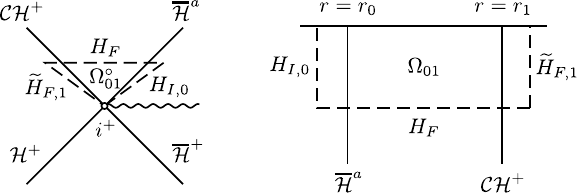}
  \caption{\textit{Left:} Penrose diagram of the region $\Omega_{01}^\circ$, bounded by the final Cauchy hypersurface $H_F$ and two initial hypersurfaces $H_{I,0}$ and $\wt H_{F,1}$. The artificial extension in the region behind the Cauchy horizon removes the curvature singularity and generates an artificial horizon $\ol\cH^a$. \textit{Right:} The same region, compactified at infinity, with the artificial hypersurfaces put in.}
\label{FigRNdSExt01}
\end{figure}

One can now analyze linear waves on the spacetime $\Omega_{12}^\circ\cup\Omega_{23}^\circ$ if one uses the reflection of singularities at $\wt H_T$. (We will describe the null-geodesic flow in \S\ref{SubsecRNdSFlow}.) However, we proceed as explained in \S\ref{SecIntro} and add an artificial exterior region to the region $r\leq r_1-2\delta$; see Figure~\ref{FigRNdSExt01}. We first note that the form of the metric in $r\leq r_1-\delta$ is
\[
  g = \mu\,dt_{12}^2 - \mu^{-1}\,dr^2 - r^2\,d\omega^2,
\]
thus of the same form as \eqref{EqRNdSMetric}. Define a function $\mu_*\in\CI((0,\infty))$ such that
\begin{equation}
\label{EqRNdSMuStar}
\begin{split}
  \mu_* & \equiv\mu\tn{ in }[r_1-2\delta,\infty), \\
  \mu_* & \ \tn{has a single simple zero at }r_0\in(0,r_1), \tn{and} \\
  r^{-2}\mu_* & \ \tn{has a unique non-degenerate critical point }r_{P,*}\in(r_0,r_1),
\end{split}
\end{equation}
so $\mu_*'>0$ on $[r_0,r_{P,*})$ and $\mu_*'<0$ on $(r_{P,*},r_1]$, see Figure~\ref{FigRNdSMuStar}. One can in fact drop the last assumption on $\mu_*$, as we will do in the Kerr--de Sitter discussion for simplicity, but in the present situation, this assumption allows for the nice interpretation of the appended region as a `past' or `backwards' version of the exterior region of a black hole.

\begin{figure}[!ht]
  \centering
  \includegraphics{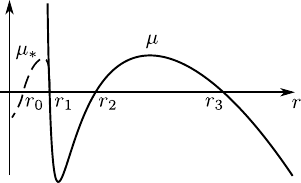}
  \caption{Illustration of the modification of $\mu$ (solid) in the region $r<r_1$ beyond the Cauchy horizon to a smooth function $\mu_*$ (dashed where different from $\mu$). Notice that the $\mu_*$ has the same qualitative properties near $[r_0,r_1]$ as near $[r_2,r_3]$.}
\label{FigRNdSMuStar}
\end{figure}

We extend the metric to $(r_0,r_1)$ by defining $g:=\mu_*\,dt_*^2 - \mu_*^{-1}\,dr^2 - r^2\,d\omega^2$. We then extend $g$ beyond $r=r_0$ as in \eqref{EqRNdSg23}: put
\[
  t_{01} = t_{12} - F_{01},
\]
with $F_{01}\in\CI((r_0,r_1))$, $F_{01}'=s_j(\mu_*^{-1}+c_0)$ when $|r-r_j|<2\delta$, $j=0,1$, where we set $s_0=-\sgn\mu_*'(r_0)=-1$; further let $c_0=-\mu_*^{-1}$ for $|r-r_j|\geq\delta$, so $t_{01}=t_{12}$ in $(r_1-2\delta,r_1-\delta)$ (up to redefining $F_{01}$ by an additive constant). Then, in $(t_{01},r,\omega)$-coordinates, the metric $g$ takes the form \eqref{EqRNdSg23} near $r_0$, with $t_{23}$ replaced by $t_{01}$ and $s_j=s_0=-1$; hence $g$ extends across $r=r_0$ as a non-degenerate stationary Lorentzian metric, and we can choose $c_0$ to be smooth across $r=r_0$ so that $dt_{01}$ is timelike in $[r_0-2\delta,r_1)$, and such that moreover $c_0=\mu_*^{-1}$ in $r<r_0-\delta$, thus ensuring the form \eqref{EqRNdSg23NearR2} of the metric (replacing $t_{23}$ and $s_j$ by $t_{01}$ and $-1$, respectively).

We can glue the functions $t_{01}$ and $t_{12}$ together by defining the smooth function $t_*$ in $[r_0-2\delta,r_2)$ to be equal to $t_{01}$ in $[r_0-2\delta,r_1)$ and equal to $t_{12}$ in $[r_1-\delta,r_2)$. Define
\begin{equation}
\label{EqRNdSHypersurfacesFI0}
\begin{gathered}
  H_F = \{ t_*\geq t_{12,0}, \ r_0-2\delta\leq r\leq r_1+2\delta \},
  H_{I,0} = \{ t_*\geq t_{12,0}, \ r=r_0-2\delta \};
\end{gathered}
\end{equation}
note here that $dt_{01}$ is \emph{past} timelike in $[r_0-2\delta,r_1)$. Lastly, we put
\begin{equation}
\label{EqRNdSOmega01}
  \Omega_{01}^\circ = \{ t_*\geq t_{12,0}, \ r_0-2\delta\leq r\leq r_1+2\delta \}.
\end{equation}
Note that in the region $\Omega_{01}^\circ$, we have produced an artificial horizon $\ol\cH^a$ at $r=r_0$. Again, the notation $\wt H_{F,1}$ is incorrect from the perspective of $\Omega_{01}^\circ$, but is consistent with the notation used in the discussion of $\Omega_{12}^\circ$.

Let us summarize our construction:

\begin{prop}
\label{PropRNdSMfd}
  Fix parameters $\Lambda>0,\bhm>0,Q>0$ of a Reissner--Nordstr\"om--de Sitter spacetime which is non-degenerate in the sense of Definition~\ref{DefRNdSNonDegenerate}. Let $\mu_*$ be a smooth function on $(0,\infty)$ satisfying \eqref{EqRNdSMuStar}, where $\mu$ is given by \eqref{EqRNdSMetric}. For $\delta>0$ small, define the manifold $M^\circ=\R_{t_*}\times (r_0-4\delta,r_3+4\delta)_r \times\Sph^2_\omega$ and equip $M^\circ$ with a smooth, stationary, non-degenerate Lorentzian metric $g$, which has the form
  \begin{align}
    \label{EqRNdSMetricOuterComm} g &= \mu_*\,dt_*^2 - \mu_*^{-1}\,dr^2 - r^2\,d\omega^2, \quad r\in[r_0+\delta,r_1-\delta]\cup[r_2+\delta,r_3-\delta], \\
    \begin{split}
    \label{EqRNdSMetricTransition} g &= \mu_*\,dt_*^2 + 2s_j(1+\mu_* c_j)\,dt_*\,dr+(2c_j+\mu_* c_j^2)\,dr^2-r^2\,d\omega^2, \\
        &\qquad\qquad |r-r_j|\leq 2\delta, \tn{ or }r\in[r_1+2\delta,r_2-2\delta],\ j=1,
    \end{split} \\
    \begin{split}
    \label{EqRNdSMetricBeyond} g &= \mu_*\,dt_*^2 + 4s_j\,dt_*\,dr + 3\mu_*^{-1}\,dr^2-r^2\,d\omega^2, \\
        &\qquad\qquad r\in[r_j-2\delta,r_j-\delta],\ j=0,2,\tn{ or }r\in[r_j+\delta,r_j+2\delta],\ j=1,3,
     \end{split}
  \end{align}
  in $[r_0-2\delta,r_3+2\delta]$, where $s_j=-\sgn\mu_*'(r_j)$. Then the region $r_1-2\delta\leq r\leq r_3+2\delta$ is isometric to a region in the Reissner--Nordstr\"om--de Sitter spacetime with parameters $\Lambda,\bhm,Q$, with $r_2<r<r_3$ isometric to the exterior domain (bounded by the event horizon $\cH^+$ at $r=r_2$ and the cosmological horizon $\ol\cH^+$ at $r=r_3$), $r_1<r<r_2$ isometric to the black hole region (bounded by the future Cauchy horizon $\cCH^+$ at $r=r_1$ and the event horizon), and $r_1-2\delta<r<r_1$ isometric to a region beyond the future Cauchy horizon. (See Figure~\ref{FigRNdSExtFull}.) Furthermore, $M^\circ$ is time-orientable.

  One can choose the smooth functions $c_j=c_j(r)$ such that $c_j(r_j)<0$ and
  \[
  \begin{split}
    dt_*\tn{ is } & \tn{past timelike in }r\leq r_1+2\delta,\tn{ and} \\
        & \tn{future timelike in }r\geq r_2-2\delta;
  \end{split}
  \]
  The hypersurfaces
  \begin{equation}
  \label{EqRNdSMfdHyper}
  \begin{split}
    H_{I,0} &= \{ t_*\geq t_{*,0},\ r=r_0-2\delta \}, \\
    H_F &= \{ t_*=t_{*,0},\ r_0-2\delta\leq r\leq r_1+2\delta \}, \\
    H_{F,2} &= \{ \eps t_*+r=r_2-2\delta,\ r_1+2\delta\leq r\leq r_2-2\delta \}, \\
    H_I &= \{ t_*=0,\ r_2-2\delta\leq r\leq r_3+2\delta \}, \\
    H_{F,3} &= \{ t_*\geq 0,\ r_3=2\delta \}
  \end{split}
  \end{equation}
  are spacelike provided $\eps>0$ is sufficiently small; here $t_{*,0}=\eps^{-1}(r_2-r_1-4\delta)$. They bound a domain $\Omega^\circ$, which is a submanifold of $M^\circ$ with corners. (Recall Remark~\ref{RmkRNdSHypersurfaceNotation} for our conventions in naming the hypersurfaces.)

  $M^\circ$ and $\Omega^\circ$ possess natural partial compactifications $M$ and $\Omega$, respectively, obtained by introducing $\tau=e^{-t_*}$ and adding to them their ideal boundary at infinity, $\tau=0$; the metric $g$ is a non-degenerate Lorentzian b-metric on $M$ and $\Omega$.
\end{prop}

Adding $\tau=0$ to $M^\circ$ means defining
\[
  M = \bigl(M^\circ \sqcup [0,1)_\tau\times(r_0-4\delta,r_3+4\delta)\times\Sph^2_\omega\bigr)/\sim,
\]
where $(\tau,r,\omega)\in(0,1)\times(r_0-4\delta,r_3+4\delta)\times\Sph^2$ is identified with the point $(t_*=-\log\tau,r,\omega)\in M^\circ$, and we define the smooth structure on $M$ by declaring $\tau$ to be a smooth boundary defining function.

\begin{proof}[Proof of Proposition~\ref{PropRNdSMfd}.]
  The extensions described above amount to a direct construction of a manifold $\R_{t_*}\times[r_0-2\delta,r_3+2\delta]_r\times\Sph^2_\omega$, where we obtained the function $t_*$ by gluing $t_{01}$ and $t_{12}$ in $[r_1-2\delta,r_1-\delta]$, and similarly $t_{12}$ and $t_{23}$ in $[r_2-2\delta,r_2-\delta]$; we then extend the metric $g$ non-degenerately to a stationary metric in $r>r_0-4\delta$ and $r<r_3+4\delta$, thus obtaining a metric $g$ on $M^\circ$ with the listed properties.
\end{proof}

\begin{figure}[!ht]
  \centering
  \includegraphics{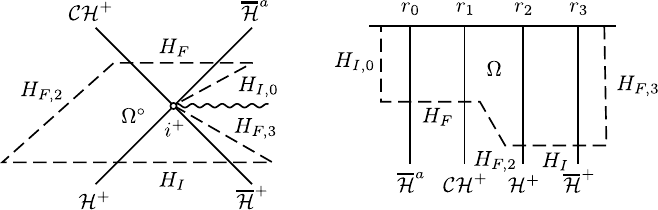}
  \caption{\emph{Left:} The Penrose diagram for $\Omega^\circ$, which is the diagram of Reissner--Nordstr\"om--de Sitter in a neighborhood of the exterior domain and of the black hole region as well as near the Cauchy horizon; further beyond the Cauchy horizon, we glue in an artificial exterior region, eliminating the singularity at $r=0$. \emph{Right:} The compactification of $\Omega^\circ$ to a manifold with corners $\Omega$; the smooth structure of $\Omega$ is the one induced by the embedding of $\Omega$ into the plane (cross $\Sph^2$) as displayed here.}
\label{FigRNdSExtFull}
\end{figure}

We define the regions $\Omega_{01}^\circ,\Omega_{12}^\circ$ and $\Omega_{23}^\circ$ as in \eqref{EqRNdSOmega01}, \eqref{EqRNdSOmega12} and \eqref{EqRNdSOmega23}, respectively, as submanifolds of $\Omega^\circ$ with corners; their boundary hypersurfaces are hypersurfaces within $\Omega^\circ$. We denote the closures of these domains and hypersurfaces in $\Omega$ by the same names, but dropping the superscript `$\circ$'. Furthermore, we write
\begin{equation}
\label{EqRNdSBoundaries}
  X = \pa M, \quad Y=\Omega\cap\pa M
\end{equation}
for the ideal boundaries at infinity.

\subsection{Global behavior of the null-geodesic flow}
\label{SubsecRNdSFlow}

One reason for constructing the compactification $\Omega$ step by step is that the null-geodesic dynamics almost decouple in the subdomains $\Omega_{01}$, $\Omega_{12}$ and $\Omega_{23}$, see Figures~\ref{FigRNdSExt01}, \ref{FigRNdSExt12} and \ref{FigRNdSExt23}.

We denote by $G$ the dual metric of $g$. We recall that we can glue $\frac{d\tau}{\tau}=-dt_*$ in $\Omega_{01}$, $-dr$ in $[r_1+\delta,r_2-\delta]$ and $-\frac{d\tau}{\tau}=dt_*$ in $\Omega_{23}$ together using a non-negative partition of unity and obtain a 1-form
\[
  \varpi\in\CI(\Omega,\Tb^*_\Omega M)
\]
which is everywhere future timelike in $\Omega$. Thus, the characteristic set of $\Box_g$,
\[
  \Sigma = G^{-1}(0) \subset \Tb^*_\Omega M\setminus o,
\]
with $G(\zeta)=\la\zeta,\zeta\ra_G$ the dual metric function, globally splits into two connected components
\begin{equation}
\label{EqRNdSFlowCharComp}
  \Sigma=\Sigma_+\cup\Sigma_-, \quad \Sigma_\pm = \{ \zeta\in\Sigma \colon \mp\la\zeta,\varpi\ra_G>0 \}.
\end{equation}
(Indeed, if $\la\zeta,\varpi\ra_G=0$, then $\zeta\in\la\varpi\ra^\perp$, which is spacelike, so $G(\zeta)=\la\zeta,\zeta\ra_G=0$ shows that $\zeta=0$.) Thus, $\Sigma_+$, resp.\ $\Sigma_-$, is the union of the past, resp.\ future, causal cones. We note that $\Sigma$ and $\Sigma_\pm$ are smooth codimension 1 submanifolds of $\Tb^*_\Omega M\setminus o$ in view of the Lorentzian nature of the dual metric $G$. Moreover, $\Sigma_\pm$ is transversal to $\Tb^*_Y M$, in fact the differentials $d G$ and $d\tau$ ($\tau$ lifted to a function on $\Tb^*M$) are linearly independent everywhere in $\Tb^*_\Omega M\setminus o$.

We begin by analyzing the null-geodesic flow (in the b-cotangent bundle) near the horizons: we will see that the Hamilton vector field $\ham_G$ has critical points where the horizons intersect the ideal boundary $Y$ of $\Omega$; more precisely, $\ham_G$ is radial there. In order to simplify the calculations of the behavior of $\ham_G$ nearby, we observe that the smooth structure of the compactification $\Omega$, which is determined by the function $\tau=e^{-t_*}$, is unaffected by the choice of the functions $c_j$ in Proposition~\ref{PropRNdSMfd}, since changing $c_j$ merely multiplies $\tau$ by a positive function that only depends on $r$, hence is smooth on our initial compactification $\Omega$. Now, the intersections $Y\cap\{r=r_j\}$ are smooth boundary submanifolds of $M$, and we define
\[
  L_j := \Nb^*(Y\cap\{r=r_j\}),
\]
which is well-defined given merely the smooth structure on $\Omega$. The point of our observation then is that we can study the Hamilton flow near $L_j$ using any choice of $c_j$. Thus, introducing $t_0=t-F(r)$, with $F'=s_j\mu^{-1}$ near $r_j$, we find from \eqref{EqRNdSMetricTransition} that
\[
  g = \mu_*\,dt_0^2 + 2s_j\,dt_0\,dr - r^2\,d\omega^2, \quad G = 2s_j\pa_{t_0}\pa_r - \mu_*\pa_r^2 - r^{-2}\pa_\omega^2.
\]
Let $\tau_0:=e^{-t_0}$. Then, with $\pa_{t_0}=-\tau_0\pa_{\tau_0}$, and writing b-covectors as
\[
  \sigma\,\frac{d\tau_0}{\tau_0} + \xi\,dr + \eta\,d\omega,
\]
the dual metric function $G\in\CI(\Tb^*_\Omega M)$ near $L_j$ is then given by
\begin{equation}
\label{EqRNdSMetricWithTau0}
  G = -2s_j\sigma\xi - \mu_*\xi^2 - r^{-2}|\eta|^2.
\end{equation}
Correspondingly, the Hamilton vector field is
\[
  \ham_G = -2s_j\xi\tau_0\pa_{\tau_0} - 2(s_j\sigma+\mu_*\xi)\pa_r - r^{-2}\ham_{|\eta|^2} + (\mu_*'\xi^2-2r^{-3}|\eta|^2)\pa_\xi.
\]
To study the $\ham_G$-flow in the radially compactified b-cotangent bundle near $\pa L_j$, we introduce rescaled coordinates
\begin{equation}
\label{EqRNdSFlowRescaledCoord}
  \wh\rho = \frac{1}{|\xi|}, \quad \wh\eta=\frac{\eta}{|\xi|}, \quad \wh\sigma=\frac{\sigma}{|\xi|}.
\end{equation}
We then compute the rescaled Hamilton vector field in $\pm\xi>0$ to be
\begin{align*}
  \rham_G = \wh\rho\ham_G &= \mp 2s_j\tau_0\pa_{\tau_0} - 2(s_j\wh\sigma\pm\mu_*)\pa_r - \wh\rho r^{-2}\ham_{|\eta|^2} \\
    &\qquad \mp (\mu_*'-2r^{-3}|\wh\eta|^2)(\wh\rho\pa_{\wh\rho}+\wh\eta\pa_{\wh\eta}+\wh\sigma\pa_{\wh\sigma});
\end{align*}
writing $|\eta|^2=k^{ij}(y)\eta_i\eta_j$ in a local coordinate chart on $\Sph^2$, we have $\wh\rho\ham_{|\eta|^2}=2k^{ij}\wh\eta_i\pa_{y_j}-\pa_{y^k}g^{ij}(y)\wh\eta_i\wh\eta_j\pa_{\wh\eta_k}$. Thus, $\rham_G = \mp 2s_j\tau_0\pa_{\tau_0} \mp \mu_*'\wh\rho\pa_{\wh\rho}$ at $L_j\cap\{\pm\xi>0\}$. In particular,
\begin{equation}
\label{EqRNdSRadialPointHamDer}
  \tau_0^{-1}\rham_G\tau_0 = \mp 2s_j, \quad \wh\rho^{-1}\rham_G\wh\rho=\mp\mu_*'
\end{equation}
have opposite signs (by definition of $s_j$), and the quantity which will control regularity and decay thresholds at the radial set $L_j$ is the quotient
\begin{equation}
\label{EqRNdSRadialPointBeta}
  \beta_j := -\frac{\tau_0^{-1}\rham_G\tau_0}{\wh\rho^{-1}\rham_G\wh\rho} = \frac{2}{|\mu_*'(r_j)|};
\end{equation}
see Definition~\ref{DefRNdSOrderFunctions} and the proof of Proposition~\ref{PropRNdSGlobalReg} for their role. We remark that the reciprocal
\begin{equation}
\label{EqRNdSSurfGrav}
  \kappa_j:=\beta_j^{-1}
\end{equation}
is equal to the \emph{surface gravity} of the horizon at $r=r_j$, see e.g.\ \cite{DafermosRodnianskiLectureNotes}.

We proceed to verify that $\pa L_j\subset\rcTb^*_X M$ is a source/sink for the $\rham_G$-flow within $\rcTb^*_X M$ by constructing a quadratic defining function $\rho_0$ of $\pa L_j$ within $\Sigma\cap\Sb^*_X M$ for which
\begin{equation}
\label{EqRNdSQuadraticDef}
  \pm s_j\rham_G\rho_0 \geq \beta_q\rho_0,\quad \beta_q>0,
\end{equation}
modulo terms which vanish cubically at $L_j$; note that $\pm s_j\rham_G\wh\rho=|\mu_*'|\wh\rho$ has the same relative sign. Now, $\pa L_j$ is defined within $\tau=0,\wh\rho=0$ by the vanishing of $\wh\eta$ and $\wh\sigma$, and we have $\pm s_j\rham_G|\wh\eta|^2=2|\mu_*'| |\wh\eta|^2$, likewise for $\wh\sigma$; therefore
\[
  \rho_0 := |\wh\eta|^2 + |\wh\sigma|^2
\]
satisfies \eqref{EqRNdSQuadraticDef}. (One can in fact easily diagonalize the linearization of $\rham_G$ at its critical set $\pa L_j$ by observing that
\[
  \rham_G\wh\sigma=\mp\mu'\wh\sigma,\quad \rham_G\wh\eta=\mp\mu'\wh\eta, \quad \rham_G\bigl((r-r_j)\mp\beta_j\wh\sigma\bigr)=\mp 2\mu'\bigl((r-r_j)\mp\beta_j\wh\sigma\bigr)
\]
modulo quadratically vanishing terms.)

Further studying the flow at $r=r_j$, we note that $dr$ is null there, and writing
\begin{equation}
\label{EqRNdSTbCoords}
  \zeta=\sigma\,\frac{d\tau}{\tau}+\xi\,dr+\eta\,d\omega,
\end{equation}
a covector $\zeta\in\Sigma\cap\{r=r_j\}$ is in the orthocomplement of $dr$ if and only if $0=\la dr,\zeta\ra_G=-s_j\sigma$ (using the form \eqref{EqRNdSMetricTransition} of the metric), which then implies $\eta=0$ in view of $\zeta\in\Sigma$. Since $\ham_G r=2\la dr,\zeta\ra_G$, we deduce that $\ham_G r\neq 0$ at $\Sigma\cap\{r=r_j\}\setminus\cL_j$, where we let
\[
  \cL_j=\Nb^*\{r=r_j\} = \{ r=r_j,\ \sigma=0,\ \eta=0 \};
\]
we note that this set is invariant under the Hamilton flow. More precisely, we have $\la dr,\frac{d\tau}{\tau}\ra_G=-s_j$, so for $j=3$, i.e.\ at $r=r_3$, $dr$ is in the same causal cone as $-d\tau/\tau$, hence in the future null cone; thus, letting $\cL_{j,\pm}=\cL_j\cap\Sigma_\pm$ and taking $\zeta\in\Sigma_-\cap\{r=r_3\}\setminus\cL_{3,-}$, we find that $\zeta$ lies in the same causal cone as $dr$, but $\zeta$ is not orthogonal to $dr$, hence we obtain $\ham_G r>0$; more generally,
\begin{equation}
\label{EqRNdSCrossingHorizon}
  \pm\ham_G r < 0\quad \tn{ at }\Sigma_\pm\cap\{r=r_3\}\setminus\cL_{3,\pm}.
\end{equation}
It follows that forward null-bicharacteristics in $\Sigma_+$ can only cross $r=r_3$ in the inward direction ($r$ decreasing), while those in $\Sigma_-$ can only cross in the outward direction ($r$ increasing). At $r=r_0$, there is a sign switch both in the definition of $\Sigma_\pm$ (because there $-d\tau/\tau$ is \emph{past} timelike) and in $s_0=-1$, so the same statement holds there. At $r=r_2$, there is a single sign switch in the calculation because of $s_2=-1$, and at $r=r_1$ there is a single sign switch because of the definition of $\Sigma_\pm$ there, so forward null-bicharacteristics in $\Sigma_+$ can only cross $r=r_1$ or $r=r_2$ in the inward direction ($r$ decreasing), and forward bicharacteristics in $\Sigma_-$ only in the outward direction ($r$ increasing).

Next, we locate the radial sets $L_j$ within the two components of the characteristic set, i.e.\ determining the components
\[
  L_{j,\pm} := L_j \cap \Sigma_\pm
\]
of the radial sets. The calculations verifying the initial/final character of the artificial hypersurfaces appearing in the arguments of the previous section show that $\la dr,d\tau/\tau\ra<0$ at $r_1$ and $r_3$, while $\la dr,d\tau/\tau\ra>0$ at $r_0$ and $r_2$, so since $\Sigma_-$, resp.\ $\Sigma_+$, is the union of the future, resp.\ past, null cones, we have
\begin{gather*}
  L_{j,\pm}=\{\mp\xi>0\}\cap L_j,\quad j=0,3, \\
  L_{j,\pm}=\{\pm\xi>0\}\cap L_j,\quad j=1,2.
\end{gather*}
In view of \eqref{EqRNdSRadialPointHamDer} and taking into account that $\tau_0$ differs from $\tau$ by an $r$-dependent factor, while $\ham_G r=0$ at $L_j$, we thus have
\begin{equation}
\label{EqRNdSRadialPointHamTau}
\begin{split}
  \mp\tau^{-1}\rham_G\tau &= 2\quad\tn{at }L_{j,\pm},\ j=0,1, \\
  \pm\tau^{-1}\rham_G\tau &= 2\quad\tn{at }L_{j,\pm},\ j=2,3, \\
\end{split}
\end{equation}
We connect this with Figure~\ref{FigRNdSExtFull}. Namely, if we let $\cL_{j,\pm}=\cL_j\cap\Sigma_\pm$, then $\cL_{j,-}$ is the unstable manifold at $L_{j,-}$ for $j=0,1$ and the stable manifold at $L_{j,-}$ for $j=2,3$, and the other way around for $\cL_{j,+}$. In view of \eqref{EqRNdSRadialPointHamDer}, $L_{j,-}$ is a sink for the $\rham_G$ flow within $\Sb^*_X M$ for $j=0,1$, while it is a source for $j=2,3$, with sink/source switched for the `$+$' sign. See Figure~\ref{FigRNdSRadial}.

\begin{figure}[!ht]
  \centering
  \includegraphics{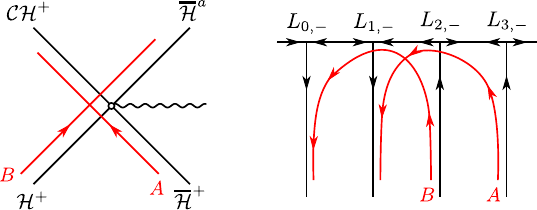}
  \caption{Saddle point structure of the null-geodesic flow within the component $\Sigma_-$ of the characteristic set, and the behavior of two null-geodesics. The arrows on the horizons are future timelike. In $\Sigma_+$, all arrows are reversed.}
\label{FigRNdSRadial}
\end{figure}

We next shift our attention to the two domains of outer communications, $r_0<r<r_1$ in $\Omega_{01}$ and $r_2<r<r_3$ in $\Omega_{23}$, where we study the behavior of the radius function along the flow using the form \eqref{EqRNdSMetricOuterComm} of the metric: thus, at a point $\zeta=\sigma\,\frac{d\tau}{\tau}+\xi\,dr+\eta\,d\omega\in\Sigma$, we have $\ham_G r=-2\mu_*\xi$, so $\ham_G r=0$ necessitates $\xi=0$, hence $r^{-2}|\eta|^2=\mu_*^{-1}\sigma^2$, and thus we get
\begin{equation}
\label{EqRNdSRadiusCritical}
  \ham_G^2 r=-2r^2\mu_*^{-1}\sigma^2(r^{-2}\mu_*)'.
\end{equation}
Now for $r\in(r_2,r_3)$,
\begin{equation}
\label{EqRNdSFlowMuPrime}
  (r^{-2}\mu_*)' = -2r^{-5}(r^2-3Mr+2Q^2)
\end{equation}
vanishes at the radius $r_P=\frac{3M}{2}+\sqrt{\frac{9M^2-8Q^2}{4}}$ of the \emph{photon sphere}, and $(r-r_P)(r^{-2}\mu_*)'<0$ for $r\neq r_P$; likewise, for $r\in(r_0,r_1)$, by construction \eqref{EqRNdSMuStar} we have $(r^{-2}\mu_*)'=0$ only at $r=r_{P,*}$, and $(r-r_{P,*})(r^{-2}\mu_*)'<0$ for $r\neq r_{P,*}$. Therefore, if $\ham_G r=0$, then $\ham_G^2 r>0$ unless $r=r_{P(,*)}$, in which case $\zeta$ lies in the \emph{trapped set}
\[
  \wt\Gamma_{(*)} = \{ (\tau,r,\omega;\sigma,\xi,\eta) \in \Sigma \colon r=r_{P(,*)}, \xi=0 \}.
\]
Restricting to bicharacteristics within $X=\{\tau=0\}$ (which is invariant under the $\ham_G$-flow since $\ham_G\tau=0$ there) and defining
\[
  \Gamma_{(*)} = \wt\Gamma_{(*)} \cap \{\tau=0\},
\]
we can conclude that all critical points of $F_{(*)}(r):=(r-r_{P(,*)})^2$ along null-geodesics in $(r_2,r_3)$ (or $(r_0,r_1)$) are strict local minima: indeed, if $\ham_G F_{(*)}=2(r-r_{P(,*)})\ham_G r=0$ at $\zeta$, then \emph{either} $r=r_{P(,*)}$, in which case $\ham_G^2 F_{(*)}=2(\ham_G r)^2>0$ unless $\ham_G r=0$, hence $\zeta\in\Gamma_{(*)}$, \emph{or} $\ham_G r=0$, in which case $\ham_G^2 F_{(*)}=2(r-r_{P(,*)})\ham_G^2 r>0$ unless $r=r_{P(,*)}$, hence again $\zeta\in\Gamma_{(*)}$. As in \cite[\S6.4]{VasyMicroKerrdS}, this implies that within $X$, forward null-bicharacteristics in $(r_2,r_3)$ (resp.\ $(r_0,r_1)$) either tend to $\Gamma\cup L_{2,+}\cup L_{3,+}$ (resp.\ $\Gamma_*\cup L_{0,-}\cup L_{1,-}$), or they reach $r=r_2$ or $r=r_3$ (resp.\ $r=r_0$ or $r=r_1$) in finite time, while backward null-bicharacteristics either tend to $\Gamma\cup L_{2,-}\cup L_{3,-}$ (resp.\ $\Gamma_*\cup L_{0,+}\cup L_{1,+}$), or they reach $r=r_2$ or $r=r_3$ (resp.\ $r=r_0$ or $r=r_1$) in finite time. (For this argument, we make use of the source/sink dynamics at $L_{j,\pm}$.) Further, they cannot tend to $\Gamma$, resp.\ $\Gamma_*$, in both the forward and backward direction \emph{while remaining in $(r_0,r_1)$, resp.\ $(r_2,r_3)$,} unless they are trapped, i.e.\ contained in $\Gamma$, resp.\ $\Gamma_*$, since otherwise $F_{(*)}$ would attain a local maximum along them. Lastly, bicharacteristics reaching a horizon $r=r_j$ in finite time in fact cross the horizon by our earlier observation. The trapping at $\Gamma_{(*)}$ is in fact \emph{$r$-normally hyperbolic for every $r$} \cite{WunschZworskiNormHypResolvent}.

Next, in $\mu_*<0$, we recall that $dr$ is future, resp.\ past, timelike in $r<r_0$ and $r>r_3$, resp.\ $r\in(r_1,r_2)$; therefore, if $\zeta\in\Sigma$ lies in one of these three regions, $\ham_G r=2\la dr,\zeta\ra_G$ implies
\begin{equation}
\label{EqRNdSHamRHypRegion}
\begin{split}
  \mp\ham_G r>0 &\tn{ in }\Sigma_\pm\cap\bigl(\{r<r_0\}\cup\{r>r_3\}\bigr), \\
  \pm\ham_G r>0 &\tn{ in }\Sigma_\pm\cap\{r_1<r<r_2\}.
\end{split}
\end{equation}
(This is consistent with \eqref{EqRNdSCrossingHorizon} and the paragraph following it.)

In order to describe the global structure of the null-bicharacteristic flow, we define the connected components of the trapped set in the exterior domain of the spacetime,
\[
  \Gamma = \Gamma^+ \cup \Gamma^-, \quad \Gamma^\pm = \Gamma\cap\Sigma_\pm;
\]
then $\Gamma^\pm$ have stable/unstable manifolds $\Gamma^\pm_\pm$, with the convention that $\Gamma^\pm_+\subset\Sb^*_X M$, while $\Gamma^\pm_-\subset\Sb^* M$ is transversal to $\Sb^*_X M$. Concretely, $\Gamma^-_-$ is the union of forward trapped bicharacteristics, i.e.\ bicharacteristics which tend to $\Gamma^-$ in the forward direction, while $\Gamma^-_+$ is the union of backward trapped bicharacteristics, tending to $\Gamma^-$ in the backward direction; further $\Gamma^+_-$ is the union of backward trapped bicharacteristics, and $\Gamma^+_+$ the union of forward trapped bicharacteristics, tending to $\Gamma^+$. See Figure~\ref{FigRNdSFlow}.

\begin{figure}[!ht]
  \centering
  \includegraphics{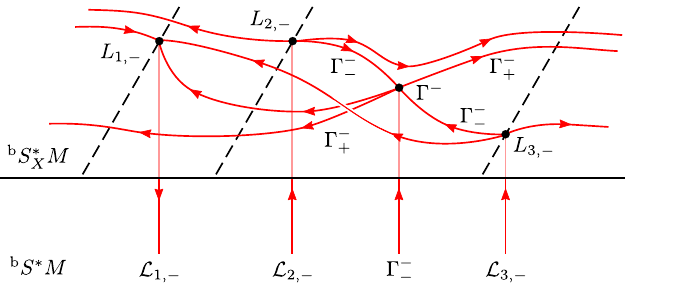}
  \caption{Global structure of the null-bicharacteristic flow in the component $\Sigma_-$ of the characteristic set and in the region $r>r_1-2\delta$ of the Reissner--Nordstr\"om--de Sitter spacetime. The picture for $\Sigma_+$ is analogous, with the direction of the arrows reversed, and $L_{j,-},\cL_{j,-},\Gamma^-_{(\pm)}$ replaced by $L_{j,+},\cL_{j,+},\Gamma^+_{(\pm)}$.}
\label{FigRNdSFlow}
\end{figure}

The structure of the flow in the neighborhood $\Omega_{01}$ of the artificial exterior region is the same as that in the neighborhood $\Omega_{23}$ of the exterior domain, except the time orientation and thus the two components of the characteristic set are reversed. Write $\Gamma^\pm_*=\Gamma_*\cap\Sigma_\pm$ a denote by $\Gamma^\pm_{*,\pm}$ the forward and backward trapped sets, with the same sign convention as for $\Gamma^\pm_\pm$ above. We note that backward, resp.\ forward, trapped null-bicharacteristics in $\Gamma^-_+\cup\Gamma^+_-$, resp.\ $\Gamma^-_-\cup\Gamma^+_+$, may be forward, resp.\ backward, trapped in the artificial exterior region, i.e.\ they may lie in $\Gamma^-_{*,-}\cup\Gamma^+_{*,+}$, resp.\ $\Gamma^-_{*,+}\cup\Gamma^+_{*,-}$, but this is the only additional trapping present in our setup. To state this succinctly, we write
\[
  L_{\tot,\pm} = \bigcup_{j=0}^3 L_{j,\pm}, \quad \Gamma_\tot^\pm = \Gamma^\pm \cup \Gamma_*^\pm.
\]
Then:

\begin{prop}
\label{PropRNdSFlow}
  The null-bicharacteristic flow in $\Sb^*_\Omega M$ has the following properties:
  \begin{enumerate}
    \item \label{ItRNdSFlowBdy} Let $\gamma$ be a null-bicharacteristic at infinity, $\gamma\subset\Sigma_-\cap\Sb^*_Y M\setminus(L_{\tot,-}\cup\Gamma_\tot^-)$, where $Y=\Omega\cap\pa M$. Then in the backward direction, $\gamma$ either crosses $H_{I,0}$ in finite time or tends to $L_{2,-}\cup L_{3,-}\cup\Gamma^-\cup\Gamma_*^-$, while in the forward direction, $\gamma$ either crosses $H_{F,3}$ in finite time or tends to $L_{0,-}\cup L_{1,-}\cup\Gamma^-\cup\Gamma_*^-$. The curve $\gamma$ can tend to $\Gamma^-$ in at most one direction, and likewise for $\Gamma_*^-$.
    \item \label{ItRNdSFlowInt} Let $\gamma$ be a null-bicharacteristic in $\Sigma_-\cap\Sb^*_{\Omega\setminus Y}M$. Then in the backward direction, $\gamma$ either crosses $H_{I,0}\cup H_I$ in finite time or tends to $L_{0,-}\cup L_{1,-}\cup\Gamma_*^-$, while in the forward direction, $\gamma$ either crosses $H_F\cup H_{F,2}\cup H_{F,3}$ in finite time or tends to $L_{2,-}\cup L_{3,-}\cup\Gamma^-$.
    \item \label{ItRNdSFlowHyp} In both cases, in the region where $r\in(r_1,r_2)$, $r\circ\gamma$ is strictly decreasing, resp.\ increasing, in the forward, resp.\ backward, direction in $\Sigma_-$, while in the regions where $r<r_0$ or $r>r_3$, $r\circ\gamma$ is strictly increasing, resp.\ decreasing, in the forward, resp.\ backward, direction in $\Sigma_-$.
    \item \label{ItRNdSFlowRadialTrapped} $L_{j,\pm}$, $j=0,\ldots,3$ as well as $\Gamma^\pm$ and $\Gamma_*^\pm$ are invariant under the flow.
  \end{enumerate}
  For null-bicharacteristics in $\Sigma_+$, the analogous statements hold with `backward' and `forward' reversed and `$+$' and `$-$' switched.
\end{prop}

Here $H_{I,0}$ etc.\ is a shorthand notation for $\Sb^*_{H_{I,0}}M$.

\begin{proof}
  Statement \itref{ItRNdSFlowHyp} follows from \eqref{EqRNdSHamRHypRegion}, and \itref{ItRNdSFlowRadialTrapped} holds by the definition of the radial and trapped sets. To prove the `backward' part of \itref{ItRNdSFlowBdy}, note that if $r<r_0$ on $\gamma$, then $\gamma$ crosses $H_{I,0}$ by \eqref{EqRNdSHamRHypRegion}; if $r=r_0$ on $\gamma$, then $\gamma$ crosses into $r<r_0$ since $\gamma\cap L_{0,-}=\emptyset$. If $\gamma$ remains in $r>r_0$ in the backward direction, it either tends to $\Gamma_*^-$, or it crosses $r=r_1$ since it cannot tend to $L_{0,-}\cup L_{1,-}$ because of the sink nature of this set. Once $\gamma$ crosses into $r>r_1$, it must tend to $r=r_2$ by \itref{ItRNdSFlowHyp} and hence either tend to the source $L_{2,-}$ or cross into $r>r_2$. In $r>r_2$, $\gamma$ must tend to $L_{2,-}\cup L_{3,-}\cup\Gamma^-$, as it cannot cross $r=r_2$ or $r=r_3$ into $r<r_2$ or $r>r_3$ in the backward direction. The analogous statement for $\Sigma_+$, now in the forward direction, is immediate, since reflecting $\gamma$ pointwise across the origin in the b-cotangent bundle but keeping the affine parameter the same gives a bijection between backward bicharacteristics in $\Sigma_-$ and forward bicharacteristics in $\Sigma_+$. The `forward' part of \itref{ItRNdSFlowBdy} is completely analogous.

  It remains to prove \itref{ItRNdSFlowInt}. Note that $\tau^{-1}\ham_G\tau=2\la\frac{d\tau}{\tau},\zeta\ra$ at $\zeta\in\Tb^*_\Omega M$; thus in $r\leq r_1+2\delta$, where $d\tau/\tau$ is future timelike, $\tau$ is strictly decreasing in the backward direction along bicharacteristics $\gamma\subset\Sigma_-$, hence the arguments for part~\itref{ItRNdSFlowBdy} show that $\gamma$ crosses $H_{I,0}$, or tends to $L_{0,-}\cup L_{1,-}\cup\Gamma_*^-$ if it lies in $\cL_{0,-}\cup\cL_{1,-}\cup\Gamma_{*,-}^-$; otherwise it crosses into $r>r_1$ in the backward direction. In the latter case, recall that in $r_1<r<r_2$, $r\circ\gamma$ is monotonically increasing in the backward direction; we claim that $\gamma$ cannot cross $H_{F,2}$: with the defining function $f:=\eps t_*+r$ of $H_{F,2}$, we arranged for $df$ to be past timelike, so $\ham_G f=2\la df,\zeta\ra<0$ for $\zeta\in\Sigma_-\cap\Tb^*_{H_{F,2}}M$, i.e.\ $f$ is increasing in the backward direction along the $\ham_G$-integral curve $\gamma$ near $H_{F,2}$, which proves our claim. This now implies that $\gamma$ enters $r\geq r_2-2\delta$ in the backward direction, from which point on $\tau$ is strictly increasing, hence $\gamma$ either crosses $H_I$ in $r\leq r_2$, or it crosses into $r>r_2$. In the latter case, it in fact crosses $H_I$ by the arguments proving \itref{ItRNdSFlowBdy}. The `forward' part is proved in a similar fashion.
\end{proof}

\subsection{Global regularity analysis}
\label{SubsecRNdSRegularity}

Forward solutions to the wave equation $\Box_g u=f$ in the domain of dependence of $H_I$, i.e.\ in $\Omega\cap\{r>r_1\}$, are not affected by any modifications of the operator $\Box_g$ outside, i.e.\ in $r\leq r_1$. As indicated in \S\ref{SecIntro}, we are therefore free to place complex absorbing operators at $\Gamma_*$ and $L_0$ which obviate the need for delicate estimates at normally hyperbolic trapping (see the proof of Proposition~\ref{PropRNdSGlobalReg}) and for a treatment of regularity issues at the artificial horizon (related to $\beta_j$ in \eqref{EqRNdSRadialPointBeta}, see also Definition~\ref{DefRNdSOrderFunctions}).

Concretely, let $\cU$ be a small neighborhood of $\pi L_0\cup\pi\Gamma_*$, with $\pi\colon\Tb^*M\to M$ the projection, so that
\begin{equation}
\label{EqRNdSComplexAbsorptionSupp}
  \cU\subset\{r_1-\delta<r<r_2-\delta,\ \tau\leq e^{-(t_{*,0}+2)}\}
\end{equation}
in the notation of Proposition~\ref{PropRNdSMfd}; thus, $\cU$ stays away from $H_{I,0}\cup H_F$. Choose $\cQ\in\Psib^2(M)$ with Schwartz kernel supported in $\cU\times\cU$ and real principal principal symbol satisfying
\[
  \mp\sigma(\cQ)(\zeta)\geq 0,\quad \zeta\in\Sigma_\pm,
\]
with the inequality strict at $L_{0,\pm}\cup\Gamma_*^\pm$, thus $\cQ$ is elliptic at $L_0\cup\Gamma_*$. We then study the operator
\begin{equation}
\label{EqRNdSOperator}
  \cP=\Box_g-i\cQ;
\end{equation}
the convention for the sign of $\Box_g$ is such that $\sigma_2(\Box_g)=G$. We will use weighted, variable order b-Sobolev spaces, with weight $\alpha\in\R$ and the order given by a function $\sfs\in\CI(\Sb^*M)$; in fact, the regularity will vary only in the base, not in the fibers of the b-cotangent bundle. We refer the reader to \cite[Appendix~A]{BaskinVasyWunschRadMink} and Appendix~\ref{SecVariable} for details on variable order spaces. We define the function space
\[
  \Hbfw^{\sfs,\alpha}(\Omega)
\]
as the space of restrictions to $\Omega$ of elements of $\Hb^{\sfs,\alpha}(M)=\tau^\alpha\Hb^s(M)$ which are supported in the causal future of $H_I\cup H_{I,0}$; thus, distributions in $\Hbfw^{\sfs,\alpha}$ are supported distributions at $H_I\cup H_{I,0}$ and extendible distributions at $H_F\cup H_{F,2}\cup H_{F,3}$ (and at $\pa M$), see \cite[Appendix~B]{HormanderAnalysisPDE3}; in fact, on manifolds with corners, there are some subtleties concerning such mixed supported/extendible spaces and their duals, which we discuss in Appendix~\ref{SecSuppExt}. The supported character at the initial surfaces, encoding vanishing Cauchy data, is the reason for the subscript `fw' (`forward'). The norm on $\Hbfw^{\sfs,\alpha}$ is the quotient norm induced by the restriction map, which takes elements of $\Hb^{\sfs,\alpha}(M)$ with the stated support property to their restriction to $\Omega$. Dually, we also consider the space
\[
  \Hbbw^{\sfs,\alpha}(\Omega),
\]
consisting of restrictions to $\Omega$ of distributions in $\Hb^{\sfs,\alpha}(M)$ which are supported in the causal past of $H_F\cup H_{F,2}\cup H_{F,3}$.

Concretely, for the analysis of $\cP$, we will work on slightly growing function spaces, i.e.\ allowing exponential growth of solutions in $t_*$; we will obtain precise asymptotics (in particular, boundedness) in the next section. In the present section, the stationary nature of the metric $g$ and of $\cP$ near $X$ is irrelevant; only the dynamical structure of the null-geodesic flow and the spacelike nature of the artificial boundaries are used.

Fix a weight
\[
  \alpha<0.
\]
The Sobolev regularity is dictated by the radial sets $L_1,L_2$ and $L_3$, as captured by the following definition:
\begin{definition}
\label{DefRNdSOrderFunctions}
  Let $\alpha\in\R$. Then a smooth function $\sfs=\sfs(r)$ is called a \emph{forward order function for the weight $\alpha$} if
  \begin{equation}
  \label{EqRNdSOrderFunctionFw}
  \begin{split}
    \sfs(r)&\tn{ is constant for }r<r_1+\delta'\tn{ and }r>r_1+2\delta', \\
    \sfs(r)&<1/2+\beta_1\alpha,\quad r<r_1+\delta', \\
    \sfs(r)&>1/2+\max(\beta_2\alpha,\beta_3\alpha),\quad r>r_1+2\delta', \\
    \sfs'(r)&\geq 0,
  \end{split}
  \end{equation}
  with $\beta_j$ defined in \eqref{EqRNdSRadialPointBeta}; here $\delta'\in(0,\delta)$ is any small number. The function $\sfs$ is called a \emph{backward order function for the weight $\alpha$} if
  \begin{equation}
  \label{EqRNdSOrderFunctionBw}
  \begin{split}
    \sfs(r)&\tn{ is constant for }r<r_1+\delta'\tn{ and }r>r_1+2\delta', \\
    \sfs(r)&>1/2+\beta_1\alpha,\quad r<r_1+\delta', \\
    \sfs(r)&<1/2+\max(\beta_2\alpha,\beta_3\alpha),\quad r>r_1+2\delta', \\
    \sfs'(r)&\leq 0;
  \end{split}
  \end{equation}
\end{definition}

Backward order functions will be used for the analysis of the dual problem.

\begin{rmk}
\label{RmkRNdSBeta1Computation}
  If $\beta_1<\max(\beta_2,\beta_3)$ (and $\alpha<0$ still), a forward order function $\sfs$ can be taken constant, and thus one can work on fixed order Sobolev spaces in Proposition~\ref{PropRNdSGlobalReg} below. This is the case for small charges $Q>0$: indeed, a straightforward computation in the variable $q=Q^2$ using Lemma~\ref{LemmaRNdSNonDegenerate} shows that
  \[
    \beta_1 = \frac{Q^4}{4\bhm^3} + \cO(Q^6).
  \]
\end{rmk}

Note that $\sfs$ is a forward order function for the weight $\alpha$ if and only if $1-\sfs$ is a backward order function for the weight $-\alpha$. The lower, resp.\ upper, bounds on the order functions at the radial sets are forced by the propagation estimate \cite[Proposition~2.1]{HintzVasySemilinear} which will we use at the radial sets: one can propagate high regularity from $\tau>0$ into the radial set and into the boundary (`red-shift effect'), while there is an upper limit on the regularity one can propagate out of the radial set and the boundary into the interior $\tau>0$ of the spacetime (`blue-shift effect'); the definition of order functions here reflects the precise relationship of the a priori decay or growth rate $\alpha$ and the regularity $\sfs$ (i.e.\ the `strength' of the red- or blue-shift effect depending on a priori decay or growth along the horizon). We recall the radial point propagation result in a qualitative form (the quantitative version of this, yielding estimates, follows from the proof of this result, or can be recovered from the qualitative statement using the closed graph theorem):

\begin{prop}
\label{PropRNdSRadialRecall}
  \cite[Proposition~2.1]{HintzVasySemilinear}. Suppose $\cP$ is as above, and let $\alpha\in\R$. Let $j=1,2,3$.

  If $s\geq s'$, $s'>1/2+\beta_j\alpha$, and if $u\in\Hb^{-\infty,\alpha}(M)$ then $L_{j,\pm}$ (and thus a neighborhood of $L_{j,\pm}$) is disjoint from $\WFb^{s,\alpha}(u)$ provided $L_{j,\pm}\cap\WFb^{s-1,\alpha}(\cP u)=\emptyset$, $L_{j,\pm}\cap\WFb^{s',\alpha}(u)=\emptyset$, and in a neighborhood of $L_{j,\pm}$, $\cL_{j,\pm}\cap\{\tau>0\}$ is disjoint from $\WFb^{s,\alpha}(u)$.

  On the other hand, if $s<1/2+\beta_j\alpha$, and if $u\in\Hb^{-\infty,\alpha}(M)$ then $L_{j,\pm}$ (and thus a neighborhood of $L_{j,\pm}$) is disjoint from $\WFb^{s,\alpha}(u)$ provided $L_{j,\pm}\cap\WFb^{s-1,\alpha}(\cP u)=\emptyset$ and a punctured neighborhood of $L_{j,\pm}$, with $L_{j,\pm}$ removed, in $\Sigma\cap\Sb^*_X M$ is disjoint from $\WFb^{s,\alpha}(u)$.
\end{prop}

We then have:

\begin{prop}
\label{PropRNdSGlobalReg}
  Suppose $\alpha<0$ and $\sfs$ is a forward order function for the weight $\alpha$; let $\sfs_0=\sfs_0(r)$ be a forward order function for the weight $\alpha$ with $\sfs_0<\sfs$. Then
  \begin{equation}
  \label{EqRNdSGlobalRegFw}
    \|u\|_{\Hbfw^{\sfs,\alpha}(\Omega)} \leq C(\|\cP u\|_{\Hbfw^{\sfs-1,\alpha}(\Omega)} + \|u\|_{\Hbfw^{\sfs_0,\alpha}(\Omega)}),
  \end{equation}

  We also have the dual estimate
  \begin{equation}
  \label{EqRNdSGlobalRegBw}
    \|u\|_{\Hbbw^{\sfs',-\alpha}(\Omega)} \leq C(\|\cP^* u\|_{\Hbbw^{\sfs'-1,-\alpha}(\Omega)} + \|u\|_{\Hbbw^{\sfs'_0,-\alpha}(\Omega)})
  \end{equation}
  for backward order functions $\sfs'$ and $\sfs'_0$ for the weight $-\alpha$ with $\sfs'_0<\sfs'$.

  Both estimates hold in the sense that if the quantities on the right hand side are finite, then so is the left hand side, and the inequality is valid.
\end{prop}
\begin{proof}
  The arguments are very similar to the ones used in \cite[\S2.1]{HintzVasySemilinear}. The proof relies on standard energy estimates near the artificial hypersurfaces, various microlocal propagation estimates, and crucially relies on the description of the null-bicharacteristic flow given in Proposition~\ref{PropRNdSFlow}.
  
  Let $u\in\Hbfw^{\sfs_0,\alpha}(\Omega)$ be such that $f=\cP u\in\Hbfw^{\sfs-1,\alpha}(\Omega)$. First of all, we can extend $f$ to $\wt f\in\Hb^{\sfs-1,\alpha}(M)$, with $\wt f$ supported in $r\geq r_0-2\delta$, $t_*\geq 0$ still, and $\|f\|_{\Hbfw^{\sfs-1,\alpha}(\Omega)}=\|\wt f\|_{\Hb^{\sfs-1,\alpha}(M)}$. Near $H_I$, we can then use the unique solvability of the forward problem for the wave equation $\Box\wt u=\wt f$ to obtain an estimate for $u$ there: indeed, using an approximation argument, approximating $\wt f$ by smooth functions $\wt f_\eps$, and using the propagation of singularities, propagating $H^\sfs$-regularity from $t_*<0$ (where the forward solution $\wt u_\eps$ of $\Box\wt u_\eps=\wt f_\eps$ vanishes), which can be done on this regularity scale uniformly in $\eps$, we obtain an estimate
  \[
    \|u\|_{H^\sfs(\Omega\cap\{0\leq t_*\leq 1\})} \leq C\|f\|_{H^{\sfs-1}(\Omega\cap\{0\leq t_*\leq 2\})},
  \]
  since $u$ agrees with $\wt u$ in the domain of dependence of $H_I$. The same argument shows that we can control the $H^{\sfs,\alpha}$-norm of $u$ in a neighborhood of $H_{I,0}$, say in $r<r_1-\delta$, in terms of $\|f\|_{\Hbfw^{\sfs-1,\alpha}(\Omega\cap\{r<r_1-\delta/2\})}$.

  Then, in $r>r_2-2\delta$, we use the propagation of singularities (forwards in $\Sigma_-$, backwards in $\Sigma_+$) to obtain local $H^\sfs$-regularity away from the boundary at infinity, $\tau=0$. At the radial sets $L_2$ and $L_3$, the radial point estimate, Proposition~\ref{PropRNdSRadialRecall}, allows us, using the a priori $\Hb^{\sfs_0,\alpha}$-regularity of $u$, to propagate $\Hb^{\sfs,\alpha}$-regularity into $L_2\cup L_3$; propagation within $\Sb^*_Y M$ then shows that we have $\Hb^{\sfs,\alpha}$-control on $u$ on $(\Gamma^-_-\cup\Gamma^+_-)\setminus\Gamma$. Since $\alpha<0$, we can then use \cite[Theorem~3.2]{HintzVasyNormHyp} to control $u$ in $\Hb^{\sfs,\alpha}$ microlocally at $\Gamma$ and propagate this control along $\Gamma^-_+\cup\Gamma^+_+$. Near $H_{F,3}$, the microlocal propagation of singularities only gives local control away from $H_{F,3}$, but we can get uniform regularity up to $H_{F,3}$ by standard energy estimates, using a cutoff near $H_{F,3}$ and the propagation of singularities for an extended problem (solving the forward wave equation with forcing $\wt f$, cut off near $H_{F,3}$, plus an error term coming from the cutoff), see \cite[Proposition~2.13]{HintzVasySemilinear} and the similar discussion around \eqref{EqRNdSWaveCutoff} below in the present proof. We thus obtain an estimate for the $\Hb^{\sfs,\alpha}$-norm of $u$ in $r\geq r_2-2\delta$.

  Next, we propagate regularity in $r_1<r<r_2$, using part \itref{ItRNdSFlowHyp} of Proposition~\ref{PropRNdSFlow} and our assumption $\sfs'\leq 0$; the only technical issue is now at $H_{F,2}$, where the microlocal propagation only gives local regularity away from $H_{F,2}$; this will be resolved shortly.

  Focusing on the remaining region $r_0-2\delta\leq r\leq r_1+2\delta$, we start with the control on $u$ near $H_{I,0}$, which we propagate forwards in $\Sigma_-$ and backwards in $\Sigma_+$, either up to $H_F$ or into the complex absorption hiding $L_0\cup\Gamma_*$; see \cite[\S2]{VasyMicroKerrdS} for the propagation of singularities with complex absorption. (This is a purely symbolic argument, hence the present b-setting is handled in exactly the same way as the standard \psdo{}\ setting discussed in the reference.) Moreover, at the elliptic set of the complex absorbing operator $\cQ$, we get $\Hb^{\sfs+1,\alpha}$-control on $u$, and we can propagate $\Hb^{\sfs,\alpha}$-estimates from there. The result is that we get $\Hb^{\sfs,\alpha}$-estimates of $u$ in a punctured neighborhood of $L_1$ within $\Sb^*_Y M$; thus, the low regularity part of Proposition~\ref{PropRNdSRadialRecall} applies. We can then propagate regularity from a neighborhood $L_1$ along $\cL_1$. This gives us local regularity away from $H_F\cup H_{F,2}$, where the microlocal propagation results do not directly give uniform estimates.
  
  In order to obtain uniform regularity up to $H_F\cup H_{F,2}$, we use the aforementioned cutoff argument for an extended problem near $H_F\cup H_{F,2}$: choose $\chi\in\CI(\Omega)$ such that $\chi\equiv 1$ for $r_0-\delta/2<r<r_2+\delta/2$, $t_*<t_{*,0}+1/2$, and such that $\chi\equiv 0$ if $r<r_1-\delta$ or $r>r_2+\delta$ or $t_*>t_{*,0}+1$; see Figure~\ref{FigRNdSCutoff} for an illustration. In particular, $[\chi,\cQ]=0$ by the support properties of $\cQ$. Therefore, we have
  \begin{equation}
  \label{EqRNdSWaveCutoff}
    \Box_g u'=f':=\chi f+[\Box_g,\chi]u,\quad u':=\chi u;
  \end{equation}
  note that we have (uniform) $\Hb^{\sfs-1,\alpha}$-control on $[\Box_g,\chi]u$ by the support properties of $d\chi$. Extend $f'$ beyond $H_F\cup H_{F,2}$ to $\wt f'\in\Hb^{\sfs-1,\alpha}$ with support in $r_1-\delta<r<r_2+\delta$ so that the global norm of $\wt f'$ is bounded by a fixed constant times the quotient norm of $f'$. The solution of the equation $\Box_g\wt u'=\wt f'$ with support of $\wt u'$ in $t_*<t_{*,0}+2$ is unique (it is simply the forward solution, taking into account the time orientation in the artificial exterior region); but then the local regularity estimates for $\wt u'$ for the extended problem, which follow from the propagation of singularities (using the approximation argument sketched above), give by restriction uniform regularity of $u$ up to $H_F\cup H_{F,2}$.

  \begin{figure}[!ht]
    \centering
    \includegraphics{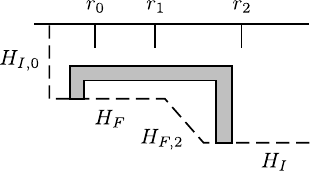}
    \caption{Illustration of the argument giving uniform regularity up to $H_F\cup H_{F,2}$: the cutoff $\chi$ is supported in and below the shaded region; the shaded region itself, containing $\supp d\chi$, is where we have already established $H^{\sfs}$-bounds for $u$.}
  \label{FigRNdSCutoff}
  \end{figure}
  
  Putting all these estimates together, we obtain an estimate for $u\in\Hbfw^{\sfs,\alpha}(\Omega)$ in terms of $f\in\Hbfw^{\sfs-1,\alpha}(\Omega)$.
  
  The proof of the dual estimate is completely analogous: we now obtain initial regularity (that we can then propagate as above) by solving the backward problem for $\Box_g$ near $H_F\cup H_{F,2}$ and $H_{F,3}$.
\end{proof}

\subsection{Fredholm analysis and solvability}
\label{SubsecRNdSFredholm}

The estimates in Proposition~\ref{PropRNdSGlobalReg} do not yet yield the Fredholm property of $\cP$. As explained in \cite[\S2]{HintzVasySemilinear}, we therefore study the \emph{Mellin-transformed normal operator family} $\wh\cP(\sigma)$, see \cite[\S5.2]{MelroseAPS}, which in the present (dilation-invariant in $\tau$, or translation-invariant in $t_*$) setting is simply obtained by conjugating $\cP$ by the Mellin transform in $\tau$, or equivalently the Fourier transform in $-t_*$, i.e.\ $\wh\cP(\sigma)=e^{i\sigma t_*}\cP e^{-i\sigma t_*}$, acting on functions on the boundary at infinity $\{\tau=0\}$. Concretely, we need to show that $\wh\cP(\sigma)$ is invertible between suitable function spaces on $\Im\sigma=-\alpha$ for a weight $\alpha<0$, since this will allow us to improve the $\Hbfw^{\sfs_0,\alpha}$ error term in \eqref{EqRNdSGlobalRegFw} by a space with an improved weight, so $\Hbfw^{\sfs,\alpha}$ injects compactly into it; an analogous procedure for the dual problem gives the full Fredholm property for $\cP$; see \cite{HintzVasySemilinear} and below for details. As in the previous section, only dynamical and geometric properties of the metric $g$ and the operator $\cP$ are used here; in fact, only their properties \emph{at infinity} matter for the analysis of $\wh\cP(\sigma)$, which is in general defined by conjugating the normal operator $N(\cP)$, obtained by freezing coefficients of $\cP$ at the boundary $X$ (i.e.\ $N(\cP)$ can be thought of as the stationary part of $\cP$), by the Mellin transform in $\tau$.

For any finite value of $\sigma$, we can analyze the operator $\wh\cP(\sigma)\in\Diff^2(X)$, $X=\pa M$, using standard microlocal analysis (and energy estimates near $H_{I,0}\cap X$ and $H_{F,3}\cap X$). The natural function spaces are variable order Sobolev spaces
\begin{equation}
\label{EqRNdSBdySpace}
  \Hfw^\sfs(Y),
\end{equation}
which we define to be the restrictions to $Y=\Omega\cap\pa M$ of elements of $H^\sfs(X)$ with support in $r\geq r_0-2\delta$, and dually on $\Hbw^\sfs(Y)$, the restrictions to $Y$ of elements of $H^\sfs(X)$ with support in $r\leq r_3+2\delta$, obtaining Fredholm mapping properties between suitable function spaces. However, in order to obtain useful estimates for our global b-problem, we need uniform estimates for $\wh\cP(\sigma)$ as $|\Re\sigma|\to\infty$ in strips of bounded $\Im\sigma$, on function spaces which are related to the variable order b-Sobolev spaces on which we analyze $\cP$.

Thus, let $h=\la\sigma\ra^{-1}$, $z=h\sigma$, and consider the semiclassical rescaling \cite[\S2]{VasyMicroKerrdS}
\begin{equation}
\label{EqRNdSSemiclassicalRescaling}
  \cP_{h,z} := h^2\wh\cP(h^{-1}z) \in \Diffh^2(X).
\end{equation}
We refer to \cite[\S4]{HintzVasyQuasilinearKdS} for details on the relationship between the b-operator $\cP$ and its semiclassical rescaling; in particular, we recall that the Hamilton vector field of the semiclassical principal symbol of $\cP_{h,z}$ for $z=\pm 1+\cO(h)$ is naturally identified with the Hamilton vector field of the b-principal symbol of $\cP$ restricted to $\{\sigma=\pm 1\}\subset\Tb^*_X M$, where we use the coordinates \eqref{EqRNdSTbCoords} in the b-cotangent bundle. For any Sobolev order function $\sfs\in\CI(X)$ and a weight $\alpha\in\R$, the Mellin transform in $\tau$ gives an isomorphism
\begin{equation}
\label{EqRNdSBSobolevMellin}
  \Hb^{\sfs,\alpha}([0,\infty)_\tau\times X)  \cong  L^2(\{\Im\sigma=-\alpha\}; H_{\la\sigma\ra^{-1}}^{\sfs,-\sfs}(X)),
\end{equation}
where $H_h^{\sfs,\sfw}(X)=h^{-\sfw}H_h^\sfs(X)$ (for $\sfw\in\CI(X)$) is a semiclassical variable order Sobolev space with a non-constant weighting in $h$; see Appendix~\ref{SecVariable} for definitions and properties of such spaces.

The analysis of $\cP_{h,z}$, $\Im z=\cO(h)$, acting on $H_h^{\sfs,-\sfs}(X)$-type spaces is now straightforward, given the properties of the Hamilton flow of $\cP$. Indeed, in view of the supported/extendible nature of the b-spaces $\Hbfw$ and $\Hbbw$ into account, we are led to define the corresponding semiclassical space
\[
  H_{h,\fw}^\sfs(Y),\tn{ resp. }H_{h,\bw}^\sfs(Y),
\]
to be the space of restrictions to $Y$ of elements of $H_h^{\sfs,-\sfs}(X)$ with support in $r\geq r_0-2\delta$, resp.\ $r\leq r_3-2\delta$. Then, in the region where $\sfs$ is not constant (recall that this is a subset of $\{r_1<r<r_2\}$), $\cP_{h,z}$ is a (semiclassical) real principal type operator, as follows from \eqref{EqRNdSHamRHypRegion}, and hence the only microlocal estimates we need there are elliptic regularity and the real principal type propagation for variable order semiclassical Sobolev spaces; these estimates are proved in Propositions~\ref{PropVariableSclElliptic} and \ref{PropVariableSclPropagation}. The more delicate estimates take place in standard semiclassical function spaces; these are the radial point estimates near $r=r_j$, in the present context proved in \cite[\S2]{VasyMicroKerrdS}, and the semiclassical estimates of Wunsch--Zworski \cite{WunschZworskiNormHypResolvent} and Dyatlov \cite{DyatlovSpectralGaps} (microlocalized in \cite[\S4]{HintzVasyQuasilinearKdS}) at the normally hyperbolic trapping. Near the artificial hypersurfaces $H_{I,0}\cup H_{F,3}$, intersected with $\pa M$, the operator $\cP_{h,z}$ is a (semiclassical) wave operator, and we use standard energy estimates there similar to the proof of Proposition~\ref{PropRNdSGlobalReg}, but keeping track of powers of $h$; see \cite[\S3]{VasyMicroKerrdS} for details.

We thus obtain:
\begin{prop}
\label{PropRNdSSemiclassical}
  Let $0<c_1<c_2$. Then for $h>0$ and $h^{-1}\Im z\in[c_1,c_2]$, we have the estimate
  \begin{equation}
  \label{EqRNdSSemiclassicalFw}
    \|v\|_{H_{h,\fw}^\sfs(Y)} \leq C(\|h^{-2}\cP_{h,z}v\|_{H_{h,\fw}^{\sfs-1}(Y)} + \|v\|_{H_{h,\fw}^{\sfs_0}(Y)})
  \end{equation}
  with a uniform constant $C$; here $\sfs$ and $\sfs_0<\sfs$ are forward order functions for all weights in $[-c_2,-c_1]$, see Definition~\ref{DefRNdSOrderFunctions}. For the dual problem, we similarly have
  \begin{equation}
  \label{EqRNdSSemiclassicalBw}
    \|v\|_{H_{h,\bw}^\sfs(Y)} \leq C(\|h^{-2}\cP_{h,z}^*v\|_{H_{h,\bw}^{\sfs-1}(Y)} + \|v\|_{H_{h,\bw}^{\sfs_0}(Y)}),
  \end{equation}
  where $\sfs$ and $\sfs_0<\sfs$ are backward order functions for all weights in $[c_1,c_2]$.
\end{prop}

Notice here that if $\sfs$ were constant, the estimate \eqref{EqRNdSSemiclassicalFw} would read $\|v\|_{H_h^\sfs}\leq h^{-1}\|\cP_{h,z}v\|_{H_h^{\sfs-1}}+h^{\sfs-\sfs_0}\|v\|_{H_h^{\sfs_0}}$, which is the usual hyperbolic loss of one derivative and one power of $h$. The estimate \eqref{EqRNdSSemiclassicalFw} is conceptually the same, but in addition takes care of the variable orders. Trapping causes no additional losses here, since $h^{-1}\Im z>0$.

\begin{rmk}
\label{RmkRNdSSemiclassicalDual}
  We have $h^{-2}\cP_{h,z}=\wh\cP(\sigma)$, and $h^{-2}\cP_{h,z}^*=\wh{\cP^*}(\ol\sigma)$; the change of sign in $\Im\sigma=h^{-1}\Im z$ when going from \eqref{EqRNdSSemiclassicalFw} to the dual estimate \eqref{EqRNdSSemiclassicalBw} is analogous to the change of sign in the weight $\alpha$ in Proposition~\ref{PropRNdSGlobalReg}.
\end{rmk}

For future reference, we note that we still have high energy estimates for $h^{-1}z$ in strips including and extending \emph{below} the real line: the only delicate part is the estimate at the normally hyperbolic trapping, more precisely at the semiclassical trapped set $\Gamma_{h,z}$, which can be naturally identified with the intersection of the trapped set $\Gamma$ with $\{\sigma=\pm 1\}$ for $z=\pm 1+\cO(h)$. Thus, let $\numin$ be the minimal expansion rate at the semiclassical trapped set in the normal direction as in \cite{DyatlovSpectralGaps} or \cite[\S5]{DyatlovResonanceProjectors}; let us then write
\[
  \sigma_h\Bigl(\frac{1}{2ih}(\cP_{h,z}-\cP_{h,z}^*)\Bigr)=\pm\frac{\numin}{2\gamma_0} h^{-1}\Im z,\quad z=\pm 1+\cO(h),
\]
for some real number $\gamma_0>0$ (in the Kerr--de Sitter case discussed later, $\gamma_0$ is a smooth function on $\Gamma_{h,z}$); see \S\ref{SubsecRNdSHighReg}, in particular \eqref{EqRNdSHighRegNumin} and \eqref{EqRNdSHighRegSubpr}, for the ingredients for the calculation of $\gamma_0$ in a limiting case. Therefore, if $h^{-1}\Im z>-\gamma_0$, then $\pm\sigma_h((2ih)^{-1}(\cP_{h,z}-\cP_{h,z}^*))>-\numin/2$. The reason for the `$\pm$' appearing here is the following: for the `$-$' case, note that for $z=-1+\cO(h)$, corresponding to semiclassical analysis in $\{\sigma=-1\}$, which near the trapped set $\Gamma$ intersects the \emph{forward} light cone $\Sigma_-$ non-trivially, we propagate regularity \emph{forwards along the Hamilton flow}, while in the `$+$' case, corresponding to propagation in the \emph{backward} light cone $\Sigma_+$, we propagate \emph{backwards along the flow}. Using \cite{DyatlovSpectralGaps}, see also the discussion in \cite[\S4.4]{HintzVasyQuasilinearKdS}, we conclude:

\begin{prop}
\label{PropRNdSSemiclassical2}
  Using the above notation, the (uniform) estimates \eqref{EqRNdSSemiclassicalFw} and \eqref{EqRNdSSemiclassicalBw} hold with $\sfs-1$ replaced by $\sfs$ on the right hand sides, provided $h^{-1}\Im z\in[-\gamma,c_2]$, where $-\gamma_0<-\gamma<c_2$.
\end{prop}

The effect of replacing $\sfs-1$ by $\sfs$ is that this adds an additional $h^{-1}$ to the right hand side, i.e.\ we get a weaker estimate (which in the presence of trapping cannot be avoided by \cite{BonyBurqRamondTrapping}); the strengthening of the norm in the regularity sense is unnecessary, but does not affect our arguments later.

We return to the case $h^{-1}\Im z>0$. If we define the space $\cX_h^\sfs=\{u\in H_{h,\fw}^{\sfs}(Y)\colon h^{-2}\cP_{h,z}\in H_{h,\fw}^{\sfs-1}(Y)\}$, then the estimates in Proposition~\ref{PropRNdSSemiclassical} imply that the map
\begin{equation}
\label{EqRNdSNormalOpMap}
  \wh\cP(\sigma) \colon \cX^\sfs_{\la\sigma\ra^{-1}} \to H_{\la\sigma\ra^{-1},\fw}^{\sfs-1}(Y)
\end{equation}
is Fredholm for $\Im\sigma\in[c_1,c_2]$, with high energy estimates as $|\Re\sigma|\to\infty$. Moreover, for small $h>0$, the error terms on the right hand sides of \eqref{EqRNdSSemiclassicalFw} and \eqref{EqRNdSSemiclassicalBw} can be absorbed into the left hand sides, hence in this case we obtain the invertibility of the map \eqref{EqRNdSNormalOpMap}. This implies that $\wh\cP(\sigma)$ is invertible for $\Im\sigma\in[c_1,c_2]$, $|\Re\sigma|\gg 0$. Since therefore there are only finitely many resonances (poles of $\wh\cP(\sigma)^{-1}$) in $c_1\leq\Im\sigma\leq c_2$ for any $0<c_1<c_2$, we may therefore pick a weight $\alpha<0$ such that there are no resonances on the line $\Im\sigma=-\alpha$, which in view of \eqref{EqRNdSBSobolevMellin} implies the estimate
\begin{equation}
\label{EqRNdSNormalOpInv}
  \|u\|_{\Hbfw^{\sfs,\alpha}(\Omega_I)} \leq C\|N(\cP)u\|_{\Hbfw^{\sfs-1,\alpha}(\Omega_I)},
\end{equation}
where $\Omega_I=[0,\infty)_\tau\times Y$ is the manifold on which the dilation-invariant operator $N(\cP)$ naturally lives; here $\sfs$ is a forward order function for the weight $\alpha$, and the subscript `fw' on the b-Sobolev spaces denotes distributions with supported character at $[0,\infty)_\tau\times(H_{I,0}\cap Y)$ and extendible at $[0,\infty)_\tau\times(H_{F,3}\cap Y)$. We point out that the choice $\tau$ of boundary defining function and the choice of $\tau$-dilation orbits fixes an isomorphism of a collar neighborhood of $Y$ in $\Omega$ with a neighborhood of $\{0\}\times Y$ in $\Omega_I$, and the two $\Hb^{\sfs,\alpha}$-norms on functions supported in this neighborhood, given by the restriction of the $\Hbfw^{\sfs,\alpha}(\Omega)$-norm and the restriction of the $\Hbfw^{\sfs,\alpha}(\Omega_I)$-norm, respectively, are equivalent.

Equipped with \eqref{EqRNdSNormalOpInv}, we can now improve Proposition~\ref{PropRNdSGlobalReg} to obtain the Fredholm property of $\cP$: first, we let $\sfs$ be a forward order function for the weight $\alpha$, but with the more stringent requirement
\begin{equation}
\label{EqRNdSFredholmFwOrder}
  \sfs(r)>3/2+\max(\beta_2\alpha,\beta_3\alpha),\quad r>r_1+2\delta',
\end{equation}
and we require that the forward order function $\sfs_0$ satisfies $\sfs_0<\sfs-1$. Using \eqref{EqRNdSNormalOpInv} with $\sfs$ replaced by $\sfs_0$, and a cutoff $\chi\in\CI(\Omega)$, identically $1$ near $Y$ and supported in a small collar neighborhood of $Y$, the estimate \eqref{EqRNdSGlobalRegFw} then implies (as in \cite[\S2]{HintzVasySemilinear})
\begin{align*}
  \|u\|_{\Hbfw^{\sfs,\alpha}} &\lesssim \|\cP u\|_{\Hbfw^{\sfs-1,\alpha}} + \|(1-\chi)u\|_{\Hbfw^{\sfs_0,\alpha}} + \|N(\cP)\chi u\|_{\Hbfw^{\sfs_0-1,\alpha}} \\
    &\lesssim \|\cP u\|_{\Hbfw^{\sfs-1,\alpha}} + \|u\|_{\Hbfw^{\sfs_0,\alpha+1}} + \|\chi\cP u\|_{\Hbfw^{\sfs_0-1,\alpha}} + \|[N(\cP),\chi]u\|_{\Hbfw^{\sfs_0-1,\alpha}} \\
    &\qquad\qquad + \|\chi(\cP-N(\cP)) u\|_{\Hbfw^{\sfs_0-1,\alpha}}.
\end{align*}
Noting that $[N(\cP),\chi]\in\tau\Diffb^1$, the second to last term can be estimated by $\|u\|_{\Hbfw^{\sfs_0,\alpha+1}}$, while the last term can be estimated by $\|u\|_{\Hbfw^{\sfs_0+1,\alpha+1}}$; thus, we obtain
\begin{equation}
\label{EqRNdSFredholmFw}
  \|u\|_{\Hbfw^{\sfs,\alpha}} \lesssim \|\cP u\|_{\Hbfw^{\sfs-1,\alpha}} + \|u\|_{\Hbfw^{\sfs_0+1,\alpha+1}},
\end{equation}
where the inclusion $\Hbfw^{\sfs,\alpha}\hookrightarrow\Hbfw^{\sfs_0+1,\alpha+1}$ is now compact. This estimate implies that $\ker\cP$ is finite-dimensional and $\ran\cP$ is closed. The dual estimate is
\begin{equation}
\label{EqRNdSFredholmBw}
  \|u\|_{\Hbbw^{\sfs',-\alpha}} \lesssim \|\cP^* u\|_{\Hbbw^{\sfs'-1,-\alpha}} + \|u\|_{\Hbbw^{\sfs_0'+1,-\alpha+1}},
\end{equation}
where now $\sfs_0'<\sfs'-1$ is a backward order function for the weight $-\alpha$, and the backward order function $\sfs'$ satisfies the more stringent bound
\[
  \sfs'(r)>3/2-\beta_1\alpha,\quad r<r_1+\delta'.
\]
Note that $\wh\cP(\sigma)^{-1}$ not having a pole on the line $\Im\sigma=-\alpha$ is equivalent to $\wh{\cP^*}(\sigma)^{-1}$ not having a pole on the line $\Im\sigma=\alpha$, since $\wh{\cP^*}(\sigma)=\wh\cP(\ol\sigma)^*$. We wish to take $\sfs'=1-\sfs$ with $\sfs$ as in the estimate \eqref{EqRNdSFredholmFw}; so if we require in addition to \eqref{EqRNdSFredholmFwOrder} that
\begin{equation}
\label{EqRNdSFredholmFwOrder2}
  \sfs(r)<-1/2+\beta_1\alpha, \quad r<r_1+\delta',
\end{equation}
the estimates \eqref{EqRNdSFredholmFw} and \eqref{EqRNdSFredholmBw} for $\sfs'=1-\sfs$ imply by a standard functional analytic argument, see e.g.\ \cite[Proof of Theorem~26.1.7]{HormanderAnalysisPDE4}, that
\begin{equation}
\label{EqRNdSFredholmMap}
  \cP \colon \cX^{\sfs,\alpha} \to \Hbfw^{\sfs-1,\alpha}(\Omega)
\end{equation}
is Fredholm, where
\begin{equation}
\label{EqRNdSFredholmFwSpace}
  \cX^{\sfs,\alpha} = \{ u\in\Hbfw^{\sfs,\alpha}(\Omega)\colon \cP u\in\Hbfw^{\sfs-1,\alpha}(\Omega)\},
\end{equation}
and the range of $\cP$ is the annihilator of the kernel of $\cP^*$ acting on $\Hbbw^{1-\sfs,-\alpha}(\Omega)$. We can strengthen the regularity at the Cauchy horizon by dropping \eqref{EqRNdSFredholmFwOrder2}, cf.\ \cite[\S5]{HintzVasySemilinear}:
\begin{thm}
\label{ThmRNdSFredholm}
  Suppose $\alpha<0$ is such that $\cP$ has no resonances on the line $\Im\sigma=-\alpha$. Let $\sfs$ be a forward order function for the weight $\alpha$, and assume \eqref{EqRNdSFredholmFwOrder} holds. Then the map $\cP$, defined in \eqref{EqRNdSOperator}, is Fredholm as a map \eqref{EqRNdSFredholmMap}, with range equal to the annihilator of $\ker_{\Hbbw^{1-\sfs,-\alpha}(\Omega)}\cP^*$.
\end{thm}
\begin{proof}
  Let $\wt\sfs\leq\sfs$ be an order function satisfying both \eqref{EqRNdSFredholmFwOrder} and \eqref{EqRNdSFredholmFwOrder2}, so by the above discussion, $\cP\colon\cX^{\wt\sfs,\alpha}\to\Hbfw^{\wt\sfs-1,\alpha}(\Omega)$ is Fredholm. Since $\cX^{\sfs,\alpha}\subset\cX^{\wt\sfs,\alpha}$, we a forteriori get the finite-dimensionality of $\dim\ker_{\cX^{\sfs,\alpha}}\cP$. On the other hand, if $v\in\Hbfw^{\sfs-1,\alpha}(\Omega)$ annihilates $\ker_{\Hbbw^{1-\sfs,-\alpha}(\Omega)}\cP^*$, it also annihilates $\ker_{\Hbbw^{1-\wt\sfs,-\alpha}(\Omega)}\cP^*$, hence we can find $u\in\Hbfw^{\wt\sfs,\alpha}(\Omega)$ solving $\cP u=v$. The propagation of singularities, Proposition~\ref{PropRNdSGlobalReg}, implies $u\in\Hbfw^{\sfs,\alpha}(\Omega)$, and the proof is complete.
\end{proof}

To obtain a better result, we need to study the structure of resonances. Notice that for the purpose of dealing with a single resonance, one can simplify notation by working with the space $\Hfw^\sfs$, see \eqref{EqRNdSBdySpace}, rather than $H_{h,\fw}^{\sfs,-\sfs}$, since the semiclassical (high energy) parameter is irrelevant then.

\begin{lemma}
\label{LemmaRNdSResonanceSupp}
\begin{enumerate}
  \item \label{ItRNdSResonanceSupp} Every resonant state $u\in\ker_{\Hfw^\sfs}\wh\cP(\sigma)$ corresponding to a resonance $\sigma$ with $\Im\sigma>0$ is supported in the artificial exterior region $\{r_0\leq r\leq r_1\}$; more precisely, every element in the range of the singular part of the Laurent series expansion of $\wh\cP(\sigma)^{-1}$ at such a resonance $\sigma$ is supported in $\{r_0\leq r\leq r_1\}$. In fact, this holds more generally for any $\sigma\in\C$ which is not a resonance of the forward problem for the wave equation in a neighborhood $\Omega_{23}$ of the black hole exterior.
  \item \label{ItRNdSResonanceSupp0Res} If  $R$ denotes the restriction of distributions on $Y$ to $r>r_1$, then the only pole of $R\circ\wh\cP(\sigma)^{-1}$ with $\Im\sigma\geq 0$ is at $\sigma=0$, has rank $1$, and the space of resonant states consists of constant functions.
\end{enumerate}
\end{lemma}
\begin{proof}
  Since $u$ has supported character at $H_{I,0}\cap Y$, we obtain $u\equiv 0$ in $r<r_0$, since $u$ solves the \emph{wave} equation $\wh\cP(\sigma)u=0$ there. On the other hand, the forward problem for the wave equation in the neighborhood $\Omega_{23}$ of the black hole exterior does not have any resonances with positive imaginary part; this is well-known for Schwarzschild--de Sitter spacetimes \cite{SaBarretoZworskiResonances,BonyHaefnerDecay} and for slowly rotating Kerr--de Sitter spacetimes, either by direct computation \cite{DyatlovQNM} or by a perturbation argument \cite{HintzVasyKdsFormResonances,VasyMicroKerrdS}. For the convenience of the reader, we recall the argument for the Schwarzschild--de Sitter case, which applies without change in the present setting as well: a simple integration by parts argument, see e.g.\ \cite{DyatlovQNM} or \cite[\S2]{HintzVasyKdsFormResonances}, shows that $u$ must vanish in $r_2<r<r_3$. Now the propagation of singularities at radial points implies that $u$ is smooth at $r=r_2$ and $r=r_3$ (where the a priori regularity exceeds the threshold value), and hence in $r>r_3$, $u$ is a solution to the homogeneous wave equation on an asymptotically de Sitter space which decays rapidly at the conformal boundary (which is $r=r_3$), hence must vanish identically in $r>r_3$ (see \cite[Footnote~58]{VasyMicroKerrdS} for details); the same argument applies in $r_1<r<r_2$, yielding $u\equiv 0$ there. Therefore, $\supp u\subset\{r_0\leq r\leq r_1\}$, as claimed. An iterative argument, similar to \cite[Proof of Lemma~8.3]{BaskinVasyWunschRadMink}, yields the more precise result.

  The more general statement follows along the same lines (and is in fact much easier to prove, since it does not entail a mode stability statement): suppose $\sigma$ is not a resonance of the forward wave equation on $\Omega_{23}$, then a resonant state $u\in\ker\wh\cP(\sigma)$ must vanish in $\Omega_{23}$, and we obtain $\supp u\subset\{r_0\leq r\leq r_1\}$ as before; likewise for the more precise result. This proves \itref{ItRNdSResonanceSupp}.

  For the proof of \itref{ItRNdSResonanceSupp0Res}, it remains to study the resonance at $0$, since the only $\Omega_{23}$ resonance in the closed upper half plane is $0$. Note that an  element in the range of the most singular Laurent coefficient of $R\circ\wh\cP(\sigma)^{-1}$ at $\sigma=0$ lies in $\ker\wh\cP(0)$; but elements in $\ker\wh\cP(0)$ which vanish near $r=r_1$ vanish identically in $r>r_1$ and hence are annihilated by $R$, while elements which are not identically $0$ near $r=r_1$ are not identically $0$ in $r>r_2$ as well; but the only non-trivial elements of $\ker\wh\cP(0)$ (which are smooth at $r_2$ and $r_3$) are constant in $r_2<r<r_3$, and since $\wh\cP(0)1=0$ in $r>r_1$, we deduce (by unique continuation) that $R(\ker\wh\cP(0))$ indeed consists of constant functions. But then the order of the pole of $R\circ\wh\cP(\sigma)^{-1}$ at $\sigma=0$ equals the order of the $0$-resonance of the forward problem for $\Box_g$ in $\Omega_{23}$, which is known to be equal to $1$, see the references above. The $1$-dimensionality of $R(\ker\wh\cP(0))$ then implies that the rank of the pole of $r\circ\wh\cP(\sigma)^{-1}$ at $0$ indeed equals $1$.
\end{proof}

Since we are dealing with an extended global problem here, involving (pseudodifferential!) complex absorption, solvability is not automatic, but it holds in the region of interest $r>r_1$; to show this, we first need:
\begin{lemma}
\label{LemmaRNdSSolvability}
  Recall the definition of the set $\cU\subset\Omega$, where the complex absorption is placed, from \eqref{EqRNdSComplexAbsorptionSupp}. Under the assumptions of Theorem~\ref{ThmRNdSFredholm} (in particular, $\alpha<0$), there exists a linear map
  \[
    E_\cQ\colon\Hbfw^{\sfs-1,\alpha}(\Omega)\to\CIc(\cU^\circ)
  \]
  such that for all $f\in\Hbfw^{\sfs-1,\alpha}$, the function $f+E_\cQ f$ lies in the range of the map $\cP$ in \eqref{EqRNdSFredholmMap}.
\end{lemma}
\begin{proof}
  By Theorem~\ref{ThmRNdSFredholm}, the statement of the lemma is equivalent to
  \[
    \ran_{\Hbfw^{\sfs-1,\alpha}}(\id+E_\cQ) \perp \ker_{\Hbbw^{1-\sfs,-\alpha}}\cP^*.
  \]
  Let $v\in\Hbbw^{1-\sfs,-\alpha}$, $\cP^*v=0$. We claim that $v|_{\cU^\circ}=0$ implies $v\equiv 0$; in other words, elements of $\ker\cP^*$ are uniquely determined by their restriction to $\cU^\circ$. To see this, note that $v=0$ on $\cU^\circ$ implies that in fact $v$ solves the homogeneous wave equation $\Box_g v=0$. Thus, we conclude by the supported character of $v$ at $H_F\cup H_{F,2}$ and $H_{F,3}$ that $v$ in fact vanishes in $r<r_2$ and $r>r_3$, so $\supp v\subset\{r_2<r<r_3\}$. Using the high energy estimates \eqref{EqRNdSSemiclassicalBw}, a contour shifting argument, see \cite[Lemma~3.5]{VasyMicroKerrdS}, and the fact that resonances of $\cP$ with $\Im\sigma>0$ have support disjoint from $\{r_2\leq r\leq r_3\}$ by Lemma~\ref{LemmaRNdSResonanceSupp} \itref{ItRNdSResonanceSupp}, we conclude that in fact $v\in\Hbbw^{1-\sfs,\infty}$, i.e.\ $v$ vanishes to infinite order at future infinity; but then, radial point estimates and the simple version of propagation of singularities at the normally hyperbolic trapping (since we are considering the \emph{backwards} problem on \emph{decaying} spaces) --- see \cite[Theorem~3.2, estimate~(3.10)]{HintzVasyNormHyp} --- imply that in fact $v\in\Hb^{\infty,\infty}$. Now the energy estimate in \cite[Lemma~2.15]{HintzVasySemilinear} applies to $v$ and yields $\|v\|_{\Hb^{1,\wt r}}\lesssim \|\cP^* v\|_{\Hb^{0,\wt r}}=0$ for $\wt r\ll 0$, hence $v=0$ as claimed.

  Therefore, if $v_1,\ldots,v_N\in\Hbbw^{1-\sfs,-\alpha}(\Omega)$ forms a basis of $\ker\cP^*$, then the restrictions $v_1|_{\cU^\circ},\ldots,v_N|_{\cU^\circ}$ are linearly independent elements of $\sD'(\cU^\circ)$, and hence one can find $\phi_1,\ldots,\phi_N\in\CIc(\cU^\circ)$ with $\la v_i,\phi_j\ra=\la v_i|_{\cU^\circ},\phi_j\ra=\delta_{ij}$. The map
  \[
    E_\cQ f := -\sum_{i=1}^N \la f,v_i\ra\phi_i,\quad f\in\Hbfw^{\sfs_1,\alpha}(\Omega),
  \]
  then satisfies all requirements.
\end{proof}

We can then conclude:
\begin{cor}
  Under the assumptions of Theorem~\ref{ThmRNdSFredholm}, all elements in the kernel of $\cP$ in \eqref{EqRNdSFredholmMap} are supported in the artificial exterior domain $\{r_0\leq r\leq r_1\}$.
  
  Moreover, for all $f\in\Hbfw^{\sfs-1,\alpha}$ with support in $r>r_1$, there exists $u\in\Hbfw^{\sfs,\alpha}$ such that $\cP u=f$ in $r>r_1$.
\end{cor}
\begin{proof}
  If $u\in\Hbfw^{\sfs,\alpha}(\Omega)$ lies in $\ker\cP$, then the supported character of $u$ at $H_I\cup H_{I,0}$ together with uniqueness for the wave equation in $r<r_0$ and $r>r_1$ implies that $u$ vanishes identically there, giving the first statement.

  For the second statement, we use Lemma~\ref{LemmaRNdSSolvability} and solve the equation $\cP u=f+E_\cQ f$, which gives the desired $u$.
\end{proof}

In particular, solutions of the equation $\cP u=f$ exist and are unique in $r>r_1$, which we of course already knew from standard hyperbolic theory in the region on `our' side $r>r_1$ of the Cauchy horizon; the point is that we now understand the regularity of $u$ \emph{up to} the Cauchy horizon. We can refine this result substantially for better-behaved forcing terms, e.g.\ for $f\in\CIc(\Omega^\circ)$ with support in $r>r_1$; we will discuss this in the next two sections.

\subsection{Partial asymptotics and decay}
\label{SubsecRNdSAsymp}

The only resonance of the forward problem in $\Omega_{23}$ in $\Im\sigma\geq 0$ is a simple resonance at $\sigma=0$, with resonant states equal to constants, see the references given in the proof of Lemma~\ref{LemmaRNdSResonanceSupp}, and there exists $\alpha>0$ such that $0$ is the only resonance in $\Im\sigma\geq-\alpha$. (This does \emph{not} mean that the global problem for $\cP$ does not have other resonances in this half space!) In the notation of Proposition~\ref{PropRNdSSemiclassical2}, we may assume $-\alpha>-\gamma_0$ so that we have high energy estimates in $\Im\sigma\geq-\alpha$.

\begin{prop}
\label{PropRNdSPartialAsymp}
  Let $\alpha>0$ be as above. Suppose $u$ is the forward solution of
  \begin{equation}
  \label{EqRNdSPartialAsympEqn}
    \Box_g u=f\in\CIc(\Omega^\circ), \quad r>r_1.
  \end{equation}
  Then $u$ has a partial asymptotic expansion
  \begin{equation}
  \label{EqRNdSPartialAsymp}
    u=u_0\chi(\tau)+u',
  \end{equation}
  with $u_0\in\C$ and $\chi\equiv 1$ near $\tau=0$, $\chi\equiv 0$ away from $\tau=0$, and $u'$ is smooth in $r>r_1$, while $u'\in\Hb^{1/2+\alpha\beta_1-0,\alpha}$ near $r=r_1$.
\end{prop}

Our proof uses the stationarity of $g$ (and $\cP$) near $X$, which allows us to pass freely between $\cP$ and the Mellin-transformed normal operator family $\wh\cP(\sigma)$; see also Remark~\ref{RmkStat}.

\begin{proof}[Proof of Proposition~\ref{PropRNdSPartialAsymp}]
  Let $\wt\alpha<0$, and let $\sfs$ be a forward order function for the weight $\wt\alpha$. Using Lemma~\ref{LemmaRNdSSolvability}, we may assume that $\cP u_*=f$ is solvable with $u_*\in\Hbfw^{\sfs,\wt\alpha}(\Omega)$ by modifying $f$ in $r<r_1$ if necessary. In fact, by the propagation of singularities, Theorem~\ref{ThmRNdSFredholm}, we may take $\sfs$ to be arbitrarily large in compact subsets of $r>r_1$. Then, a standard contour shifting argument, using the high energy estimates for $\wh\cP(\sigma)$ in $\Im\sigma\geq-\alpha$, see \cite[Lemma~3.5]{VasyMicroKerrdS} or \cite[Theorem~2.21]{HintzVasySemilinear}, implies that $u_*$ has an asymptotic expansion as $\tau\to 0$
  \begin{equation}
  \label{EqRNdSPartialAsympProof}
    u_*(\tau,x) = \sum_j\sum_{\kappa=0}^{m_j} \tau^{i\sigma_j}|\log\tau|^\kappa a_{j\kappa}(x)\chi(\tau) + u',\quad (\tau,x)\in\Omega\subset[0,\infty)\times Y,
  \end{equation}
  where the $\sigma_j$ are the resonances of $\cP$ in $-\alpha<\Im\sigma<-\wt\alpha$, the $m_j$ are their multiplicities, and the $a_{j\kappa}\in\Hfw^\sfs(Y)$ are resonant states corresponding to the resonance $\sigma_j$; lastly, $u'\in\Hbfw^{\sfs,\alpha}(\Omega)$ is the remainder term of the expansion. Even though $\cP$ is dilation-invariant near $\tau=0$, this argument requires a bit of care due to the \emph{extendible} nature of $u$ at $H_{F,2}\cup H_F$: one needs to consider the cutoff equation $\cP(\chi u)=f_1:=\chi f+[\Box_g,\chi]u$; computing the inverse Mellin transform of $\wh\cP(\sigma)^{-1}\wh{f_1}(\sigma)$ generates the expansion \eqref{EqRNdSPartialAsympProof} by a contour shifting argument, see \cite[Lemma~3.1]{VasyMicroKerrdS}. Now $\cP$ annihilates the partial expansion, so $\cP u'=f_1$ on the set where $\chi\equiv 1$; by the propagation of singularities, Proposition~\ref{PropRNdSGlobalReg}, we can improve the regularity of $u'$ on this set to $u'\in\Hb^{1/2+\alpha\beta_1-0,\alpha}$.
  
  Thus, we have shown regularity in the region where $\chi\equiv 1$, i.e.\ where we did not cut off; however, considering \eqref{EqRNdSPartialAsympEqn} on an enlarged domain and running the argument there, with the cutoff $\chi$ supported in the enlarged domain and identically $1$ on $\Omega$, we obtain the full regularity result upon restricting to $\Omega$.
  
  Now, by Lemma~\ref{LemmaRNdSResonanceSupp} \itref{ItRNdSResonanceSupp}, all resonant states of $\cP$ which are not resonant states of the forward problem in $\Omega_{23}$ must in fact vanish in $r>r_1$, and by part \itref{ItRNdSResonanceSupp0Res} of Lemma~\ref{LemmaRNdSResonanceSupp}, the only term in \eqref{EqRNdSPartialAsympProof} that survives upon restriction to $r>r_1$ is the constant term.
\end{proof}

Thus, we obtain a partial expansion with a remainder which decays exponentially in $t_*$ in an $L^2$ sense; we will improve this in particular to $L^\infty$ decay in the next section.

\begin{rmk}
\label{RmkStat}
  If $g$ and $\cP$ were \emph{not} dilation-invariant, then in the partial expansion~\eqref{EqRNdSPartialAsympProof}, one would not be able to show improved regularity for $u'$ at $r=r_1$ in general because $r=r_1$ (or rather $\cL_1=\Nb^*\{r=r_1\}$) no longer has a geometric meaning as the stable/unstable manifold of the radial set $L_1$. (See also the setup leading to Proposition~\ref{PropBConormal} below.) Concretely, $\cP u'=f_1-(\cP-N(\cP))u_1$ has an error term that in general loses two derivatives, which cannot be recouped by Proposition~\ref{PropRNdSGlobalReg}. On the other hand, \emph{assuming} that $\cL_1$ is characteristic for $\cP$ (or choosing the dilation orbits of $\tau$ more carefully), and tracking the singular nature of the resonant states $a_{j\kappa}$ more precisely should allow for the above proposition to generalize to the non-dilation-invariant setting; however, we do not pursue this further here.
\end{rmk}

\subsection{Conormal regularity at the Cauchy horizon}
\label{SubsecRNdSConormal}

Suppose $u$ solves \eqref{EqRNdSPartialAsympEqn}, hence it has an expansion \eqref{EqRNdSPartialAsymp}. For any Killing vector field $V$, we then have $\Box_g(V u)=V f$; now if $\wt u$ solves the global problem $\cP\wt u=V f+E_\cQ V f$ (using the extension operator $E_\cQ$ from Lemma~\ref{LemmaRNdSSolvability}), then $\wt u=V u$ in $r>r_1$ by the uniqueness for the Cauchy problem in this region. But by Proposition~\ref{PropRNdSPartialAsymp}, $\wt u$ has an expansion like \eqref{EqRNdSPartialAsymp}, with constant term vanishing because $X$ annihilates the constant term in the expansion of $u$, and therefore $\wt u$ lies in space $\Hb^{1/2+\alpha\beta_1-0,\alpha}$ near the Cauchy horizon $\{r=r_1\}$ as well. More generally, we can take $V$ to be any (finite) product of Killing vector fields, and therefore obtain
\[
  u=u_0\chi(\tau)+u',\quad u_0\in\C, \quad V_1\cdots V_N u'\in\Hb^{1/2+\alpha\beta_1-0,\alpha},
\]
where $N=0,1,\ldots$ is arbitrary, and the vector fields $V_j$, $j=1,\ldots,N$, are equal to $\tau D_\tau$ or rotation vector fields on the $\Sph^2$-factor of the spacetime, independent of $\tau,r$. (This uses that $g$ is stationary!) These vector fields are all tangent to the Cauchy horizon. We obtain for any small open interval $I\subset\R$ containing $r_1$ that
\begin{equation}
\label{EqRNdSConormalKilling}
  u'|_I \in H^{1/2+\alpha\beta_1-0}\left(I; \tau^\alpha\Hb^N([0,\infty)_\tau\times\Sph^2)\right),\quad N=0,1,\ldots.
\end{equation}
A posteriori, by Sobolev embedding, this gives
\begin{cor}
\label{CorRNdSBoundedness}
  Using the notation of Proposition~\ref{PropRNdSPartialAsymp}, the solution $u$ of \eqref{EqRNdSPartialAsympEqn} has an asymptotic expansion $u=u_0\chi(\tau)+u'$ with $u_0\in\C$, and there exists a constant $C>0$ such that $|u'(\tau,x)|\leq C\tau^\alpha$. In particular, $u$ is uniformly bounded in $r>r_1$ and extends continuously to $\cCH^+$.
\end{cor}

Translated back to $t_*=-\log\tau$, the estimate on the remainder states that for scalar waves, one has exponentially fast pointwise decay to a constant. This recovers Franzen's boundedness result \cite{FranzenRNBoundedness} for linear scalar waves on the Reissner--Nordstr\"om spacetime near the Cauchy horizon in the cosmological setting.

The above argument is unsatisfactory in two ways: firstly, they are not robust and in particular do not quite apply in the Kerr--de Sitter setting discussed in \S\ref{SecKdS}; however, see Remark~\ref{RmkKdSResCarter}, which shows that using a `hidden symmetry' of Kerr--de Sitter space related to the completely integrable nature of the geodesic equation, one can still conclude boundedness in this case. Secondly, the regularity statement \eqref{EqRNdSConormalKilling} is somewhat unnatural from a PDE perspective; thus, we now give a more robust microlocal proof of the \emph{conormality} of $u'$, i.e.\ iterative regularity under application of vector fields tangent to $r=r_1$, which relies on the propagation of conormal regularity at the radial set $L_1$, see Proposition~\ref{PropBConormal}.

First however, we study conormal regularity properties of $\wh\cP(\sigma)$ for fixed $\sigma$, in particular giving results for individual resonant states. \emph{From now on, we work locally near $r=r_1$ and microlocally near $L_1=N^*\{r=r_1\}\subset T^*X$, and all pseudodifferential operators we consider implicitly have wavefront set localized near $N^*\{r=r_1\}$.} As in \S\ref{SubsecRNdSFlow}, we use the function $\tau_0=e^{-t_0}$ instead of $\tau$, where $t_0=t-F(r)$, $F'=s_1\mu^{-1}=\mu^{-1}$ near $r=r_1$, hence the dual metric function $G$ is given by \eqref{EqRNdSMetricWithTau0}. Since $\tau_0$ is a smooth non-zero multiple of $\tau$, this is inconsequential from the point of view of regularity, and it even is semiclassically harmless for $\Im z=\cO(h)$. Denote the conjugation of $\cP$ by the Mellin transform in $\tau_0$ by
\[
  \wt\cP(\sigma) = \tau_0^{-i\sigma}\cP\tau_0^{i\sigma}
\]
with $\sigma$ the Mellin-dual variable to $\tau$. We first study standard (non-semiclassical) conormality using techniques developed in \cite{HassellMelroseVasySymbolicOrderZero} and used in a context closely related to ours in \cite[\S4]{BaskinVasyWunschRadMink}. We note that the standard principal symbol of $\wt\cP(\sigma)$ is given by
\[
  p(r,\omega;\xi,\eta)=-\mu\xi^2-r^{-2}|\eta|^2.
\]
Then:
\begin{lemma}
\label{LemmaRNdSModule}
  The $\Psi^0(X)$-module
  \[
    \cM := \{ A\in\Psi^1(X) \colon \sigma_1(A)|_{L_1}=0 \}
  \]
  is closed under commutators. Moreover, we can choose finitely many generators of $\cM$ over $\Psi^0(X)$, denoted $A_0:=\id$, $A_1,\ldots,A_{N-1}$ and $A_N=\Lambda_{-1}\wt\cP(\sigma)$ with $\Lambda_{-1}\in\Psi^{-1}(X)$ elliptic, such that for all $1\leq j\leq N$, we have
  \begin{equation}
  \label{EqRNdSModule}
    [\wt\cP(\sigma),A_j] = \sum_{\ell=0}^N C_{j\ell}A_\ell,\quad C_{j\ell}\in\Psi^1(X),
  \end{equation}
  where $\sigma_1(C_{j\ell})|_{L_1}=0$ for $1\leq\ell\leq N-1$.
\end{lemma}
\begin{proof}
  Since $L_1$ is Lagrangian and thus in particular coisotropic, the first statement follows from the symbol calculus.
  
  Further, \eqref{EqRNdSModule} is a symbolic statement as well (since $[\wt\cP(\sigma),A_j]\in\Psi^2(X)$, and the summand $C_{j0}A_0=C_{j0}$ is a freely specifiable first order term), so we merely need to find symbols $a_1,\ldots,a_{N-1},a_N=\rho p$, homogeneous of degree $1$, with $\rho:=\sigma_{-1}(\Lambda_{-1})$, such that $\ham_p a_j=\sum c_{j\ell}a_\ell$ with $c_{j\ell}|_{L_1}=0$ for $1\leq\ell\leq N-1$. Note that this is clear for $j=N$, since in this case $\ham_p a_N=(\rho^{-1}\ham_p\rho)a_N$. We then let $a_1=\mu\xi$, and we take $a_2,\ldots,a_{N-1}\in\CI(T^*\Sph^2)$ to be linear in the fibers and such that they span the linear functions in $\CI(T^*\Sph^2)$ over $\CI(\Sph^2)$. We extend $a_2,\ldots,a_{N-1}$ to linear functions on $T^*X$ by taking them to be constant in $r$ and $\xi$. (Thus, these $a_j$ are symbols of differential operators in the spherical variables.) We then compute
  \[
    \ham_p a_1=-\mu'\mu\xi^2=\mu' r^{-2}|\eta|^2 + \mu' p,
  \]
  which is of the desired form since $|\eta|^2$ vanishes quadratically at $L_1$; moreover, for $2\leq j\leq N-1$, one readily sees that $\ham_p a_j=-r^{-2}H_{|\eta|^2}a_j$ vanishes quadratically at $L_1$ as well, finishing the proof.
\end{proof}

In the Lagrangian setting, this is a general statement, as shown by Haber and Vasy, see \cite[Lemma~2.1, Equation~(6.1)]{HaberVasyPropagation}. The positive commutator argument yielding the low regularity estimate at (generalized) radial sets, see \cite[Proposition~2.4]{VasyMicroKerrdS}, can now be improved to yield iterative regularity under the module $\cM$: indeed, we can follow the proof of \cite[Proposition~4.4]{BaskinVasyWunschRadMink} (which is for a generalized radial source/sink in the b-setting, whereas we work on a manifold without boundary here, so the weights in the reference can be dropped) or \cite[\S6]{HaberVasyPropagation} very closely; we leave the details to the reader. In order to compress the notation for products of module derivatives, we denote
\[
  A=(A_0,\ldots,A_N)
\]
in the notation of the lemma, and then use multiindex notation $A^\alpha=\prod_{i=0}^N A_i^{\alpha_i}$. The final result, reverting back to $\wh\cP(\sigma)$, is the following; recall that $L_{1,+}$ is a source and $L_{1,-}$ is a sink for the Hamilton flow within $T^*X$:
\begin{lemma}
\label{LemmaRNdSModuleEstimate}
  Let $A$ be a vector of generators of the module $\cM$ as above. Suppose $s_0<s<1/2-\beta_1\Im\sigma$. Let $G,B_1,B_2\in\Psi^0(X)$ be such that $B_1$ and $G$ are elliptic at $L_{1,+}$, resp.\ $L_{1,-}$, and all forward, resp.\ backward, null-bicharacteristics from $\WF'(B_1)\setminus L_{1,+}$, resp.\ $\WF'(B_1)\setminus L_{1,-}$, reach $\Ell(B_2)$ while remaining in $\Ell(G)$. Then 
  \[
    \sum_{|\alpha|\leq N}\|B_1 A^\alpha v\|_{H^s} \lesssim \sum_{|\alpha|\leq N}\|G A^\alpha \wh\cP(\sigma) v\|_{H^{s-1}} + \|B_2 v\|_{H^{s+N}} + \|v\|_{H^{s_0}}
  \]
\end{lemma}

In particular:

\begin{cor}
\label{CorRNdSResonantStateConormal}
  If $u$ is a resonant state of $\cP$, i.e.\ $\wh\cP(\sigma)u=0$, then $u$ is conormal to $r=r_1$ relative to $H^{1/2-\beta_1\Im\sigma-0}(Y)$, i.e.\ for any number of vector fields $V_1,\ldots,V_N$ on $X$ which are tangent to $r=r_1$, we have $V_1\cdots V_N u\in H^{1/2-\beta_1\Im\sigma-0}(Y)$.
\end{cor}
\begin{proof}
  Indeed, by the propagation of singularities, $u$ is smooth away from $L_{1,\pm}$, and then Lemma~\ref{LemmaRNdSModuleEstimate} implies the stated conormality property.
\end{proof}

We now turn to the conormal regularity estimate in the spacetime, b-, setting. Let us define
\[
  \cM_\bop = \{ A\in\Psib^1(M) \colon \sigma_1(A)|_{\cL_1}=0 \}.
\]
Using the stationary ($\tau$-invariant) extensions of the vector field generators of the module $\cM$ defined in Lemma~\ref{LemmaRNdSModule} together with $\tau D_\tau\in\cM_\bop$, one finds that the module $\cM_\bop$ is generated over $\Psib^0(M)$ by $A_0=\id$, $A_1,\ldots,A_N\in\Diffb^1(M)$ and $A_{N+1}=\Lambda_{-1}\cP$, with $\Lambda_{-1}\in\Psib^{-1}(M)$ elliptic, satisfying
\[
  [\cP,A_j]=\sum_{\ell=0}^{N+1} C_{j\ell}A_\ell,\quad C_{j\ell}\in\Psib^1(M),
\]
with $\sigma_1(C_{j\ell})_{\cL_1}=0$ for $1\leq\ell\leq N$. The proof of \cite[Proposition~4.4]{BaskinVasyWunschRadMink} then carries over to the saddle point setting of Proposition~\ref{PropRNdSRadialRecall} and gives in the below-threshold case (which is the relevant one at the Cauchy horizon):

\begin{prop}
\label{PropBConormal}
  Suppose $\cP$ is as above, and let $\alpha\in\R$, $k\in\Z_{\geq 0}$.

  If $s<1/2+\beta_1\alpha$, and if $u\in\Hb^{-\infty,\alpha}(M)$ then $L_{1,\pm}$ (and thus a neighborhood of $L_{1,\pm}$) is disjoint from $\WFb^{s,\alpha}(\cM_\bop^j u)$ for all $0\leq j\leq k$ provided $L_{1,\pm}\cap\WFb^{s-1,\alpha}(\cM_\bop^j \cP u)=\emptyset$ for $0\leq j\leq k$, and provided a punctured neighborhood of $L_{1,\pm}$, with $L_{1,\pm}$ removed, in $\Sigma\cap\Sb^*_X M$ is disjoint from $\WFb^{s+k,\alpha}(u)$.
\end{prop}

Thus, if $\cP u$ is conormal to $\cL_1$, i.e.\ remains in $\Hb^{s-1,\alpha}$ microlocally under iterative applications of elements of $\cM_\bop$ --- this in particular holds if $\cP u=0$ ---, then $u$ is conormal relative to $\Hb^{s,\alpha}$, provided $u$ lies in $\Hb^{\infty,\alpha}$ in a punctured neighborhood of $L_1$. Using Proposition~\ref{PropBConormal} at the radial set $L_1$ in the part of the proof of Proposition~\ref{PropRNdSPartialAsymp} where the regularity of $u'$ is established, we obtain:

\begin{thm}
\label{ThmRNdSPartialAsympConormal}
  Let $\alpha>0$ be as in Proposition~\ref{PropRNdSPartialAsymp}, and suppose $u$ is the forward solution of
  \[
    \Box_g u=f\in\CIc(\Omega^\circ), \quad r>r_1.
  \]
  Then $u$ has a partial asymptotic expansion $u=u_0\chi(\tau)+u'$, where $\chi\equiv 1$ near $\tau=0$, $\chi\equiv 0$ away from $\tau=0$, and with $u_0\in\C$, and
  \[
    V_1\cdots V_N u'\in\Hb^{1/2+\alpha\beta_1-0,\alpha}
  \]
  for all $N=0,1,\ldots$ and all vector fields $V_j\in\Vb(\Omega)$ which are tangent to the Cauchy horizon $r=r_1$; here, $\beta_1$ is given by \eqref{EqRNdSRadialPointBeta}.

  The same result holds true, without the constant term $u_0$, for the forward solution of the massive Klein--Gordon equation $(\Box_g-m^2)u=f$, $m>0$ small.
\end{thm}
\begin{proof}
  For the massive Klein--Gordon equation, the only change in the analysis is that the simple resonance at $0$ moves into the lower half plane, see e.g.\ the perturbation computation in \cite[Lemma~3.5]{HintzVasySemilinear}; this leads to the constant term $u_0$, which was caused by the resonance at $0$, being absent.
\end{proof}

This implies the estimate \eqref{EqRNdSConormalKilling} and thus yields Corollary~\ref{CorRNdSBoundedness} as well.

\subsection{Existence of high regularity solutions at the Cauchy horizon for near-extremal black holes}
\label{SubsecRNdSHighReg}

The amount of decay $\alpha$ (and thus the amount of regularity we obtain) in Theorem~\ref{ThmRNdSPartialAsympConormal} is directly linked to the size of the \emph{spectral gap}, i.e.\ the size of the resonance-free strip below the real axis, as explained in \S\ref{SubsecRNdSAsymp}. Due to the work of S\'a Barreto--Zworski \cite{SaBarretoZworskiResonances} in the spherically symmetric case and general results by Dyatlov \cite{DyatlovResonanceProjectors} at ($r$-)normally hyperbolic trapping (for every $r$), the size of the \emph{essential spectral gap} is given in terms of dynamical quantities associated to the trapping, see Proposition~\ref{PropRNdSSemiclassical2}; we recall that the essential spectral gap is the supremum of all $\alpha'$ such that there are only finitely many resonances above the line $\Im\sigma=-\alpha'$. Thus, the essential spectral gap only concerns the high energy regime, i.e.\ it does not give any information about low energy resonances. In this section, we compute the size of the essential spectral gap in some limiting cases; the possibly remaining finitely many resonances between $0$ and the resonances caused by the trapping will be studied separately in future work. We give some indications of the expected results in Remark~\ref{RmkRNdSHighReg}.

In order to calculate the relevant dynamical quantities at the trapped set, we compute the linearization of the flow in the $(r,\xi)$ variables at the trapped set $\Gamma$: we have
\[
  \ham_G(r-r_P)\equiv-2\mu(r_P)\xi
\]
modulo functions vanishing quadratically at $\Gamma$, and in the same sense
\[
  \ham_G\xi\equiv\mu^{-2}\mu'-2r^{-3}|\eta|^2\equiv \mu^{-2}\mu'-2r^{-1}\mu^{-1}=\mu^{-2}r^2\pa_r(r^{-2}\mu),
\]
which in view of $\pa_{rr}(r^{-2}\mu)|_{r=r_P}=-2r_P^{-5}(2r_P-3\bhm)$ (see also \eqref{EqRNdSFlowMuPrime}) gives
\[
  \ham_G\begin{pmatrix}r-r_P\\ \xi\end{pmatrix} = \begin{pmatrix}0&-2\mu(r_P)\\ 2\mu(r_P)^{-2}r_P^{-3}(2r_P-3\bhm)&0\end{pmatrix}\begin{pmatrix}r-r_P\\ \xi\end{pmatrix}
\]
Therefore, the expansion rate of the flow in the normal direction at $\Gamma$ is equal to
\begin{equation}
\label{EqRNdSHighRegNumin}
  \numin = \frac{2}{r_P}\Bigl(\frac{2-3r_P^{-1}\bhm}{\mu(r_P)}\Bigr)^{1/2}.
\end{equation}
To find the size of the essential spectral gap for the forward problem of $\Box_g$, we need to compute the size of the imaginary part of the subprincipal symbol of the semiclassical rescaling of $\wh{\Box_g}$ at the semiclassical trapped set. Put $h=|\sigma|^{-1}$, $z=h\sigma$, then
\[
  \cP_{h,z}=h^2\wh{\Box_g}(h^{-1}z)=\mu^{-1}z^2-r^{-2}hD_r r^2\mu hD_r-r^{-2}h^2\Delta_{\Sph^2}.
\]
With $z=\pm 1-ih\alpha$, $\alpha\in\R$, we thus obtain
\begin{equation}
\label{EqRNdSHighRegSubpr}
  \sigma_\semi\Bigl(\frac{1}{2ih}(\cP_{h,z}-\cP_{h,z}^*)\Bigr) = \mp 2\mu^{-1}\alpha.
\end{equation}
The essential spectral gap thus has size at least $\alpha$ provided $\numin/2>2\mu^{-1}\alpha$, so
\[
  \alpha < \gamma_0 := \frac{\mu(r_P)\numin}{4} = \frac{1}{2r_P}\bigl(\mu(r_P)(2-3r_P^{-1}\bhm)\bigr)^{1/2}.
\]
We compute the quantity on the right for near-extremal Reissner--Nordstr\"om--de Sitter black holes with very small cosmological constant; first, using the radius of the photon sphere for the Reissner--Nordstr\"om black hole with $\Lambda=0$,
\[
  r_P = \frac{1}{2}(3\bhm+\sqrt{9\bhm^2-8 Q^2}),
\]
and the radius of the Cauchy horizon
\[
  r_1 = \bhm - \sqrt{\bhm-Q^2},
\]
we obtain
\[
  \gamma_0(\bhm,Q) = \frac{2\sqrt{(3\bhm^2-2Q^2+\bhm\sqrt{9\bhm^2-8Q^2})\sqrt{9\bhm^2-8Q^2}}}{(3\bhm+\sqrt{9\bhm^2-8Q^2})^{5/2}}
\]
for the size of the essential spectral gap for resonances caused by the trapping in the case $\Lambda=0$. (For $Q=0$, one finds $\gamma_0=1/(2\cdot 3^{3/2}\bhm)$, which agrees with \cite[Equation~(0.3)]{DyatlovAsymptoticDistribution} for $\Lambda=0$.) In the extremal case $Q=\bhm$, we find $\gamma_0=1/(8\bhm\sqrt{2})$. Furthermore, we have
\[
  \beta_1(\bhm,Q):=\frac{2}{|\mu'(r_1)|} = \frac{(\bhm-\sqrt{\bhm^2-Q^2})^2}{\sqrt{\bhm^2-Q^2}}.
\]
Thus, $\beta_1(\bhm,\bhm(1-\eps))=\bhm/(2\eps)^{1/2}+\cO(1)$; therefore,
\[
  \gamma_0(\bhm,\bhm(1-\eps))\beta_1(\bhm,\bhm(1-\eps)) = \frac{1}{16\eps^{1/2}} + \cO(1),
\]
which blows up as $\eps\to 0+$; this corresponds to the fact the surface gravity of extremal black holes vanishes. Given $s\in\R$, we can thus choose $\eps>0$ small enough so that $1/2+\gamma_0\beta_1>s$, and then taking $\Lambda>0$ to be small, the same relation holds for the $\Lambda$-dependent quantities $\gamma_0$ and $\beta_1$. Since there are only finitely many resonances in any strip $\Im\sigma>-\alpha>-\gamma_0$, we conclude by Theorem~\ref{ThmRNdSPartialAsympConormal},  taking $\alpha<\gamma_0$ close to $\gamma_0$, that for forcing terms $f$ which are orthogonal to a finite-dimensional space of dual resonant states (corresponding to resonances in $\Im\sigma>-\alpha$), the solution $u$ has regularity $\Hb^{s,\alpha}$ at the Cauchy horizon. Put differently, for near-extremal Reissner--Nordstr\"om--de Sitter black holes with very small cosmological constant $\Lambda>0$, waves with initial data in a finite codimensional space (within the space of smooth functions) achieve any fixed order of regularity at the Cauchy horizon, in particular better than $\Hloc^1$.

\begin{rmk}
\label{RmkRNdSHighReg}
  Numerical investigations of linear scalar waves \cite{BradyChambersKrivanLagunaCosmConst,BradyMossMyersCosmicCensorship,BradyChambersLaarakeersPoissonSdSFalloff} and arguments using approximations of the scattering matrix \cite{ChoudhuryPadmanabhanQNMforSdS} suggest that there are indeed resonances roughly at $\sigma=-i\kappa_j$, $j=2,3$, where $\kappa_2$ and $\kappa_3$ are the surface gravities of the cosmological horizon, see \eqref{EqRNdSSurfGrav}; as $\Lambda\to 0+$, we have $\kappa_3\to 0+$, and for extremal black holes with $\Lambda=0$, we have $\kappa_2=0$. (On the static de Sitter spacetime, there is a resonance exactly at $-i\kappa_3$, as a rescaling shows: for $\Lambda_0=3$, one has $e^{-t_0}$ decay to constants away from the cosmological horizon, $t_0$ the static time coordinate, see e.g.\ \cite{VasyWaveOndS}; now static de Sitter space $\dS_\Lambda$ with cosmological constant $\Lambda>0$ can be mapped to $\dS_3$ via $t_0=\kappa_3 t$, $r_0=\kappa_3 r$, where $\kappa_3=\sqrt{\Lambda/3}$ is the surface gravity of the cosmological horizon, and $t_0,r_0$, resp.\ $t,r$, are static coordinates on $\dS_3$, resp.\ $\dS_\Lambda$. Under this map, the metric on $\dS_3$ is pulled back to a constant multiple of the metric on $\dS_\Lambda$. Thus, waves on $\dS_\Lambda$ decay to constants with the speed $e^{-\kappa_3 t}$, which corresponds to a resonance at $-i\kappa_3$.)
  
  Our analysis is consistent with the numerical results, \emph{assuming the existence of these resonances}: we expect linear waves in this case to be generically no smoother than $H^{1/2+\min(\kappa_2,\kappa_3)/\kappa_1}$ at the Cauchy horizon, which highlights the importance of the relative sizes of the surface gravities for understanding the regularity at the Cauchy horizon. For near-extremal black holes, where $\kappa_2<\kappa_3$, this gives $H^{1/2+\kappa_2/\kappa_1}$, thus the local energy measured by an observer crossing the Cauchy horizon is of the order $(r-r_1)^{\kappa_2/\kappa_1-1}$, which diverges in view of $\kappa_2<\kappa_1$; this agrees with \cite[Equation~9]{BradyMossMyersCosmicCensorship}. We point out however that the waves are still in $\Hloc^1$ if $2\kappa_2>\kappa_1$, which is satisfied for near-extremal black holes. This is analogous to Sbierski's criterion \cite[\S4.4]{SbierskiThesis} for ensuring the finite energy of waves at the Cauchy horizon of linear waves with fast decay along the event horizon.

  The rigorous study of resonances associated with the event and cosmological horizons will be subject of future work.
\end{rmk}

\subsection{Tensor-valued waves}
\label{SubsecRNdSBundles}

The analysis presented in the previous sections goes through with only minor modifications if we consider the wave equation on natural vector bundles.

For definiteness, we focus on the wave equation, more precisely the Hodge d'Alembertian, on differential $k$-forms, $\Box_k:=d\delta+\delta d$. In this case, mode stability and asymptotic expansions up to decaying remainder terms in the region $\Omega_{23}^\circ$, a neighborhood of the black hole exterior region, were proved in \cite{HintzVasyKdsFormResonances}. The previous arguments apply to $\Box_k$; the only difference is that the threshold regularity at the radial points at the horizons shifts. At the event horizon and the cosmological horizon, this is inconsequential, as we may work in spaces of arbitrary high regularity there; at the Cauchy horizon however, one has, fixing a time-independent \emph{positive definite inner product} on the fibers of the $k$-form bundle with respect to which one computes adjoints:
\[
  \sigma_1\Bigl(\frac{1}{2i}(\Box_k^*-\Box_k)\Bigr) = \pm\wh\rho^{-1}\beta_0\wh\beta
\]
at $L_{1,\pm}$, with $\beta_0=\wh\rho^{-1}\rham_G\wh\rho=|\mu'(r_1)|$, and $\wh\beta$ and endomorphism on the $k$-form bundle; and one can compute that the lowest eigenvalue of $\wh\beta$ (which is self-adjoint with respect to the chosen inner product) is equal to $-k$. But then the regularity one can propagate into $L_{1,\pm}$ for $u'\in\Hb^{-\infty,\alpha}$, $\alpha\in\R$, solving $\Box_k u'=f'$, $f'$ compactly supported and smooth, is $\Hb^{1/2+\alpha\beta_1-k-0,\alpha}$, as follows from \cite[Proposition~2.1 and Footnote~5]{HintzVasySemilinear}. Thus, in the partial asymptotic expansion in Theorem~\ref{ThmRNdSPartialAsympConormal} (which has a different leading order term now, coming from stationary $k$-form solutions of the wave equation), we can only establish conormal regularity of the remainder term $u'$ at the Cauchy horizon relative to the space $\Hb^{1/2+\alpha\beta_1-k-0,\alpha}$, which for small $\alpha>0$ gives Sobolev regularity $1/2-k+\eps$, for small $\eps>0$. Assuming that the leading order term is smooth at the Cauchy horizon (which is the case, for example, for 2-forms, see \cite[Theorem~4.3]{HintzVasyKdsFormResonances}), we therefore conclude that, as soon as we consider $k$-forms $u$ with $k\geq 1$, our methods do not yield uniform boundedness of $u$ up to the Cauchy horizon; however, we remark that the conormality does imply uniform bounds as $r\to r_1+$ of the form $(r-r_1)^{-k+\eps}$, $\eps>0$ small.

A finer analysis would likely yield more precise results, in particular boundedness for certain components of $u$; and, as in the scalar setting, a converse result, namely showing that such a blow-up does happen, is much more subtle. We do not pursue these issues in the present work.

\section{Kerr--de Sitter space}
\label{SecKdS}

We recall from \cite[\S6]{VasyMicroKerrdS} the form of the Kerr--de Sitter metric with parameters $\Lambda>0$ (cosmological constant), $\bhm>0$ (black hole mass) and $a$ (angular momentum),
\begin{equation}
\label{EqKdSMetric}
\begin{split}
  g &= -\rho^2\Bigl(\frac{dr^2}{\wt\mu}+\frac{d\theta^2}{\kappa}\Bigr) - \frac{\kappa\sin^2\theta}{(1+\gamma)^2\rho^2}(a\,dt-(r^2+a^2)\,d\phi)^2 \\
    &\qquad + \frac{\wt\mu}{(1+\gamma)^2\rho^2}(dt-a\sin^2\theta\,d\phi)^2,
\end{split}
\end{equation}
where
\begin{gather*}
  \wt\mu(r,a,\Lambda,\bhm) = (r^2+a^2)\Bigl(1-\frac{\Lambda r^2}{3}\Bigr) - 2\bhm r, \\
  \rho^2 = r^2+a^2\cos^2\theta, \quad \gamma = \frac{\Lambda a^2}{3}, \quad \kappa = 1+\gamma\cos^2\theta.
\end{gather*}
(Our $(t,\phi)$ are denoted $(\tilde t,\tilde\phi)$ in \cite{VasyMicroKerrdS}, while our $(t_*,\phi_*)$ are denoted $(t,\phi)$ there.) In order to guarantee the existence of a Cauchy horizon, we need to assume $a\neq 0$. Analogous to Definition~\ref{DefRNdSNonDegenerate}, we make a non-degeneracy assumption:

\begin{definition}
\label{DefKdSNonDegenerate}
  We say that the Kerr--de Sitter spacetime with parameters $\Lambda>0,\bhm>0,a\neq 0$ is \emph{non-degenerate} if $\wt\mu$ has $3$ simple positive roots $0<r_1<r_2<r_3$.
\end{definition}

One easily checks that
\[
  \wt\mu>0\tn{ in }(0,r_1)\cup(r_2,r_3),\quad \wt\mu<0\tn{ in }(r_1,r_2)\cup(r_3,\infty);
\]
and again, $r=r_1$ (in the analytic extension of the spacetime) is called the \emph{Cauchy horizon}, $r=r_2$ the \emph{event horizon} and $r=r_3$ the \emph{cosmological horizon}.

We consider a simple case in which non-degeneracy can be checked immediately:
\begin{lemma}
\label{LemmaKdSNonDegenerate}
  Suppose $9\Lambda\bhm^2<1$, and denote the three non-negative roots of $\wt\mu(r,0,\Lambda,\bhm)$ by $r_{1,0}=0<r_{2,0}<r_{3,0}$. Then for small $a\neq 0$, $\wt\mu$ has three positive roots $r_j(a)$, $j=1,2,3$, with $r_j(0)=r_{j,0}$, depending smoothly on $a^2$, and $r_1(a)=\frac{a^2}{2\bhm}+\cO(a^4)$.
\end{lemma}
\begin{proof}
  We recall that the condition \eqref{EqRNdSSdSNonDegenerate} ensures the existence of the roots $r_{j,0}$ as stated. One then computes for $\wt r_1(A):=r_1(\sqrt{A})$ that $\wt r_1'(0)=1/(2\bhm)$, giving the first statement.
\end{proof}

In order to state unconditional results later on, we in fact \emph{from now on assume to be in the setting of this lemma, i.e.\ we consider slowly rotating Kerr--de Sitter black holes}; see Remark~\ref{RmkKdSResGeneral} for further details.

\subsection{Construction of the compactified spacetime}
\label{SubsecKdSMfd}

As in \S\ref{SubsecRNdSMfd}, we discuss the smooth extension of the metric $g$ across the horizons and construct the manifold on which the linear analysis will take place; all steps required for this construction are slightly more complicated algebraically but otherwise very similar to the ones in the Reissner--Nordstr\"om--de Sitter setting, so we shall be brief.

Thus, with
\[
  s_j=-\sgn\wt\mu'(r_j),\quad \tn{so }s_1=1,\ s_2=-1,\ s_3=1,
\]
we will take
\begin{equation}
\label{EqKdSTimeStar}
  t_*:=t - F_j(r),\quad \phi_*:=\phi - P_j(r)
\end{equation}
for $r$ near $r_j$, where
\[
  F_j'(r) = s_j\Bigl(\frac{(1+\gamma)(r^2+a^2)}{\wt\mu} + c_j\Bigr),\quad P_j'(r)=s_j\frac{(1+\gamma)a}{\wt\mu}.
\]
Using $aF_j'-(r^2+a^2)P_j'=a s_j c_j$ and $F_j'-a\sin^2\theta P_j' = s_j\Bigl(\frac{(1+\gamma)\rho^2}{\wt\mu}+c_j\Bigr)$, one computes
\begin{equation}
\label{EqKdSMfdMetricExt}
\begin{split}
  g &= -\frac{\kappa\sin^2\theta}{(1+\gamma)^2\rho^2}(a(s_j c_j\,dr+dt_*)-(r^2+a^2)\,d\phi_*)^2 \\
    &\qquad + \frac{\wt\mu}{(1+\gamma)^2\rho^2}((s_j c_j\,dr+dt_*)-a\sin^2\theta\,d\phi_*)^2 \\
    &\qquad + \frac{2s_j}{1+\gamma}((s_j c_j\,dr+dt_*)-a\sin^2\theta\,d\phi_*)\,dr - \frac{\rho^2\,d\theta^2}{\kappa};
\end{split}
\end{equation}
using e.g.\ the frame $v_1=\pa_r-s_j c_j\,\pa_{t_*}$, $v_2=a\sin^2\theta\,\pa_{t_*}+\pa_{\phi_*}$, $v_3=\pa_{t_*}$ and $\pa_4=\pa_\theta$, one finds the volume density to be
\[
  |dg|=(1+\gamma)^{-2}\rho^2\sin\theta\,dt_*\,dr\,d\phi_*\,d\theta;
\]
moreover, the form of the dual metric is
\begin{equation}
\label{EqKdSMfdDualMetricExt}
\begin{split}
  \rho^2 G&=-\wt\mu(\pa_r-s_j c_j\,\pa_{t_*})^2 + 2a s_j(1+\gamma)(\pa_r - s_j c_j\,\pa_{t_*})\,\pa_{\phi_*} \\
  &\qquad + 2s_j(1+\gamma)(r^2+a^2)(\pa_r-s_j c_j\,\pa_{t_*})\,\pa_{t_*} \\
  &\qquad - \frac{(1+\gamma)^2}{\kappa\sin^2\theta}(a\sin^2\theta\,\pa_{t_*}+\pa_{\phi_*})^2 - \kappa\,\pa_\theta^2.
\end{split}
\end{equation}
This is a non-degenerate Lorentzian metric apart from the usual singularity of the spherical coordinates $(\phi_*,\theta)$, which indeed is merely a coordinate singularity as shown by a change of coordinates \cite[\S6.2]{VasyMicroKerrdS}, see also Remark~\ref{RmkKdSFlowValidCoord} below.

As in the Reissner--Nordstr\"om--de Sitter case, one can start by choosing the functions $c_2$ and $c_3$ so that $F_2(r)=0$ for $r\geq r_2+\delta$ and $F_3(r)=0$ for $r\leq r_3-\delta$, so that $t_*$ in \eqref{EqKdSTimeStar} is well-defined in a neighborhood of $[r_2,r_3]$, and moreover one can choose $c_2$ and $c_3$ so that $dt_*$ is timelike in $[r_2-2\delta,r_3+2\delta]$: indeed, this is satisfied provided
\begin{equation}
\label{EqKdSMfdDtTimelike}
  \wt\mu c_j^2+2(1+\gamma)(r^2+a^2)c_j+a^2(1+\gamma)^2 < 0.
\end{equation}
We note that in $\wt\mu<0$, we can take $c_j$ to be large and negative, and then at $r=r_2-2\delta$, we obtain
\begin{equation}
\label{EqKdSMfdDtDrTimelike}
  \rho^2 G(-dr,dt_*) = \wt\mu c_2+2(1+\gamma)(r^2+a^2) > 0.
\end{equation}
Therefore, $-dr$ is future timelike for $r\in(r_1,r_2)$. Near $r_1$ then, more precisely in $r_1-2\delta<r<r_1+2\delta$, we can arrange for $-dt_*$ to be timelike again, and since $s_1=-s_2$ has the opposite sign, we find that $\rho^2 G(-dr,-dt_*)>0$, i.e.\ $-dt_*$ is future timelike there.

In order to cap off the problem in $r<r_1$, we again modify $\wt\mu$ to a smooth function $\wt\mu_*$. Since we can hide all the (possibly complicated) structure of the extension when $t_*\to\infty$ using complex absorption, we simply choose $\wt\mu_*$ such that
\begin{align*}
  \wt\mu_* & \equiv\mu\tn{ in }[r_1-2\delta,\infty), \\
  \wt\mu_* & \ \tn{has a single simple zero at }r_0\in(0,r_1).
\end{align*}
(See also the discussion following \eqref{EqRNdSMuStar}.) We can then extend the metric $g$ past $r_0$ by defining $t_*,\phi_*$ near $r=r_0$ as in \eqref{EqKdSTimeStar}, with $\wt\mu$ replaced by $\wt\mu_*$, and with $s_0=-\sgn\wt\mu_*'(r_0)=-1$. We can then arrange $-dt_*$ to be future timelike in $r_0-2\delta\leq r\leq r_1+2\delta$, and $dr$ is future timelike at $r=r_0-2\delta$ by a computation analogous to \eqref{EqKdSMfdDtDrTimelike}.

We can now define spacelike hypersurfaces $H_{I,0},H_F,H_{F,2},H_I,H_{F,3}$ exactly as in \eqref{EqRNdSMfdHyper}, bounding a domain with corners $\Omega^\circ$ inside
\[
  M^\circ = \R_{t_*}\times (r_0-4\delta,r_3+4\delta)_r \times \Sph^2,
\]
and we will analyze the wave equation (modified in $r<r_1$) on the compactified region
\[
  \Omega \subset M = [0,\infty)_\tau\times(r_0-4\delta,r_3+4\delta)_r\times\Sph^2,\quad \tau:=e^{-t_*};
\]
we further let $X=\pa M$, $Y=\Omega\cap\pa M$.

\subsection{Global behavior of the null-geodesic flow}
\label{SubsecKdSFlow}

Since it simplifies a number of computations below, we will study the null-geodesic flow of $\rho^2\Box_g$, i.e.\ the flow of $\ham_{\rho^2 G}$ within the characteristic set $\Sigma=G^{-1}(0)$, where $G$ denotes the dual metric function.

By pasting $-dt_*$ in $r\leq r_1+2\delta$, $-dr$ in $r_1+\delta\leq r\leq r_2-\delta$ and $dt_*$ in $r_2-2\delta\leq r\leq r_3+2\delta$ together using a non-negative partition of unity, we can construct a smooth, globally future timelike covector field $\varpi$ on $\Omega$ and use it to split the characteristic set into components $\Sigma_\pm$ as in \eqref{EqRNdSFlowCharComp}.

Since the global dynamics of the null-geodesic flow in a neighborhood $r_2-2\delta\leq r\leq r_3+2\delta$ of the exterior region are well-known, with saddle points of the flow (generalized radial sets) at $L_2,L_3$, where we define $L_j=\Nb^*(Y\cap\{r=r_j\})$, and a normally hyperbolically trapped set $\Gamma$.

As in parts of the discussion in \S\ref{SubsecRNdSFlow}, it is computationally convenient to work with $t_0:=t-F(r)$ instead of $t_*$ near $r=r_j$, where $F'(r)=s_j(1+\gamma)(r^2+a^2)\wt\mu_*^{-1}$ (i.e.\ effectively putting $c_j=0$). Let $\tau_0:=e^{-t_0}$ and write b-covectors as
\begin{equation}
\label{EqKdSFlowCoords}
  \sigma\,\frac{d\tau_0}{\tau_0} + \xi\,dr + \zeta\,d\phi_* + \eta\,d\theta;
\end{equation}
then the dual metric function reads
\begin{equation}
\label{EqKdSFlowCj0}
\begin{split}
  \rho^2 G&=-\wt\mu_*\xi^2 + 2a s_j(1+\gamma)\xi\zeta - 2s_j(1+\gamma)(r^2+a^2)\xi\sigma \\
    &\qquad - \frac{(1+\gamma)^2}{\kappa\sin^2\theta}(a\sin^2\theta\,\sigma-\zeta)^2-\kappa\eta^2
\end{split}
\end{equation}

\begin{rmk}
\label{RmkKdSFlowValidCoord}
  Valid coordinates near the poles $\theta=0,\pi$ are
  \[
    y=\sin\theta\cos\phi_*,\quad z=\sin\theta\sin\phi_*,
  \]
  and writing $\zeta\,d\phi_*+\eta\,d\theta=\lambda\,dy+\nu\,dz$, one finds $\sin^2\theta=y^2+z^2$ and $\zeta=-\lambda z+\nu y$; thus to see the smoothness of $-\rho^2 G$ near the poles, one merely needs to rewrite
  \[
    \wt p := \frac{(1+\gamma)^2\zeta^2}{\kappa\sin^2\theta} + \kappa\eta^2
  \]
  as
  \[
    \kappa\wt p=(1+\gamma)^2\Bigl(\frac{\zeta^2}{\sin^2\theta}+\eta^2\Bigr) - \gamma(2+\gamma(1+\cos^2\theta))\sin^2\theta\,\eta^2
  \]
  and notice that $\sin^2\theta\,\eta^2 = (1-y^2-z^2)(\lambda y+\nu z)^2$ is smooth, as is $\sin^{-2}\theta\,\zeta^2+\eta^2$ since this is simply the dual metric function on $\Sph^2$ in spherical coordinates.
\end{rmk}

We study the rescaled Hamilton flow near $L_j$ using the coordinates \eqref{EqKdSFlowCoords} and introducing $\wh\rho=|\xi|^{-1}$, $\wh\sigma=\wh\rho\sigma$, $\wh\zeta=\wh\rho\zeta$, $\wh\eta=\wh\rho\eta$ as the fiber variables similarly to \eqref{EqRNdSFlowRescaledCoord}: thus,
\[
  \rham_{\rho^2 G} = \wh\rho\ham_{\rho^2 G},
\]
and we find that at $L_j\cap\{\pm\xi>0\}$, where $\wh\sigma=\wh\zeta=\wh\eta=0$,
\begin{gather*}
  \wh\rho^{-1}\rham_{\rho^2 G}\wh\rho = \ham_{\rho^2 G}|\xi|^{-1}=\pm\xi^{-2}\pa_r(\rho^2 G) = \mp\wt\mu_*'(r_j), \\
  \tau_0^{-1}\rham_{\rho^2 G}\tau_0 = \wh\rho\pa_\sigma(\rho^2 G) = \mp 2s_j(1+\gamma)(r_j^2+a^2),
\end{gather*}
and thus the quantity controlling the threshold regularity at $L_j$ is
\begin{equation}
\label{EqKdSFlowThreshold}
  \beta_j = -\frac{\tau_0\rham_{\rho^2 G}\tau_0}{\wh\rho^{-1}\rham_{\rho^2 G}\wh\rho} = \frac{2(1+\gamma)(r_j^2+a^2)}{|\wt\mu_*'(r_j)|}.
\end{equation}
Furthermore, if we put
\begin{equation}
\label{EqKdSFlowCarter}
  p_C = \frac{(1+\gamma)^2}{\kappa\sin^2\theta}(a\sin^2\theta\,\sigma-\zeta)^2 + \kappa\eta^2,
\end{equation}
then $\rham_{\rho^2 G}p_C=0$, so the quadratic defining function $\rho_0:=\wh\rho^2(p_C+\sigma^2)$ of $L_j$ within the characteristic set over the boundary, $\Sigma\cap\Sb^*_X M=\Sigma\cap\pa\rcTb^*_X M$, satisfies
\[
  \rham_{\rho^2 G}p_C = \mp 2\wt\mu_*'(r_j)\rho_0\tn{ at }L_j\cap\{\pm\xi>0\};
\]
as in the Reissner--Nordstr\"om--de Sitter case, this implies that $L_j$ is a source or sink within $\pa\rcTb^*_X M$, with a stable or unstable manifold $\cL_j=\Nb^*\{r=r_j\}$ transversal to the boundary. For $\zeta\in\Sigma\cap\Tb^*_{\{r=r_j\}}M$ written as \eqref{EqKdSFlowCoords}, one can check that $\ham_{\rho^2 G}r=2\rho^2 G(\zeta,dr)=0$ if and only if $\eta=\zeta=\sigma=0$, i.e.\ if and only if $\zeta\in\cL_j$; in $\Sigma_\pm\cap\Tb^*_{\{r=r_j\}}M\setminus\cL_j$, the quantity $\ham_{\rho^2 G}r$ therefore has a sign (which is the same as in the discussion around \eqref{EqRNdSCrossingHorizon}), depending on the component of the characteristic set; thus, null-geodesics in a fixed component $\Sigma_\pm$ of the characteristic set can only cross $r=r_j$ in one direction. Furthermore, in the regions where $\wt\mu_*<0$, and thus $dr$ is timelike, we have $\ham_{\rho^2 G}r\neq 0$, see also \eqref{EqRNdSHamRHypRegion}.

Since we will place complex absorption immediately beyond $r=r_1$, i.e.\ in $r_0-\delta'<r<r_1-\delta'$ for $\delta'>0$ very small, it remains to check that at finite values of $t_*$ in this region, all null-geodesics escape either to $\tau=0$ or to $H_F$; but this follows from the timelike nature of $dt_*$ there, which gives that $\rham_{\rho^2 G}t_*$ is non-zero, in fact bounded away from zero.

To summarize, the global behavior of the null-geodesic flow in $r>r_1-2\delta$ is the same as that of the Reissner--Nordstr\"om--de Sitter solution; see Figure~\ref{FigRNdSFlow}. We point out that the existence of an ergoregion is irrelevant for our analysis: its manifestation is merely that null-geodesics tending to, say, the event horizon $r=r_2$ in the backward direction, may have a segment in $r>r_2$ before (possibly) crossing the event horizon into $r<r_2$; see also \cite[Figure~8]{VasyMicroKerrdS}.

\subsection{Results for scalar waves}
\label{SubsecKdSRes}

We use a complex absorbing operator $\cQ\in\Psib^2(M)$ as in \S\ref{SubsecRNdSRegularity}, with $\mp\sigma(\cQ)\geq 0$ on $\Sigma_\pm$, and which is elliptic in $t_*\geq t_{*,0}+1$, $r_0-\delta'<r<r_1-\delta'$, where $\delta'>0$ is chosen sufficiently small to ensure that the dynamics near the generalized radial set $L_1$ control the dynamics in $r_1-\delta'\leq r\leq r_1$: that is, null-geodesics near either tend to $L_1$ or enter the elliptic region of $\cQ$, i.e.\ $r<r_1-\delta'$, in finite time, unless they cross $t_*=t_{*,0}$, i.e.\ $H_F$, or $r=r_1$.

The analysis in \S\S\ref{SubsecRNdSRegularity}--\ref{SubsecRNdSConormal} now goes through \emph{mutatis mutandis}. (For completeness, we note that the threshold quantity $\beta_1$, see \eqref{EqKdSFlowThreshold}, for small $a$ is given by $\beta_1 = \frac{a^2}{\bhm} + \cO(a^4)$.) In fact, to prove conormal regularity, we can use the same module generators as those constructed in the proof of Lemma~\ref{LemmaRNdSModule}, and the b-version, see the discussion around Proposition~\ref{PropBConormal}, goes through without changes as well.

\begin{rmk}
\label{RmkKdSResCarter}
  There exists a second order `Carter operator' $\cP_C\in\Diffb^2(M)$, with principal symbol given by $p_C$ in \eqref{EqKdSFlowCarter}, that commutes with $\rho^2\Box_g$; concretely, in the coordinates used in \eqref{EqKdSFlowCj0} (which are valid near $r=r_1$),
  \begin{align*}
    \cP_C &= \frac{1}{\sin\theta}D_\theta\kappa\sin\theta D_\theta + \frac{(1+\gamma)^2}{\kappa\sin^2\theta}D_{\phi_*}^2 \\
      &\qquad + \frac{2a(1+\gamma)^2}{\kappa}D_{t_0}D_{\phi_*} + \frac{(1+\gamma)^2 a^2\sin^2\theta}{\kappa}D_{t_0}^2.
  \end{align*}
  Since $D_{t_0}$ and $D_{\phi_*}$ commute with $\rho^2\Box_g$, and since moreover the sum of the first two terms of $\cP_C$ is an elliptic operator on $\Sph^2$, we conclude, commuting $\cP_C$ through the equation $\rho^2\Box_g u=f\in\CIc(\Omega^\circ)$, $r>r_1-2\delta$, that $u$ is smooth in $t_0$ and the angular variables.

  Thus, we can deduce conormal regularity (apart from iterative regularity under application of $\wt\mu_* D_r$) for $u$ using such commutation arguments as well. Note however that the existence of such the `hidden symmetry' $\cP_C$ is closely linked to the complete integrability of the geodesic flow on Kerr--de Sitter space, while the microlocal argument proving conormality applies in much more general situations and different contexts, see e.g.\ \cite{HassellMelroseVasySymbolicOrderZero}.
\end{rmk}

We content ourselves with stating the analogues of Theorem~\ref{ThmRNdSPartialAsympConormal} and Corollary~\ref{CorRNdSBoundedness} in the Kerr--de Sitter setting:

\begin{thm}
\label{ThmKdSPartialAsympConormal}
  Suppose the angular momentum $a\neq 0$ is very small, such that there exists $\alpha>0$ with the property that the forward problem for the wave equation in the neighborhood $r_2-2\delta\leq r\leq r_3+2\delta$ of the domain of outer communications has no resonances in $\Im\sigma\geq -\alpha$ other than the simple resonance at $\sigma=0$. Let $u$ be the forward solution of
  \[
    \Box_g u=f\in\CIc(\Omega^\circ), \quad r>r_1.
  \]
  Then $u$ has a partial asymptotic expansion $u=u_0\chi(\tau)+u'$, with $u_0\in\C$ and $\chi\equiv 1$ near $\tau=0$, $\chi\equiv 0$ away from $\tau=0$, and
  \[
    V_1\cdots V_N u'\in\Hb^{1/2+\alpha\beta_1-0,\alpha}
  \]
  for all $N=0,1,\ldots$ and all vector fields $V_j\in\Vb(\Omega)$ which are tangent to the Cauchy horizon $r=r_1$; here, $\beta_1$ is given by \eqref{EqKdSFlowThreshold}. In particular, there exists a constant $C>0$ such that $|u'(\tau,x)|\leq C\tau^\alpha$, and $u$ is uniformly bounded in $r>r_1$.
\end{thm}

Again, the same result holds, without the constant term $u_0$, for solutions of the massive Klein--Gordon equation $(\Box_g-m^2)u=f$, $m>0$ small.

\begin{rmk}
\label{RmkKdSResGeneral}
  Our arguments go through for general non-degenerate Kerr--de Sitter spacetimes, \emph{assuming} the `resolvent' family $\wh\Box_g(\sigma)^{-1}$ admits a meromorphic continuation to the complex plane with (polynomially lossy) high energy estimates in a strip below the real line, and the only resonance (quasi-normal mode) in $\Im\sigma\geq 0$ is a simple resonance at $0$ (`mode stability'). Apart from the mode stability, these conditions hold for a large range of spacetime parameters \cite{WunschZworskiNormHypResolvent,DyatlovQNM,VasyMicroKerrdS}, while the mode stability has only been proved for small $a$. (For the Kerr family of black holes, mode stability is known, see \cite{WhitingKerrModeStability,ShlapentokhRothmanModeStability}.)
  
  Without the mode stability assumption, we still obtain a resonance expansion for linear waves up to the Cauchy horizon, but boundedness does not follow due to the potential existence of resonances in $\Im\sigma>0$ or higher order resonances on the real line; if such resonances should indeed exist, then boundedness would in fact be false for generic forcing terms or initial data. If on the other hand one \emph{assumes} that the wave $u$ decays to a constant at some exponential rate $\alpha>0$ in the black hole exterior region, the conclusion of Theorem~\ref{ThmKdSPartialAsympConormal} still holds.
\end{rmk}

\appendix

\section{Variable order b-Sobolev spaces}
\label{SecVariable}

The analysis in \S\S\ref{SecRNdS} and \ref{SecKdS} relies on the propagation of singularities in b-Sobolev spaces of variable order; in fact, we only use microlocal elliptic regularity and real principal type propagation on such spaces. We recall some aspects of \cite[Appendix~A]{BaskinVasyWunschRadMink} needed in the sequel, and refer the reader to \cite{BaskinVasyWunschRadMink} for the proofs of elliptic regularity and real principal type propagation in this setting; since all arguments presented there are purely symbolic, they go through in the b-setting with purely notational changes. Moreover, we remark that adding (constant!) weights to the variable order b-spaces does not affect any of the arguments.

We use Sobolev orders which vary only in the base, not in the fiber. (Adding dependence on fiber variables would require purely notational changes.) In order to introduce the relevant notation, we consider the model case $\Rnhalfc$ of a manifold with boundary, and an order function $\sfs=\sfs(z)\in\CI((\Rnhalfc)_z)$, constant outside a compact set; recalling the symbol class
\[
  a(z,\zeta) \in S^0_{\rho,\delta}((\Rnhalfc)_z;\R^n_\zeta) \quad :\Longleftrightarrow\quad |\pa_z^\alpha\pa_\zeta^\beta a(z,\zeta)|\leq C_{\alpha\beta}\la\zeta\ra^{-\rho|\beta|+\delta|\alpha|}\quad \forall \alpha,\beta,
\]
we then define
\begin{equation}
\label{EqVariableBSymbols}
  S^\sfs_{\rho,\delta}=\{\la\zeta\ra^{\sfs(z)}a(z,\zeta)\colon a\in S^0_{\rho,\delta}\}.
\end{equation}
Now $S^\sfs_{\rho,\delta}\subset S^{s_0}_{\rho,\delta}$ for $s_0=\sup\sfs$, provided $\delta>0,\rho<1$, due to derivatives falling on $\la\zeta\ra^{\sfs(z)}$, producing logarithmic terms. Therefore, we can quantize symbols in $S^\sfs_{\rho,\delta}$; we denote the class of quantizations of such symbols by $\Psi^\sfs_{\bop,\rho,\delta}(\Rnhalfc)\subset\Psi^{s_0}_{\bop,\rho,\delta}(\Rnhalfc)$. We will only work with $\rho=1-\delta$, $\delta\in(0,1/2)$, in which case one can in particular transfer this space of operators to a manifold with boundary and obtain a b-pseudodifferential calculus; see \cite{HormanderFIO1} for the analogous case of manifolds without boundary. Thus, if $A\in\Psi_{\bop,1-\delta,\delta}^\sfs$ and $B\in\Psi_{\bop,1-\delta,\delta}^{\sfs'}$ for two order functions $\sfs,\sfs'$, then
\[
  A\circ B\in\Psi_{\bop,1-\delta,\delta}^{\sfs+\sfs'},\quad \sigma(A\circ B)=\sigma(A)\sigma(B)\in S_{1-\delta,\delta}^{\sfs+\sfs'},
\]
where $\sigma$ denotes the principal symbol in the respective classes of operators; the principal symbol of an element in $\Psi_{\bop,1-\delta,\delta}^\sfs$ is well-defined in $S_{1-\delta,\delta}^\sfs/S_{1-\delta,\delta}^{\sfs-(1-2\delta)}$. Furthermore, we have
\[
  i[A,B]\in\Psi_{\bop,1-\delta,\delta}^{\sfs+\sfs'-(1-2\delta)}, \quad \sigma(i[A,B])=H_{\sigma(A)}\sigma(B).
\]

For the purposes of the analysis in \S\ref{SubsecRNdSFredholm}, we need to describe the relation of variable order b-Sobolev spaces to semiclassical function spaces via the Mellin transform. We work locally in $\Rnhalfc=[0,\infty)_\tau\times\R^{n-1}_x$, and the variable order function is $\sfs=\sfs(x)\in\CI(\R^{n-1})$, with $\sfs$ constant outside a compact set. Fixing a real number $N<\inf\sfs(x)$ and an elliptic, dilation-invariant operator $A\in\Psi^\sfs_{\bop,1-\delta,\delta}(\Rnhalfc)$, $\delta\in(0,1/2)$, the norm on $\Hb^\sfs(\Rnhalfc)$ is given by
\begin{equation}
\label{EqVariableBNorm}
  \|u\|_{\Hb^\sfs(\Rnhalfc)}^2 = \|u\|_{\Hb^N}^2 + \|Au\|_{\Hb^0}^2,
\end{equation}
and all choices of $N$ and $A$ give equivalent norms. (This follows from elliptic regularity.) Since the $\Hb^N$-part of the norm is irrelevant in a certain sense (it is only there to take care of a possible kernel of $A$), we focus on the seminorm
\[
  |u|_{\Hb^\sfs} := \|Au\|_{\Hb^0};
\]
we concretely take $A$ to be the left quantization of $\la\zeta\ra^{s(x)}$, writing b-1-forms as
\[
  \zeta = \sigma\,\frac{d\tau}{\tau} + \xi\,dx \in \Tb^*\Rnhalfc.
\]
Denote the Mellin transform of $u$ in $\tau$ by $\wh u(\sigma,\cdot)$, and the Fourier transform of $\wh u$ in $x$ by $\wt u(\sigma,\xi)$; thus, $\wt u$ is the Fourier transform of $u$ in $(-t,x)$, where $t=-\log\tau$. Then by Plancherel,
\begin{align*}
  |u|_{\Hb^\sfs}^2 &= \iint \left|\iint \tau^{i\sigma}e^{ix\xi}\la(\sigma,\xi)\ra^{\sfs(x)}\wt u(\sigma,\xi)\,d\sigma\,d\xi\right|^2\,\frac{d\tau}{\tau}\,dx \\
    &= \iint \left|\int e^{ix\xi}\la(\sigma,\xi)\ra^{\sfs(x)}\wt u(\sigma,\xi)\,d\xi\right|^2\,d\sigma\,dx,
\end{align*}
where $\la(\sigma,\xi)\ra=(1+|\sigma|^2+|\xi|^2)^{1/2}$. Using $\la(\sigma,\xi)\ra=\la\sigma\ra\la\frac{\xi}{\la\sigma\ra}\ra$, we can rewrite this integral as
\begin{equation}
\label{EqVariableBtoSclAlmostThere}
  |u|_{\Hb^\sfs}^2 = \iint \la\sigma\ra^{2\sfs(x)} \bigl| \big\la\la\sigma\ra^{-1}D\big\ra^{\sfs(x)}\wh u(\sigma,x)\bigr|^2\,dx\,d\sigma.
\end{equation}
This suggests:
\begin{definition}
\label{DefVariableScl}
  For $\sfs(x),\sfw(x)\in\CI(\R^{n-1})$, constant outside a compact set, define the semiclassical Sobolev space $H_h^{\sfs,\sfw}(\R^{n-1})$, $h>0$, by the norm
  \[
    \|v\|_{H_h^{\sfs,\sfw}(\R^{n-1})}^2 = h^N\|v\|_{H_h^{-N}}^2 + \int |h^{\sfw(x)}\la hD\ra^{\sfs(x)}v(x)|^2\,dx,
  \]
  where $N>\max(-\inf\sfs,\sup\sfw)$ is a real number.
\end{definition}
The particular choice of the value of $N$ is irrelevant, see Remark~\ref{RmkVariableSclNorm}, where we also give a better, invariant, version of Definition~\ref{DefVariableScl}. Thus, $H_h^{\sfs,\sfw}(\R^{n-1})=H^\sfs(\R^{n-1})$ as a space, but the semiclassical space captures the behavior of the norm as $h\to 0+$. We remark that the space $H_h^{\sfs,\sfw}$ becomes weaker as one increases $\sfw$ or decreases $\sfs$.
\begin{rmk}
\label{RmkVariableSclConstOrders}
  If $\sfs\equiv s$ and $\sfw\equiv w$ are constants, we can use the equivalent norm $\|v\|_{H_h^{s,w}(\R^{n-1})}=h^w\|\la hD\ra^s v\|_{L^2}$.
\end{rmk}
Using \eqref{EqVariableBtoSclAlmostThere} and taking the $\Hb^{-N}$-term in \eqref{EqVariableBNorm} into account, we thus have an equivalence of norms
\begin{equation}
\label{EqVariableBToScl}
  \|u\|_{\Hb^\sfs(\R^n)} \sim \|\wh u(\sigma,x)\|_{L^2\bigl(\R_\sigma;H_{\la\sigma\ra^{-1}}^{\sfs,-\sfs}(\R^{n-1}_x)\bigr)}.
\end{equation}

The semiclassical analogues of the symbol spaces \eqref{EqVariableBSymbols}, which are adapted to working with the spaces $H_h^{\sfs,\sfw}$, are defined by
\begin{equation}
\label{EqVariableSclSymbols}
\begin{split}
  a(h;x,\xi) &\in S_{h,1-\delta,\delta}^{\sfs,\sfw} \\
  &:\Longleftrightarrow \quad |\pa_x^\alpha\pa_\xi^\beta a(h;x,\xi)|\leq C_{\alpha\beta} h^{\sfw(x)-\delta(|\alpha|+|\beta|)} \la\xi\ra^{\sfs(x)-|\beta| + \delta(|\alpha|+|\beta|)},
\end{split}
\end{equation}
with $C_{\alpha\beta}$ independent of $h$. In our application, differentiation in $x$ or $\xi$ will in fact at most produce a \emph{logarithmic} loss, i.e.\ will produce a factor of $\log h$ or $\log\la\xi\ra$. For us, the main example of an element in $S_{h,1-\delta,\delta}^{\sfs,\sfw}$ is the symbol $h^{\sfw(x)}\la\xi\ra^{\sfs(x)}$.

Quantizations of symbols in $S_{h,1-\delta,\delta}^{\sfs,\sfw}$ are denoted $\Psi_{h,1-\delta,\delta}^{\sfs,\sfw}$, and for $A\in\Psi_{h,1-\delta,\delta}^{\sfs,\sfw}$ and $B\in\Psi_{h,1-\delta,\delta}^{\sfs',\sfw'}$, we have
\[
  A\circ B \in \Psi_{h,1-\delta,\delta}^{\sfs+\sfs',\sfw+\sfw'}
\]
and
\[
  \frac{i}{h}[A,B] \in \Psi_{h,1-\delta,\delta}^{\sfs+\sfs'-(1-2\delta),\sfw+\sfw'+2\delta},
\]
with principal symbols given by the product, resp.\ the Poisson bracket, of the respective symbols. Here, the principal symbol of an element of $\Psi_{h,1-\delta,\delta}^{\sfs,\sfw}$ is well-defined in $S_{h,1-\delta,\delta}^{\sfs,\sfw}/S_{h,1-\delta,\delta}^{\sfs-1+2\delta,\sfw-1+2\delta}$.

\begin{rmk}
\label{RmkVariableSclNorm}
  Using elliptic regularity in the calculus $\Psi_{h,1-\delta,\delta}^{*,*}$, we see that given $u\in h^N H_h^{-N}$, $N\in\R$, we have $u\in H_h^{\sfs,\sfw}$ if and only if $Au\in L^2=H_h^{0,0}$, where $A\in\Psi_{h,1-\delta,\delta}^{\sfs,\sfw}$ is a fixed elliptic operator; i.e.\ we have an equivalence of norms
  \[
    \|u\|_{H_h^{\sfs,\sfw}} \sim \|u\|_{H_h^{-N,N}} + \|Au\|_{H_h^{0,0}}.
  \]
\end{rmk}

We next discuss microlocal regularity results for variable order operators; general references for such results in the constant order (semiclassical) setting are \cite{MelroseIML,WunschMicronotes,ZworskiSemiclassical} and \cite{HormanderAnalysisPDE3}. Working on a compact manifold $X$ without boundary now, $\dim X=n$, suppose we are given a semiclassical \psdo{}\ $P\in\Psih^m(X)$. Semiclassical elliptic regularity takes the following quantitative form on variable order spaces:
\begin{prop}
\label{PropVariableSclElliptic}
  If $B,G\in\Psih^0(X)$ are such that $\WFh(B)\subset\Ellh(P)$ (the semiclassical elliptic set of $P$), and $G$ is elliptic on $\WFh(B)$, then
  \[
    \|B v\|_{H_h^{\sfs,\sfw}} \leq C(\|GP v\|_{H_h^{\sfs-m,\sfw}} + \|v\|_{H_h^{-N,N}})
  \]
  for any fixed $N$.
\end{prop}
\begin{proof}
  This follows from the usual symbolic construction of a microlocal inverse of $P$ near $\WFh(B)$.
\end{proof}

The semiclassical real principal type propagation of singularities requires a Hamilton derivative condition on the orders $\sfs,\sfw$ of the function space: let $P\in\Psih^m(X)$ with real-valued semiclassical principal symbol $p=\sigma_\semi(P)\in S^m_\cl(T^*X)$, i.e.\ $p$ is a classical symbol, which we assume for simplicity to be $h$-independent. Let $\rham_p=\wh\rho^{m-1}\ham_p$ be the rescaled Hamilton vector field, with $\wh\rho\in\CI(T^*X)$ is homogeneous of degree $-1$ in the fibers of $T^*X$ away from the zero section; thus $\rham_p$ is homogeneous of degree $0$ modulo vector fields vanishing at fiber infinity, and can thus be viewed as a smooth vector field on the radially compactified cotangent bundle $\rcT^*X$. At fiber infinity $S^*X\subset\rcT^*X$, the $\rham_p$ flow is simply the rescaled Hamilton flow of the homogeneous principal part of $p$, while at finite points $T^*X\subset\rcT^*X$, $\rham_p$ is proportional to the semiclassical Hamilton vector field.
\begin{prop}
\label{PropVariableSclPropagation}
  Under these assumptions, let $\sfs,\sfw\in\CI(X)$ be order functions, and let $U\subset\rcT^*X$ be open; suppose $\rham_p\sfs\leq 0$ and $\rham_p\sfw\geq 0$ in $U$. Suppose $B_1,B_2,G\in\Psih^0(X)$ are such that $G$ is elliptic on $\WFh(B_1)$, and all backward null-bicharacteristics of $P$ from $\WFh(B_1)\cap\Sigma_\semi(P)$ enter $\Ellh(B_2)$ while remaining in $\Ellh(G)\cap U$. Then
  \[
    \|B_1 v\|_{H_h^{\sfs,\sfw}} \leq C(\|GP v\|_{H_h^{\sfs-m+1,\sfw-1}} + \|B_2 v\|_{H_h^{\sfs,\sfw}} + \|v\|_{H_h^{-N,N}})
  \]
  for any fixed $N$.
\end{prop}

For $\sfw=0$, this gives the usual estimate of $v$ in $H_h^\sfs$ in terms of $h^{-1}Pv$ in $H_h^{\sfs-m+1}$, losing $1$ derivative and $1$ power of $h$ relative to the elliptic setting.

\begin{proof}[Proof of Proposition~\ref{PropVariableSclPropagation}.]
  The proof is almost the same as that of \cite[Proposition~A.1]{BaskinVasyWunschRadMink}, so we shall be brief. Since the result states nothing about critical points of the Hamilton flow, we may assume $\rham_p\neq 0$ on $U$ (at $S^*X$, this means that $\rham_p$ is \emph{non-radial}). Let $\alpha\in\WFh(B_1)\cap\Sigma_\semi(P)$. Let us first prove the propagation at fiber infinity: introduce coordinates $q_1,q'$ on $S^*X$, $q'=(q_2,\ldots,q_{2n-1})$, centered at $\alpha$, such that $\rham_p=\pa_{q_1}$, and suppose $t_2<t_1<0<t_0$ and the neighborhood $U'$ of $0\in\R^{2n-2}_{q'}$ are such that $[t_2,t_0]_{q_1}\times\ol{U'}\subset U$; suppose we have a priori $H_h^{\sfs,\sfw}$-regularity in $[t_2,t_1]_{q_1}\times\ol{U'}$, i.e.\ $B_2$ is elliptic there. We use a commutant (omitting the necessary regularization in the weight $\wh\rho$ for brevity)
  \[
    a = h^{-\sfw}\wh\rho^{-\sfs+(m-1)/2}\chi(q_1)\phi(q'),
  \]
  where $\chi=\chi_0\chi_1$, $\chi_0(t)=e^{-\digamma(t-t_0)}$ for $t<t_0$, $\chi_0(t)=0$ for $t\geq t_0$, with $\digamma>0$ large, and $\chi_1(t)\equiv 1$ near $[t_0,\infty)$, $\chi_1(t)\equiv 0$ near $(-\infty,t_2]$; moreover $\phi\in\CIc(U')$, $\phi(0)=1$. We then compute
  \begin{align*}
    \rham_p a = h^{-\sfw}\wh\rho^{-\sfs+(m-1)/2}&\bigl(\chi_0'(q_1)\chi_1(q_1) + \chi_0(q_1)\chi_1'(q_1) \\
      &+ (-\sfs+(m-1)/2)\chi_0(q_1)\chi_1(q_1)\wh\rho^{-1}\rham_p\wh\rho \\
      &- (\rham_p\sfs)\log\wh\rho + (\rham_p\sfw)\log(h^{-1})\bigr).
  \end{align*}
  Now $\chi_0'\leq 0$, giving rise to the main `good' term, while the $\chi_1'$ term (which has the opposite sign) is supported where one has a priori regularity. The term on the second line can be absorbed into the first by making $\digamma$ large (since $\chi_0$ can then be dominated by a small multiple of $\chi_0'$), while the last two terms have the same sign as the main term by our assumptions on $\sfs$ and $\sfw$. A positive commutator computation, a standard regularization argument, and absorbing the contribution of the imaginary part of $P$ by making $\digamma$ larger if necessary, gives the desired result.

  For the propagation within $T^*X$, a similar argument applies; we use local coordinates $q_1,q'$ in $T^*X$ with $q'=(q_2,\ldots,q_{2n})$ now, centered at $\alpha$, so that $\rham_p=\pa_{q_1}$; and the differentiability order $\sfs$ becomes irrelevant now, as we are away from fiber infinity. Thus, we can use the commutant $a=h^{-\sfw}\chi(q_1)\phi(q')$, with $\chi$ exactly as above, and $\phi$ localizing near $0$; the positive commutator argument then proceeds as usual.
\end{proof}

Returning to \eqref{EqVariableBToScl}, we observe that for $u\in\Hb^\sfs$, we can apply this proposition to $\wh u(\sigma,x)\in H_{\la\sigma\ra^{-1}}^{\sfs,-\sfs}$ under the single condition $\rham_p\sfs\leq 0$, which is the same condition as for real principal type propagation for $u$ in b-Sobolev spaces (as it should).

Finally, we point out that completely analogous results hold for \emph{weighted} b-Sobolev spaces $\Hb^{\sfs,\alpha}$ and their semiclassical analogues: the only necessary modification is that now we have to restrict the Mellin-dual variable to $\tau$, called $\sigma$ here, to $\Im\sigma=-\alpha$, since the Mellin transform in $\tau$ induces an isometric isomorphism
\[
  \|u(\tau,x)\|_{\Hb^{0,\alpha}} = \|\wh u(\sigma,x)\|_{L^2(\{\Im\sigma=-\alpha\};L^2_x)}.
\]

\section{Supported and extendible function spaces on manifolds with corners}
\label{SecSuppExt}

We briefly recall supported and extendible distributions on manifolds with boundary, following \cite[Appendix~B]{HormanderAnalysisPDE3}. The model case is $\R^n=\R_{x_1}\times\R^{n-1}_{x'}$, and we consider Sobolev spaces with regularity $s\in\R$. For notational brevity, we omit the factor $\R^{n-1}_{x'}$. Thus, we let
\begin{align*}
  H^s([0,\infty))^\bullet &:= \{ u\in H^s(\R)\colon\supp u\subset[0,\infty) \}, \\
  H^s((0,\infty))^- &:= \{ u|_{(0,\infty)} \colon u\in H^s(\R) \},
\end{align*}
called $H^s$ space with supported $(\bullet)$, resp.\ extendible $(-)$, character at the boundary $\{x_1=0\}$. The Hilbert norm on the supported space is defined by restriction from $H^s$, while the Hilbert norm on the extendible space comes from the isomorphism
\[
  H^s((0,\infty))^- \cong H^s(\R) / H^s((-\infty,0])^\bullet;
\]
since the supported space on the right hand side is a closed subspace of $H^s(\R)$, we immediately get an isometric extension map
\[
  E\colon H^s((0,\infty))^- \to H^s(\R),
\]
which identifies $H^s((0,\infty))^-$ with the orthogonal complement of $H^s((-\infty,0])^\bullet$ in $H^s(\R)$; thus,
\[
  \|u\|_{H^s((0,\infty))^-} \equiv \inf_{\genfrac{}{}{0pt}{}{U\in H^s(\R)}{U|_{(0,\infty)}=u}} \|U\|_{H^s(\R)} = \|Eu\|_{H^s(\R)}.
\]
The dual spaces relative to $L^2$ are given by
\[
  \bigl(H^s([0,\infty))^\bullet\bigr)^* = H^{-s}((0,\infty))^-.
\]

We now discuss the case of codimension $2$ corners, which is all we need for our application; treating the case of higher codimension corners requires purely notational changes. We work locally on $\R^n_x$, $n\geq 2$, $x=(x_1,x_2,x')$. Consider the domain
\[
  [0,\infty)_{x_1}\times[0,\infty)_{x_2}\times\R^{n-2}_{x'}
\]
which is a submanifold of with corners of $\R^n$. Again, since the $x'$ variables will carry through our arguments below, we simplify notation by dropping them, i.e.\ by letting $n=2$. 

Let $s\in\R$. There are two natural ways to define a space $H^s((0,\infty)\times[0,\infty))^{-,\bullet}$ of distributions in $H^s$ with extendible character at $\{x_1=0\}$ and supported character at $\{x_2=0\}$, which give rise to two a priori different norms and dual spaces. Namely,
\begin{equation}
\label{EqSuppExtTwoChoices}
\begin{split}
  H^s((0,\infty)\times[0,\infty))^{-,\bullet}_1 &= H^s(\R\times[0,\infty))^\bullet / H^s((-\infty,0]\times[0,\infty))^{\bullet,\bullet}, \\
  H^s((0,\infty)\times[0,\infty))^{-,\bullet}_2 &= \{ u\in H^s((0,\infty)\times\R)^- \colon \supp u\subset(0,\infty)\times[0,\infty)\};
\end{split}
\end{equation}
we equip the first space with the quotient topology, and the second space with the subspace topology. The first space is the space of restrictions to $(0,\infty)\times[0,\infty)$ of distributions with support in $x_2\geq 0$, while the second space is the space of extendible distributions in the half space $(0,\infty)\times\R$ which have support in $(0,\infty)\times[0,\infty)$; see Figure~\ref{FigSuppExtDef}.

\begin{figure}[!ht]
  \centering
  \includegraphics{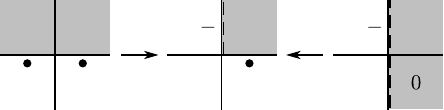}
  \caption{The two choices for defining $H^s((0,\infty)\times[0,\infty))^{-,\bullet}$ (middle). Left: choice $1$. Right: choice $2$. The supports of elements of the spaces that $H^s_1$, resp.\ $H^s_2$, are quotients, resp.\ subspaces of are shaded; the `$0$' indicates the vanishing condition in the definition of $H^s_2$.}
\label{FigSuppExtDef}
\end{figure}

As in the case of manifolds with boundary discussed above, both spaces come equipped with isometric (by the definition of their norms) extension operators
\[
  E_j\colon H^s((0,\infty)\times[0,\infty))^{-,\bullet}_j \to H^s(\R^2),\quad j=1,2,
\]
with $\ran E_1\subset H^s(\R\times[0,\infty))^\bullet\subset H^s(\R^2)$ and $\ran E_2$ contained in the space
\[
  H^s(\R^2\setminus(0,\infty)\times(-\infty,0))^{\bullet,\bullet}
\]
of distributions in $H^s(\R^2)$ with support in $\R^2\setminus(0,\infty)\times(-\infty,0)$, see Figure~\ref{FigSuppExtMaps}. We can thus also describe the second variant of $(H^s)^{-,\bullet}$ equivalently as the quotient
\begin{equation}
\label{EqSuppExtMap2}
  H^s((0,\infty)\times[0,\infty))^{-,\bullet}_2 = H^s(\R^2\setminus(0,\infty)\times(-\infty,0))^{\bullet,\bullet} / H^s((-\infty,0]\times\R)^\bullet.
\end{equation}

\begin{figure}[!ht]
  \centering
  \includegraphics{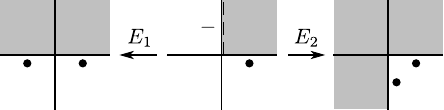}
  \caption{Ranges of the extension maps for the two possible definitions of $H^s((0,\infty)\times[0,\infty))^{-,\bullet}$. The supports of elements in the ranges of the extension maps are shaded.}
\label{FigSuppExtMaps}
\end{figure}

Furthermore, the dual spaces are isometric to
\begin{align*}
  \bigl(H^s((0,\infty)&\times[0,\infty))^{-,\bullet}_1\bigr)^* = H^{-s}([0,\infty)\times(0,\infty))^{\bullet,-}_2 \\
    & := \{ u\in H^s(\R\times(0,\infty))^-\colon \supp u\subset[0,\infty)\times(0,\infty)\}, \\
  \bigl(H^s((0,\infty)&\times[0,\infty))^{-,\bullet}_2\bigr)^* = H^{-s}([0,\infty)\times(0,\infty))^{\bullet,-}_1 \\
    & := H^{-s}([0,\infty)\times\R)^\bullet / H^{-s}([0,\infty)\times(-\infty,0])^{\bullet,\bullet},
\end{align*}
i.e.\ dualizing switches choices $1$ and $2$ for the definition of the mixed supported and extendible spaces.

We observe that the composition
\[
  H^s(\R\times[0,\infty))^\bullet\hookrightarrow H^s(\R^2)\to H^s((0,\infty)\times\R)^-
\]
induces a continuous inclusion
\begin{equation}
\label{EqSuppExtIncl}
  i_{12} \colon H^s((0,\infty)\times[0,\infty))^{-,\bullet}_1 \hookrightarrow H^s((0,\infty)\times[0,\infty))^{-,\bullet}_2.
\end{equation}

The main result of this section is:
\begin{prop}
\label{PropSuppExtSame}
  The inclusion map $i_{12}$ in \eqref{EqSuppExtIncl} is an isomorphism.
\end{prop}

The important consequence is that we can now make the following
\begin{definition}
\label{DefSuppExt}
  For $s\in\R$, define $H^s((0,\infty)\times[0,\infty))^{-,\bullet} := H^s((0,\infty)\times[0,\infty))^{-,\bullet}_1$, with the latter space defined in \eqref{EqSuppExtTwoChoices}.
\end{definition}
The point is that we can now also identify the dual space with
\[
  H^{-s}([0,\infty)\times(0,\infty))^{\bullet,-},
\]
and we can work  with either definition in \eqref{EqSuppExtMap2}.

\begin{proof}[Proof of Proposition~\ref{PropSuppExtSame}.]
  Since $i_{12}$ is an isomorphism if and only if the dual map
  \[
    H^{-s}([0,\infty)\times(0,\infty))^{\bullet,-}_1 \hookrightarrow H^{-s}([0,\infty)\times(0,\infty))^{\bullet,-}_2
  \]
  is an isomorphism, it suffices to consider the case $s\geq 0$. In view of the characterizations \eqref{EqSuppExtTwoChoices} and \eqref{EqSuppExtMap2} of the two versions of $(H^s)^{-,\bullet}$ as quotients (equipped with the quotient norm!), it suffices to prove the existence of a bounded linear map
  \begin{equation}
  \label{EqSuppExtPhi}
    \Phi \colon H^s(\R^2\setminus(0,\infty)\times(-\infty,0))^{\bullet,\bullet} \to H^s(\R\times[0,\infty))^\bullet
  \end{equation}
  with $\Phi(u)|_{(0,\infty)\times\R}=u|_{(0,\infty)\times\R}$. The idea is to use the fact that for integer $k\geq 0$, $H^k$-spaces of extendible distributions are \emph{intrinsically} defined: thus, for $u\in H^k(\R^2\setminus(0,\infty)\times(-\infty,0))^{\bullet,\bullet}$, the restriction $u|_{\R\times(-\infty,0)}$ to the lower half plane is an element of $H^k(\R\times(-\infty,0))$ with support in $(-\infty,0]\times(-\infty,0)$; but then we can use an extension map
  \[
    \Psi_k\colon H^k(\R_{x_1}\times(-\infty,0)_{x_2}) \to H^k(\R^2_{x_1,x_2}),
  \]
  defined using reflections and rescalings (see \cite[\S4.4]{TaylorPDE}), which in addition preserves the property of being supported in $x_1\leq 0$. We can then define the map $\Phi$ on $H^s$ by
  \[
    \Phi(u) := u - \Psi_k(u|_{\R\times(-\infty,0)})
  \]
  for all integer $0\leq s\leq k$; by interpolation, the same map in fact works for all $s\in[0,k]$. Since $k$ can be chosen arbitrarily, this proves the existence of a map \eqref{EqSuppExtPhi} for any fixed real $s$.
\end{proof}

\subsection*{Acknowledgments}

We are very grateful to Jonathan Luk and Maciej Zworski for many helpful discussions. We would also like to thank Sung-Jin Oh for many helpful discussions and suggestions, for reading parts of the manuscript, and for pointing out a result in \cite{SbierskiThesis} which led to the discussion in Remark~\ref{RmkRNdSHighReg}; thanks also to Elmar Schrohe for very useful discussions leading to Appendix~\ref{SecSuppExt}. We are grateful for the hospitality of the Erwin Schr\"odinger Institute in Vienna, where part of this work was carried out. Finally, we are very grateful to the anonymous referee for many suggestions which improved the exposition and clarity of the paper.

We gratefully acknowledge support by A.V.'s National Science Foundation grants DMS-1068742 and DMS-1361432. P.H.\ is a Miller Fellow and thanks the Adolph C.\ and Mary Sprague Miller Institute for Basic Research in Science, University of California Berkeley, for support.


\end{document}